\documentclass[a4paper,11pt]{article}
\usepackage[utf8]{inputenc}

\usepackage{boxedminipage}
\usepackage{amsfonts}
\usepackage{amsmath} 
\usepackage{amssymb}
\usepackage{graphicx}
\usepackage{amsthm}
\usepackage{t1enc}
\usepackage{subfig}
\usepackage{authblk} 
 \usepackage{mathabx}
 \usepackage{url}

\newtheorem{theorem}{Theorem}[section]
\newtheorem{lemma}[theorem]{Lemma}
\newtheorem{proposition}[theorem]{Proposition}
\newtheorem{corollary}[theorem]{Corollary}
\newtheorem{definition}[theorem]{Definition}

\newtheorem{fact}[theorem]{Fact}
\newtheorem{question}{Question}

\newcommand{\forceP}{\mathbb{P}}
\newcommand{\forceQ}{\mathbb{Q}}

\newcommand{\ZFC}{\mathsf{ZFC}}

\newcommand{\ZFP}{\mathsf{ZF}^-}

\newcommand{\MRP}{\mathsf{MRP}}

\newcommand{\MA}{\mathsf{MA_{\omega_1}}}

\newcommand{\PD}{\mathsf{PD}}

\def\undertilde#1{\mathord{\vtop{\ialign{##\crcr
$\hfil\displaystyle{#1}\hfil$\crcr\noalign{\kern1.5pt\nointerlineskip}
$\hfil\tilde{}\hfil$\crcr\noalign{\kern1.5pt}}}}}

\title{On a local variant of the 12th Delfino problem---the $\Sigma$-side} 
\author{ Stefan Hoffelner$^{1}$ and Sandra M\"uller$^{1}$ \footnote{The first author's research was funded in whole by the Austrian Science Fund (FWF) Grant-DOI 10.55776/P37228. The second author's research was funded in part by the Austrian Science Fund (FWF) [10.55776/Y1498, 10.55776/I6087]. For the purpose of open access, the authors have applied a CC BY public copyright license to any Author Accepted Manuscript version arising from this submission. }  }
\date{
    $^1$TU Wien \\
    \today
}

\begin{document}

\maketitle
\begin{abstract}
Assuming that $M_n$, the canonical inner model with $n$ Woodin cardinals,
exists, we force a model in which every $\boldsymbol{\Sigma}^1_{n+2}$ set is
Lebesgue measurable and has the Baire property, and in which
$\Sigma^1_{n+2+m}$-uniformization holds for every $m\in\omega$. Additionally,
this universe has a $\Delta^1_{n+3}$-definable wellorder of the reals. This
answers a question of S. D. Friedman and R. Schindler from
1999. In the case $n=1$, the construction also gives a model with one Woodin
cardinal in which all $\Sigma^1_3$ sets are measurable with respect to the
random, Cohen, Sacks and Miller notions of measurability, while a
$\Delta^1_4$-definable wellorder of the reals exists answering an instance of a question of S. D. Friedman and D. Schrittesser.
\end{abstract}
\section{Introduction}

The structural theory of the projective hierarchy exhibits a profound dichotomy between regularity properties---typically induced by large cardinals and determinacy hypotheses---and the existence of definable wellorders, which are traditionally derived from the fine structure of inner models. Under full Projective Determinacy ($\mathsf{PD}$), the projective sets exhibit complete regularity, precluding the existence of any projective wellorder of the reals. Conversely, in canonical inner models, definability and regularity are sharply demarcated. 

The primary motivation for this article is rooted in a version of Woodin's 12th Delfino problem, which in its original formulation asks whether $\mathsf{PD}$ is strictly equivalent to all its structural consequences for the projective hierarchy. 

A localized formulation of this question, first posed by S. Friedman and R. Schindler (see \cite{friedman2003universally} and \cite{schindler_delfino}) goes as follows: if we observe a behavior of the projective sets of reals which is the same as the one in some $M_{n+1}$, the canonical inner model with $n+1$ Woodin cardinals, is it true that the local hypothesis of projective determinacy, namely $\boldsymbol{\Pi}^1_{n+1}$-determinacy or even the weaker (if $n+1$ is odd) $\boldsymbol{\Delta}^1_{n+1}$-determinacy, must be true? Specifically:

\begin{question}
    Assume $n \ge 3$. Suppose that every $\boldsymbol{\Sigma}^1_{n+2}$ set is Lebesgue measurable and has the Baire property, and there is a $\Delta^1_{n+3}$-definable wellorder of the reals. Does $\boldsymbol{\Delta}^1_{n+1}$-determinacy hold?
\end{question}

For $n=1$ and $n=2$, S. Friedman used his substantial refinements of Jensen's coding the universe technique to answer it in the negative (see theorem 8.51 \cite{friedman2011fine}). The methods stall however after the fourth level due to considerable technical difficulties one inevitably faces when trying to apply coding arguments over more complicated inner models. In \cite{friedman2003universally} the authors were able to show that assuming $n$ many strong cardinals, there are universes where all $\boldsymbol{\Sigma}^1_{n+3}$-sets are Lebesgue measurable, have the Baire property and which allow a $\Delta^1_{n+5}$-definable wellorder of the reals. Their arguments use Jensen coding combined with collapsing strong cardinals which implies that the associated levels of the projective hierarchy become universally Baire due to a result of H. Woodin.

The main result of this work demonstrates that the answer to this localized Delfino question is unequivocally no, even if we additionally require the $\Sigma^1_{n+2+m}$-uniformization property for every $m\in \omega$. 

\begin{theorem}
    Assume that $M_n$, the canonical inner model with $n$ Woodin cardinals, exists. Then there is a generic extension of $M_n$ in which Martin's Axiom ($\MA$) holds, every $\boldsymbol{\Sigma}^1_{n+2}$ set is Lebesgue measurable and has the Baire property, $\Sigma^1_{n+2+m}$-uniformization holds for all $m \in \omega$, and the universe admits a $\Delta^1_{n+3}$-definable wellorder of the reals.
\end{theorem}

This theorem should be contextualized against the properties of the projective sets of reals in $M_{n+1}$. In $M_{n+1}$, the presence of $n+1$ Woodin cardinals yields $\boldsymbol{\Pi}^1_{n+1}$-determinacy by the Martin-Steel theorem (\cite{martin1989proof}), guaranteeing that all sets in $\boldsymbol{\Sigma}^1_{n+2}$ are Lebesgue measurable and possess the Baire property via unfolding games. Simultaneously, the fine structure of $M_{n+1}$ provides a good $\Sigma^1_{n+3}$-definable wellorder of the reals (see \cite{Steel2}), which in turn implies $\Sigma^1_{n+3+m}$-uniformization for every $m \in \omega$ by an old result of J. Addison (see \cite{Addison}). Thus, the universe we force over $M_n$ mimics the behaviour of the reals in $M_{n+1}$, but uses significantly less large cardinal strength. As in $M_n$, $\Delta^1_{n+1}$-determinacy fails \footnote{The failure of $\boldsymbol{\Delta}^1_{n+1}$-determinacy in $M_n$ can be seen via comparison. Woodin's inner model characterization establishes that $\boldsymbol{\Delta}^1_{n+1}$-determinacy is equivalent to the statement that for every real $x$, there exists a countable, iterable $x$-mouse with $n$ Woodin cardinals. If $\boldsymbol{\Delta}^1_{n+1}$-determinacy were to hold in $M_n$, then $M_n$ would contain a set sized, iterable mouse $N$ with $n$ Woodin cardinals (see \cite{muller2020mice}). Coiterating $N$ with $M_n$ yields a proper initial segment of an iterate of $M_n$ containing $n$ Woodin cardinals. However, by the minimality of $M_n$, no proper initial segment of $M_n$ (or its iterates) can contain $n$ Woodin cardinals. This contradiction demonstrates that no such mouse $N$ exists in $M_n$, and thus $\boldsymbol{\Delta}^1_{n+1}$-determinacy must fail in $M_n$.} the negative answer to the Friedman-Schindler question follows.

Our proof takes advantage of a result of G. Hjorth (\cite{Hjorth}, Corollary 2.4) who observed that in the presence of ``every real has a sharp'', $\MA$ implies that every $\boldsymbol{\Sigma}^1_3$-set is Lebesgue measurable and has the Baire property. Hjorth's result relies on the classical Martin-Solovay tree (\cite{martin1969basis}) denoted by $T_2$. Using the generalized versions $T_n$ of the Martin-Solovay tree, Hjorth's result can be pushed up to higher projective pointclasses, thus establishing a path for a proof of the theorem: working over $M_n$ as the ground model, the objective is to force $\mathsf{MA}_{\omega_1}$ while simultaneously introducing a $\Delta^1_{n+3}$-definable wellorder and forcing global tail $\Sigma$-uniformization. Then apply a generalization of Hjorth's argument to argue for Lebesgue measurability and the Baire property. To achieve the former, we adapt the coding techniques pioneered in \cite{BPFA_and_global_Sigma-uniformization}, which possess the necessary flexibility to accommodate the forcing of $\mathsf{MA}$. 

There are several technical obstacles in the suggested path of our proof. We list two of them. First there is considerable tension between forcing a projectively definable wellorder of the reals and the $\Sigma^1_{n+2+m}$-uniformization on the one hand and forcing
$\MA$ on the other. Indeed forcing for the first objective is in need of  a very tight control over certain projective predicates, whereas forcing for $\MA$ necessarily creates a substantial amount of generically added chaos. 
A second substantial technical obstacle arises because the methods in \cite{BPFA_and_global_Sigma-uniformization}  rely on $L$ as the ground model, which behaves considerably nicer when applying coding forcings than $M_n$, in particular in iterations of length $\omega_2$ which we must use. To bridge this gap, we implement and refine the forcing machinery introduced in \cite{NS_saturated_and_definable} and \cite{NS_saturated_and_definable_and_MA}. The main innovation here is to replace $M_n$ with an intermediate generic extension that exhibits greater robustness under the subsequent coding forcings required for the definable wellorder and the uniformization. Notably, this intermediate extension is itself constructed via coding forcings, meaning the entire architecture functions as a ``coding over codes'' construction, effectively circumventing the shortcomings of $M_n$.

We also record a byproduct of the $M_1$ case which connects the construction
with generalized notions of measurability for tree forcings. Fischer,
Friedman and Khomskii study, among others, the random, Cohen, Sacks, Miller,
Laver, Silver and Mathias notions of measurability, and Friedman--Schrittesser
ask whether one can preserve Woodin cardinals while arranging prescribed
combinations of projective $P$-measurability and failures of projective
$P$-measurability. The $M_1$ instance of our construction gives such a model
for $P\in\{B,C,S,M\}$, where $B$ denotes random/Lebesgue measurability,
$C$ denotes Cohen/Baire measurability, $S$ denotes Sacks/Marczewski
measurability, and $M$ denotes Miller measurability. This is not an additional
forcing argument: it follows from Hjorth's observation together with the
Brendle--L\"owe implications
\[
\Gamma(C)\Rightarrow \Gamma(M)
\quad\text{and}\quad
\Gamma(B)\Rightarrow \Gamma(S)
\]
for pointclasses $\Gamma$ closed under continuous preimages.

We
mention that this paper has a companion paper~\cite{LocalDelfinoPi}, which
treats the complementary $\Pi$-side direction.  There the same exact-placement
background is combined with adjacent $\Pi^1_{n+3}$-uniformization, using a
local coding predicate and a derivative hierarchy of allowable hybrid forcings.
Thus the two papers should be viewed as complementary: the present paper
develops the coding-over-codes construction needed for Martin's Axiom and the
global $\Sigma$-uniformization tail, while the companion paper develops a
separate local mechanism for the adjacent $\Pi$-uniformization theorem.

We end with a brief description of the organization of this article. In the first sections we introduce the tools we shall use in our proof. This includes several properties of $M_1$ and $M_n$, and three coding forcings which we will weave together for a suitable coding machinery. The main proof will be shown with $M_1$ as the ground model and we later rather extensively indicate how to implement the changes one has to make in order to lift the proof to $M_n$.

\section{Canonical inner models with Woodin cardinals and the trees $T_n$}\label{sec:canonical-inner-models}

\subsection{The canonical inner model with one Woodin cardinal}\label{subsec:m1-background}

We recall the fragment of inner model theory which will be used.  The ground model for the main construction is the canonical proper class mouse $M_1$, the minimal iterable proper class premouse with one Woodin cardinal. In this article, when we say $M_1$, or more generally $M_n$ for $n>1$, exists, we implicitly assume that $M_1^\#$ exists, the minimal active inner model with $1$ Woodin cardinal that is not $1$-small (see below), and, in this case, $M_1$ is the result of iterating the active measure of $M_1^\#$ out of the universe.  We use the notation and comparison conventions of Steel's outline of inner model theory, and the projective definability analysis of Steel's work on projectively well-ordered inner models; see \cite{Steel3,Steel2}.  In particular, every proper initial segment of $M_1$ is $1$-small and $\omega$-sound, and the order of construction gives a good $\Delta^1_3$ well-order of the reals of $M_1$.

A premouse $\mathcal M$ is \emph{$1$-small above $\eta$} if whenever $E$ is an extender on the $\mathcal M$-sequence and $\eta<\operatorname{crit}(E)$, the initial segment $\mathcal J^{\mathcal M}_{\operatorname{crit}(E)}$ has no Woodin cardinal above $\eta$.  We say that $\mathcal M$ is \emph{$1$-small} if it is $1$-small above $0$.

We also recall the weak iterability notion used by Steel at the first odd level.  Let $\mathcal T$ be an $\omega$-maximal putative iteration tree on $\mathcal M$, let $b$ be a maximal branch through $\mathcal T$, and let $\alpha$ be a countable ordinal.  The branch $b$ is \emph{$\alpha$-good} if whenever $\mathcal N$ is either $\mathcal M_b^{\mathcal T}$ itself, or the $\alpha$-th linear iterate of an initial segment $\mathcal P\trianglelefteq \mathcal M_b^{\mathcal T}$ by one extender on the $\mathcal P$-sequence and its images, then either $\mathcal N$ is well-founded or $\alpha\in\operatorname{wfp}(\mathcal N)$.  A countable premouse $\mathcal M$ is \emph{$\Pi^1_2$-iterable} if player II wins the corresponding one-round weak iteration game: player I plays a countable putative $\omega$-maximal tree together with a countable ordinal $\alpha$, and player II either accepts a last well-founded model or plays a maximal $\alpha$-good branch.  Steel proves that this iterability condition is $\Pi^1_2$ in the codes.  For $1$-small mice it is the $n=1$ instance of the general $\Pi_n$-iterability analysis; see \cite[Lemma~1.7]{Steel2}.

We now fix the local notation for limit length trees.

\begin{definition}\label{def:delta-common-part}
Let $\mathcal T$ be a $k$-maximal iteration tree of limit length on a premouse $\mathcal M$, where $k\leq\omega$.
\begin{enumerate}
   \item We set
   \[
      \delta(\mathcal T)=
      \sup\{\operatorname{lh}(E_\xi^{\mathcal T})\mid
             \xi+1<\operatorname{lh}(\mathcal T)\}.
   \]
   \item $\mathcal M(\mathcal T)$ denotes the common part of the models along $\mathcal T$ below $\delta(\mathcal T)$, i.e. the unique passive premouse $\mathcal P$ of height $\delta(\mathcal T)$ such that, for every extender $E_\xi^{\mathcal T}$ used in $\mathcal T$, $\mathcal P$ agrees with the corresponding model of the tree below $\operatorname{lh}(E_\xi^{\mathcal T})$.
\end{enumerate}
\end{definition}

We shall use the following form of Steel's branch uniqueness theorem, often called the zipper lemma.

\begin{theorem}[Zipper lemma]\label{thm:zipper-local}
Let $\mathcal T$ be a $k$-maximal iteration tree of limit length on a premouse $\mathcal M$, where $k\leq\omega$, and let $b,c$ be distinct cofinal branches through $\mathcal T$.  Put $\delta=\delta(\mathcal T)$.  Suppose that $A\subseteq\delta$ and that $A,\delta\in\operatorname{wfp}(\mathcal M_b^{\mathcal T})\cap\operatorname{wfp}(\mathcal M_c^{\mathcal T})$.  Then one of the two branch models satisfies
\[
   \exists\kappa<\delta\,
   (\kappa\text{ is }A\text{-strong up to }\delta).
\]
\end{theorem}

\begin{definition}\label{def:q-structures-local}
Let $\mathcal T$ be a $k$-maximal iteration tree of limit length on a premouse $\mathcal M$, where $k\leq\omega$, and let $b$ be a cofinal well-founded branch through $\mathcal T$.  The $\mathcal Q$-structure $\mathcal Q(b,\mathcal T)$ is the least initial segment of $\mathcal M_b^{\mathcal T}$, if it exists, which either sees that $\delta(\mathcal T)$ is not Woodin over the common part $\mathcal M(\mathcal T)$, or projects strictly below $\delta(\mathcal T)$ at an allowed finite degree.  More explicitly, $\mathcal Q(b,\mathcal T)=\mathcal J_\gamma^{\mathcal M_b^{\mathcal T}}$ for the least $\gamma$ such that either
\[
   \mathcal J_{\gamma+1}^{\mathcal M_b^{\mathcal T}}
      \models ``\delta(\mathcal T)\text{ is not Woodin}'',
\]
or, for some $m<\omega$ allowed by the degree of the tree,
\[
   \rho_{m+1}(\mathcal J_\gamma^{\mathcal M_b^{\mathcal T}})
      <\delta(\mathcal T).
\]
If no such $\gamma$ exists, then $\mathcal Q(b,\mathcal T)$ is undefined.
\end{definition}

The point of the preceding definition is that, in the $1$-small context, a $\mathcal Q$-structure determines at most one good cofinal branch.  If two distinct branches had the same relevant well-founded $\mathcal Q$-structure, the zipper lemma would produce strength below $\delta(\mathcal T)$, contradicting the initial segment which witnesses that $\delta(\mathcal T)$ is not Woodin.

In the next few statements, an \emph{ordinary premouse} means a premouse in Steel's usual premouse language, with no real parameter, predicate parameter, or base set added to the structure.  This is only a terminological convention distinguishing these premice from mice over a real or over another base; no additional fine-structural notion is being introduced.

\begin{lemma}\label{lem:local-comparison-m1}
Let $\mathcal M$ and $\mathcal N$ be countable ordinary premice.  Assume that both are $\omega$-sound and project to $\omega$, that $\mathcal M\triangleleft M_1$, and that $\mathcal N$ is $1$-small and $\Pi^1_2$-iterable.  Then the comparison of $\mathcal M$ with $\mathcal N$ is successful.  Consequently
\[
   \mathcal M\trianglelefteq\mathcal N
   \quad\text{or}\quad
   \mathcal N\trianglelefteq\mathcal M.
\]
\end{lemma}

\begin{proof}
Run the usual coiteration by least disagreement, producing trees $\mathcal T$ on $\mathcal M$ and $\mathcal U$ on $\mathcal N$.  The $\mathcal M$-side is governed by the strategy inherited from $M_1$.  The $\mathcal N$-side is governed by the winning strategy witnessing $\Pi^1_2$-iterability.

There is no new issue at successor stages.  Consider a countable limit stage and suppose first that both sides have reached the same comparison height, so that $\delta(\mathcal T)=\delta(\mathcal U)$.  Let $b$ be the branch selected on the $\mathcal M$-side.  Since $\mathcal M\triangleleft M_1$ and $M_1$ is $1$-small below its Woodin, the branch model $\mathcal M_b^{\mathcal T}$ has a $\mathcal Q$-structure witnessing the relevant failure of Woodinness at this common $\delta$.  Write this $\mathcal Q$-structure as $L_\alpha(\mathcal M(\mathcal T))$, using the standard coding of the common part.

Apply the $\Pi^1_2$-iterability of $\mathcal N$ to the tree $\mathcal U$ with ordinal parameter $\alpha$.  Player II supplies a maximal $\alpha$-good branch $c$.  Because the two sides of the comparison agree below the common $\delta$, the initial segment $L_\alpha(\mathcal M(\mathcal T))$ is also the corresponding $\mathcal Q$-structure for $c$ on the $\mathcal N$-side, provided the branch model is well-founded past that structure.  If there were a second sufficiently good cofinal branch through $\mathcal U$, then the zipper lemma, applied with a code for the common $\mathcal Q$-structure as parameter $A$, would produce strength below $\delta(\mathcal U)$.  This contradicts the witness to the failure of Woodinness coded by the $\mathcal Q$-structure.  Hence the good branch through $\mathcal U$ is unique, and it is fully well-founded.

If one side has stopped using extenders, the same argument applies with the last model on the stopped side replacing the common branch model.  Its relevant initial segment supplies the $\mathcal Q$-structure which witnesses that the remaining comparison height is not Woodin, and therefore determines the unique well-founded branch on the other side.

It remains to rule out a comparison of length $\omega_1$.  Suppose that such a putative comparison existed.  Force with the L\'evy collapse $\operatorname{Col}(\omega,\omega_1)$ over the ambient model.  In the collapse extension, the putative tree becomes countable, and $\Pi^1_2$-iterability is preserved.  Thus there is a branch supplied by the $\Pi^1_2$ strategy.  By the uniqueness just proved, this branch is ordinal definable from the ground-model data of the comparison.  Homogeneity of the collapse therefore puts the branch back in the ground model, contradicting the assumption that the comparison had no branch at stage $\omega_1$.

Thus the comparison terminates below $\omega_1$.  The terminal models are linearly ordered by initial segment.  Since both premice are $\omega$-sound and project to $\omega$, the usual no-drop argument for the shorter side pulls the conclusion back to the original premice.  Hence $\mathcal M\trianglelefteq\mathcal N$ or $\mathcal N\trianglelefteq\mathcal M$.
\end{proof}

\paragraph{Preservation convention for the outer models used below.}\label{par:steel-preservation-convention}
We shall use Lemma~\ref{lem:local-comparison-m1} only in the forcing extensions which occur in this paper.  The preservation fact needed there is the following: if $\mathcal M\triangleleft M_1$ is countable and belongs to one of these extensions, then $\mathcal M$ remains $\Pi^1_2$-iterable there.  This is the realizability preservation supplied by Steel's weak iteration-game analysis for initial segments of $M_1$, applied to the proper, $\omega_1$-preserving forcing extensions used below.  We do not claim that the set of all real codes for $\Pi^1_2$-iterable premice is absolute between arbitrary outer models with the same $\omega_1$.

\begin{lemma}
\label{lem:definable-m1-initial-segments}
Let $M_1[G]$ be an $\omega_1$-preserving forcing extension of $M_1$ of the kind fixed in Paragraph~\ref{par:steel-preservation-convention}.  In $M_1[G]$ let
\[
\begin{split}
   \mathcal I=\bigl\{\mathcal N\mid {}&
      \mathcal N\text{ is a countable ordinary premouse, }\\
      &\mathcal N\text{ is }1\text{-small, }\mathcal N\text{ is }\omega\text{-sound, }\\
      &\rho_\omega(\mathcal N)=\omega,
        \text{ and }\mathcal N\text{ is }\Pi^1_2\text{-iterable}\bigr\}.
\end{split}
\]
Then $\mathcal I$ is $\Pi^1_2$-definable in the codes.  Moreover every member of $\mathcal I$ is of the form $\mathcal J_\eta^{M_1}$ for some $\eta<\omega_1$, and
\[
   \{\eta<\omega_1\mid \mathcal J_\eta^{M_1}\in\mathcal I\}
\]
is cofinal in $\omega_1$.
\end{lemma}

\begin{proof}
The clauses saying that a real codes a countable ordinary premouse, that the premouse is $1$-small, that it is $\omega$-sound, and that $\rho_\omega=\omega$ are arithmetic, after fixing the usual coding of countable premice by reals.  By Steel's analysis of the weak iteration games, $\Pi^1_2$-iterability is a $\Pi^1_2$ condition in the code.  Hence the displayed definition of $\mathcal I$ is $\Pi^1_2$.

We next show that no nonstandard premouse enters $\mathcal I$.  Work in $M_1[G]$ and let $\mathcal N\in\mathcal I$.  Since $\omega_1$ is preserved, $\mathcal N$ has countable height below the true $\omega_1^{M_1}$.  Choose $\eta<\omega_1$ such that $\operatorname{Ord}^{\mathcal N}<\eta$ and such that $\mathcal J_\eta^{M_1}$ is $\omega$-sound and projects to $\omega$.  The initial segment $\mathcal J_\eta^{M_1}$ is an ordinary premouse and is realizable inside $M_1$; hence it is $\Pi^1_2$-iterable in the present extension by Paragraph~\ref{par:steel-preservation-convention}.  Applying Lemma~\ref{lem:local-comparison-m1} to $\mathcal J_\eta^{M_1}$ and $\mathcal N$, the alternative $\mathcal J_\eta^{M_1}\triangleleft\mathcal N$ is impossible by the choice of $\eta$.  Therefore $\mathcal N\trianglelefteq\mathcal J_\eta^{M_1}$, so $\mathcal N$ is itself an initial segment of $M_1$.

Finally, the fine structure of $M_1$ gives cofinally many $\eta<\omega_1$ such that $\mathcal J_\eta^{M_1}$ is $\omega$-sound and projects to $\omega$.  These levels are ordinary premice, $1$-small, and realizable, hence $\Pi^1_2$-iterable in the present extension by the same preservation convention.  Therefore they belong to $\mathcal I$, proving cofinality.
\end{proof}

\begin{definition}\label{def:recover-m1-omega1}
In any $\omega_1$-preserving forcing extension of $M_1$ we write
\[
   \mathcal N=M_1|\omega_1
\]
for the assertion that
\[
   \mathcal N=\bigcup\{\mathcal M\mid \mathcal M\in\mathcal I\},
\]
where $\mathcal I$ is the class from Lemma~\ref{lem:definable-m1-initial-segments}.
Equivalently, $x\in\mathcal N$ iff $x$ belongs to some countable ordinary premouse which is $1$-small, $\omega$-sound, $\Pi^1_2$-iterable, and projects to $\omega$.
\end{definition}

\begin{lemma}\label{lem:recover-m1-omega1}
Let $M_1[G]$ be an $\omega_1$-preserving forcing extension of $M_1$.  Then Definition~\ref{def:recover-m1-omega1} defines the true initial segment
\[
   M_1|\omega_1=\mathcal J^{M_1}_{\omega_1}.
\]
Moreover the definition is uniform in all such extensions.
\end{lemma}

\begin{proof}
By Lemma~\ref{lem:definable-m1-initial-segments}, every member of $\mathcal I$ is an initial segment $\mathcal J^{M_1}_\eta$ with $\eta<\omega_1$.  Hence the union in Definition~\ref{def:recover-m1-omega1} is contained in $M_1|\omega_1$.  Conversely, the same lemma gives cofinally many $\eta<\omega_1$ with $\mathcal J_\eta^{M_1}\in\mathcal I$.  Their union is $\mathcal J^{M_1}_{\omega_1}$.  The same formula defining $\mathcal I$ is used in every such extension.  We do not assert that the same real codes belong to $\mathcal I$ in different extensions, since new reals may code new countable putative iteration trees.  Rather, applying Lemma~\ref{lem:definable-m1-initial-segments} inside the given extension identifies the union of the premice satisfying that formula with the true $\mathcal J^{M_1}_{\omega_1}$.
\end{proof}

\paragraph{Remark.}\label{rem:m1-omega1-not-absolute}
The notation in Definition~\ref{def:recover-m1-omega1} is a convention for outer models of $M_1$ which preserve $\omega_1$.  The same first-order-looking assertion, if evaluated inside an arbitrary transitive model, need not imply that the object obtained is the true $M_1|\omega_1$.  Later, whenever this definition is used inside countable auxiliary models, the relevant countable premouse is also required externally to belong to the class $\mathcal I$.

We shall need a canonical diamond sequence which is available from the same fine structure.  The argument is Jensen's proof of diamond in $L$, with Steel's condensation theorem for initial segments of $M_1$ replacing condensation for $L$.

\begin{theorem}\label{thm:steel-condensation-m1}
Let $\mathcal M\trianglelefteq M_1$ be an $\omega$-sound initial segment and let
\[
   \pi:\bar N\rightarrow\mathcal M
\]
be the inverse of the transitive collapse of a sufficiently elementary substructure of $\mathcal M$.  Suppose that the critical point of $\pi$ is the relevant standard projectum of $\bar N$.  Then either
\begin{enumerate}
   \item $\bar N\trianglelefteq\mathcal M$, or
   \item $\bar N$ is an initial segment of a degree-zero ultrapower of an initial segment of $\mathcal M$ by an extender on the $M_1$-sequence whose length is that projectum.
\end{enumerate}
In the hulls used in the diamond argument below, the second alternative is impossible.  Hence the transitive collapse is an initial segment of $M_1$.
\end{theorem}

\begin{proof}
This is the standard condensation theorem for the Steel $M_1$-construction; see \cite[Theorem~5.1]{Steel3}.  We only spell out why the ultrapower alternative does not occur in the present application.  The hulls below are chosen so that their collapse has projectum equal to its internal $\omega_1$.  If the second alternative held, there would be an extender on the $M_1$-sequence indexed exactly at this internal $\omega_1$.  The lower part below that index is the collapse of the corresponding lower part of the hull and sees the index as the successor cut reached by the construction.  An extender indexed there would make the index inaccessible in the relevant extender model.  This contradicts the fact that the collapsed lower part computes it as its $\omega_1$.  Therefore only the initial-segment alternative remains.
\end{proof}

\begin{lemma}\label{lem:m1-diamond}
There is a sequence
\[
   \vec D=\langle D_\alpha\mid \alpha<\omega_1\rangle\in M_1
\]
with $D_\alpha\subseteq\alpha$ for all $\alpha<\omega_1$ such that
\[
   M_1\models ``\vec D\text{ is a }\diamondsuit_{\omega_1}\text{-sequence}''.
\]
Moreover $\vec D$ is uniformly definable over $M_1|\omega_1$ from the canonical order of construction of $M_1$.
\end{lemma}

\begin{proof}
Work in $M_1$.  Use the canonical well-order $<_{M_1}$ of the construction to define $\vec D$ recursively.  At a limit ordinal $\alpha<\omega_1$, suppose $\langle D_\beta\mid\beta<\alpha\rangle$ has been defined.  If there is a pair $(A,C)$ such that $A\subseteq\alpha$, $C\subseteq\alpha$ is club in $\alpha$, and
\[
   \forall\beta\in C\, (D_\beta\neq A\cap\beta),
\]
then let $(A_\alpha,C_\alpha)$ be the $<_{M_1}$-least such pair and set $D_\alpha=A_\alpha$.  If there is no such pair, set $D_\alpha=\emptyset$.  At successor stages we may again set $D_\alpha=\emptyset$.

The recursion is carried out over $M_1|\omega_1$.  Indeed, all objects considered at stage $\alpha$ are subsets of the countable ordinal $\alpha$ in $M_1$, and the order used to choose the least pair is the restriction of the canonical $M_1$ construction order.  Hence the resulting sequence is uniformly definable over $M_1|\omega_1$.

It remains to verify that $\vec D$ is a diamond sequence.  Suppose not.  Let $(A,C)$ be the $<_{M_1}$-least counterexample, so $A\subseteq\omega_1$, $C\subseteq\omega_1$ is club, and
\[
   \forall\alpha\in C\,(D_\alpha\neq A\cap\alpha).
\]
Choose an $\omega$-sound initial segment $\mathcal J^{M_1}_\theta$ containing $A$, $C$, and the sequence $\vec D$, and then take a countable elementary substructure
\[
   X\prec \mathcal J^{M_1}_\theta
\]
with $A,C,\vec D\in X$ and with $\alpha=X\cap\omega_1\in C$.  Let
\[
   \pi:\bar X\rightarrow X
\]
be the inverse of the transitive collapse.  By Theorem~\ref{thm:steel-condensation-m1}, $\bar X$ is an initial segment of $M_1$.  Consequently the recursive construction of $\vec D$ inside $\bar X$ is exactly the initial part
\[
   \langle D_\beta\mid\beta<\alpha\rangle.
\]
Moreover the collapse sends $A$ to $A\cap\alpha$ and $C$ to $C\cap\alpha$, and by elementarity $\bar X$ sees $(A\cap\alpha,C\cap\alpha)$ as the $<_{M_1}$-least witness that the previous sequence is not diamond on $\alpha$.  Therefore the recursion at stage $\alpha$ gives
\[
   D_\alpha=A\cap\alpha.
\]
Since $\alpha\in C$, this contradicts the choice of $(A,C)$.  Thus no counterexample exists, and $\vec D$ is a $\diamondsuit_{\omega_1}$-sequence in $M_1$.
\end{proof}

We fix once and for all the $<_{M_1}$-least sequence $\vec D$ satisfying Lemma~\ref{lem:m1-diamond}.  In later sections this sequence will be used to define a canonical $\omega_1$-sequence of independent Suslin trees over $M_1$.  The construction of those trees is postponed until the point where the coding machinery needs them.

\subsection{The canonical inner model with $n$-many Woodin cardinals}\label{subsec:mn-background}

We now record the higher-level analogue of Subsection~\ref{subsec:m1-background}.  Throughout this subsection $1\leq n<\omega$ is fixed, and $M_n$ denotes the canonical minimal proper class premouse with $n$ Woodin cardinals, in the sense of Steel's construction of tame mice with full background extenders.  Thus $M_0=L$, and for $n>0$ the model $M_n$ is obtained from the Steel background construction by stopping at the first failure of $n$-smallness, or as the limit of the $n$-small construction if no such failure occurs.  We shall use only the following standard consequences of Steel's analysis.

First, if there are $n$ Woodin cardinals, then $M_n$ exists and satisfies that there are $n$ Woodin cardinals.  Secondly, every proper initial segment of $M_n$ is $n$-small and $\omega$-sound.  Thirdly, the order of construction of $M_n$ gives a canonical good well-order of the reals of $M_n$, and Steel's projective analysis shows that this well-order is $\Delta^1_{n+2}$ over the reals of $M_n$; see \cite{Steel2,Steel3}.  The case $n=1$ is exactly the situation isolated in the previous subsection.

\begin{definition}\label{def:n-smallness-general}
Let $\mathcal M$ be a premouse and let $\eta<\operatorname{Ord}^{\mathcal M}$.  We say that $\mathcal M$ is \emph{$n$-small above $\eta$} if whenever $E$ is an extender on the $\mathcal M$-sequence and
\[
   \eta<\operatorname{crit}(E),
\]
then the initial segment $\mathcal J^{\mathcal M}_{\operatorname{crit}(E)}$ does not have $n$ Woodin cardinals above $\eta$.  We say that $\mathcal M$ is \emph{$n$-small} if it is $n$-small above $0$.
\end{definition}

We shall also use Steel's finite-level iterability condition.  The exact definition depends on the parity of $n$.  For even $n$, one uses the $n$-round weak iteration game in which the branches played by player II are required to be sufficiently definable over the trees played by player I.  For odd $n$, one adds the corresponding $\alpha$-goodness requirement at the last round.  In both cases the definition is arranged so that realizable mice satisfy it, and so that it is projectively simple in the codes.

\begin{definition}\label{def:pin-iterability}
Let $\mathcal M$ be a countable premouse and let $\eta<\operatorname{Ord}^{\mathcal M}$ be a cutpoint.  We say that $\mathcal M$ is \emph{$\Pi_n$-iterable above $\eta$} if player II has a winning strategy in Steel's game $\mathcal G(\mathcal M,\eta,n)$ for $\Pi_n$-iterability above $\eta$.  We say simply that $\mathcal M$ is \emph{$\Pi_n$-iterable} if it is $\Pi_n$-iterable above $0$.
\end{definition}

For later reference we isolate the three consequences of the definition which are used in this paper.

\begin{fact}\label{fact:steel-pin-iterability}
For each fixed $1\leq n<\omega$ the following hold.
\begin{enumerate}
   \item The relation
   \[
      ``x\text{ codes a countable premouse which is }\Pi_n\text{-iterable}''
   \]
   is $\boldsymbol{\Pi}^1_{n+1}$ in the real code $x$.
   \item If $\mathcal M$ is a countable initial segment of $M_n$, then $\mathcal M$ is $\Pi_n$-iterable in the Steel sense, in all outer models considered in this paper in which the relevant realizability argument is preserved.
   \item If $\mathcal M$ is countable, $n$-small and realizable, and $\mathcal N$ is countable, $n$-small and $\Pi_n$-iterable, then the Steel comparison of $\mathcal M$ and $\mathcal N$ is successful, provided the two mice have the same lower part.  In the applications below the lower part is empty, so this is the $\eta=0$ case.
\end{enumerate}
\end{fact}

The next lemma is the direct analogue of Lemma~\ref{lem:local-comparison-m1}.  We state it separately because it is the form in which it will be used later.

\begin{lemma}\label{lem:local-comparison-mn}
Let $\mathcal M$ and $\mathcal N$ be countable ordinary premice.  Assume that both are $\omega$-sound and project to $\omega$, that $\mathcal M\triangleleft M_n$, and that $\mathcal N$ is $n$-small and $\Pi_n$-iterable.  Then the comparison of $\mathcal M$ with $\mathcal N$ is successful.  Consequently
\[
   \mathcal M\trianglelefteq\mathcal N
   \quad\text{or}\quad
   \mathcal N\trianglelefteq\mathcal M .
\]
\end{lemma}

\begin{proof}
This is the $\eta=0$ instance of Steel's comparison theorem for $n$-small $\Pi_n$-iterable mice.  Since $\mathcal M\triangleleft M_n$, the premouse $\mathcal M$ is realizable in the background construction of $M_n$.  Since $\mathcal N$ is $n$-small and $\Pi_n$-iterable, the $\mathcal N$-side has exactly the amount of iterability required by Steel's comparison lemma.  Because both premice are ordinary premice in the parameter-free sense fixed above, the lower parts agree trivially.

The proof is the usual least-disagreement coiteration.  The $\mathcal M$-side uses the realization strategy inherited from the construction of $M_n$, while the $\mathcal N$-side uses the $\Pi_n$-iterability strategy.  At limit stages, the good branch is characterized by the corresponding $\mathcal Q$-structure: for $n=1$ this is the branch uniqueness argument described in Subsection~\ref{subsec:m1-background}, and for $n>1$ the assertion is proved by induction on $n$, because the $\mathcal Q$-structure appearing at a limit stage is $(n-1)$-small and $\Pi_{n-1}$-iterable above the new cutpoint.  If two distinct good branches were available, Steel's zipper argument would produce the forbidden strength below the comparison height, contradicting $n$-smallness at the relevant level.  Thus the comparison has well-founded branches and reaches terminal models which are linearly ordered by initial segment.

Finally, since both premice are $\omega$-sound and project to $\omega$, the standard no-drop pullback argument gives the initial-segment relation already between the original premice.  Hence $\mathcal M\trianglelefteq\mathcal N$ or $\mathcal N\trianglelefteq\mathcal M$.
\end{proof}

\paragraph{Preservation convention for the outer models used below.}\label{par:steel-mn-preservation-convention}
As in the $M_1$ case, we shall only use the preceding comparison in forcing extensions which occur in this paper.  The preservation fact needed is the following one: if $\mathcal M\triangleleft M_n$ is countable and belongs to such an extension, then $\mathcal M$ remains $\Pi_n$-iterable there.  This is the higher-level version of the realizability preservation used in Paragraph~\ref{par:steel-preservation-convention}.  We do not claim that the class of all real codes for $\Pi_n$-iterable premice is absolute between arbitrary outer models with the same $\omega_1$.

For the coding arguments we only need to recover the lower part $M_n|\omega_1$.  Thus we isolate the lower-part initial segments of $M_n$, rather than the larger class of all $n$-small $\Pi_n$-iterable mice.  This distinction is harmless for $n=1$, but it is important at higher odd levels, where one has to avoid the familiar nonstandard $\Pi_n$-iterable mice.  The lower-part restriction below excludes these objects from the definition used in the coding apparatus.

\begin{definition}\label{def:mn-lower-part-approximations}
Let $M_n[G]$ be an $\omega_1$-preserving forcing extension of $M_n$ of the kind fixed in Paragraph~\ref{par:steel-mn-preservation-convention}.  In $M_n[G]$ let $\mathcal I_n$ be the class of all $\mathcal N$ such that
\begin{enumerate}
   \item $\mathcal N$ is a countable passive ordinary premouse;
   \item $\mathcal N$ is a lower-part premouse, i.e. $\mathcal N$ has no Woodin cardinal;
   \item $\mathcal N$ is $n$-small, $\omega$-sound, and $\rho_\omega(\mathcal N)=\omega$;
   \item $\mathcal N$ is $\Pi_n$-iterable.
\end{enumerate}
\end{definition}

\begin{lemma}\label{lem:definable-mn-initial-segments}
Let $M_n[G]$ be an $\omega_1$-preserving forcing extension of $M_n$ of the kind fixed above.  Then $\mathcal I_n$ is $\boldsymbol{\Pi}^1_{n+1}$-definable in the codes.  Moreover every member of $\mathcal I_n$ is of the form $\mathcal J^{M_n}_\eta$ for some $\eta<\omega_1$, and
\[
   \{\eta<\omega_1\mid \mathcal J^{M_n}_\eta\in\mathcal I_n\}
\]
is cofinal in $\omega_1$.
\end{lemma}

\begin{proof}
The clauses saying that a real codes a countable passive ordinary premouse, that the premouse is lower-part, $n$-small, $\omega$-sound, and satisfies $\rho_\omega=\omega$, are arithmetic in the usual code for countable premice.  By Fact~\ref{fact:steel-pin-iterability}, the $\Pi_n$-iterability clause is $\boldsymbol{\Pi}^1_{n+1}$.  Therefore $\mathcal I_n$ is $\boldsymbol{\Pi}^1_{n+1}$ in the codes.

We next prove that no nonstandard lower-part premouse enters $\mathcal I_n$.  Work in $M_n[G]$ and let $\mathcal N\in\mathcal I_n$.  Since $\omega_1$ is preserved, the height of $\mathcal N$ is below the true $\omega_1^{M_n}$.  Choose $\eta<\omega_1$ such that $\operatorname{Ord}^{\mathcal N}<\eta$ and such that $\mathcal J^{M_n}_\eta$ is passive, lower-part, $\omega$-sound, and projects to $\omega$.  The fine structure of $M_n$ gives cofinally many such $\eta$.  The initial segment $\mathcal J^{M_n}_\eta$ is realizable inside $M_n$, and hence is $\Pi_n$-iterable in the present extension by the preservation convention.

Apply Lemma~\ref{lem:local-comparison-mn} to $\mathcal J^{M_n}_\eta$ and $\mathcal N$.  The alternative $\mathcal J^{M_n}_\eta\trianglelefteq\mathcal N$ is impossible by the choice of $\eta$, because $\operatorname{Ord}^{\mathcal N}<\eta$.  Therefore
\[
   \mathcal N\trianglelefteq \mathcal J^{M_n}_\eta,
\]
and hence $\mathcal N$ is itself an initial segment of $M_n$.

Conversely, let $\eta<\omega_1$ be such that $\mathcal J^{M_n}_\eta$ is passive, lower-part, $\omega$-sound, and projects to $\omega$.  Then $\mathcal J^{M_n}_\eta$ is $n$-small and realizable in the Steel construction of $M_n$, and by the preservation convention it is $\Pi_n$-iterable in $M_n[G]$.  Thus $\mathcal J^{M_n}_\eta\in\mathcal I_n$.  Since such $\eta$ are cofinal in $\omega_1$, the cofinality assertion follows.
\end{proof}

\begin{definition}\label{def:recover-mn-omega1}
In any $\omega_1$-preserving forcing extension of $M_n$ of the kind fixed above, we write
\[
   \mathcal N=M_n|\omega_1
\]
if
\[
   \mathcal N=\bigcup\mathcal I_n,
\]
where $\mathcal I_n$ is the class from Definition~\ref{def:mn-lower-part-approximations}, computed in that extension.
\end{definition}

\begin{lemma}\label{lem:recover-mn-omega1}
Let $M_n[G]$ be an $\omega_1$-preserving forcing extension of $M_n$ of the kind fixed above.  Then Definition~\ref{def:recover-mn-omega1} defines the true initial segment
\[
   M_n|\omega_1=\mathcal J^{M_n}_{\omega_1}.
\]
Moreover the definition is uniform in all such extensions.
\end{lemma}

\begin{proof}
By Lemma~\ref{lem:definable-mn-initial-segments}, every member of $\mathcal I_n$ is an initial segment $\mathcal J^{M_n}_\eta$ with $\eta<\omega_1$.  Hence the union in Definition~\ref{def:recover-mn-omega1} is contained in $M_n|\omega_1$.  Conversely, the same lemma gives cofinally many $\eta<\omega_1$ with $\mathcal J^{M_n}_\eta\in\mathcal I_n$.  The union of these initial segments is $\mathcal J^{M_n}_{\omega_1}$.

The same formula defining $\mathcal I_n$ is used in every relevant extension.  We do not assert that the same real codes belong to $\mathcal I_n$ in different extensions, since new reals may code new countable putative iteration trees.  Rather, applying Lemma~\ref{lem:definable-mn-initial-segments} inside the given extension identifies the union of the premice satisfying the formula with the true $\mathcal J^{M_n}_{\omega_1}$.
\end{proof}

\paragraph{Remark.}\label{rem:mn-omega1-not-absolute}
The notation $M_n|\omega_1$ in Definition~\ref{def:recover-mn-omega1} is a convention for the outer models of $M_n$ used in this paper.  It should not be read as saying that an arbitrary transitive model which internally satisfies the same first-order-looking definition has computed the true $M_n|\omega_1$.  Later, whenever countable auxiliary models are used, the relevant countable premouse is also required externally to belong to $\mathcal I_n$.

The condensation and diamond arguments from the previous subsection lift without change once $M_1$ is replaced by $M_n$.

\begin{theorem}\label{thm:steel-condensation-mn}
Let $\mathcal M\trianglelefteq M_n$ be an $\omega$-sound initial segment and let
\[
   \pi:\bar N\rightarrow\mathcal M
\]
be the inverse of the transitive collapse of a sufficiently elementary substructure of $\mathcal M$.  Suppose that the critical point of $\pi$ is the relevant standard projectum of $\bar N$.  Then either
\begin{enumerate}
   \item $\bar N\trianglelefteq\mathcal M$, or
   \item $\bar N$ is an initial segment of a degree-zero ultrapower of an initial segment of $\mathcal M$ by an extender on the $M_n$-sequence whose length is that projectum.
\end{enumerate}
In the hulls used in the diamond argument below, the second alternative is impossible.  Hence the transitive collapse is an initial segment of $M_n$.
\end{theorem}

\begin{proof}
This is the condensation theorem for initial segments of the Steel $M_n$-construction.  The proof is the same fine-structural argument as in the $M_1$ case, with $n$-smallness replacing $1$-smallness.  The only point needed below is the exclusion of the ultrapower alternative.  The hulls in the diamond argument are chosen so that the collapsed structure has projectum equal to its internal $\omega_1$.  If the ultrapower alternative occurred, an extender on the $M_n$-sequence would be indexed at this internal $\omega_1$.  The lower part below that index computes the index as its first uncountable cardinal, while the presence of such an extender would make it inaccessible in the relevant extender model.  This contradiction leaves only the initial-segment alternative.
\end{proof}

\begin{fact}\label{fact:mn-projective-wellorder}
The construction order $<_{M_n}$ restricted to the reals of $M_n$ is a $\Delta^1_{n+2}$ well-order.  Equivalently, for reals $x,y\in M_n$, the assertion that $x$ is constructed before $y$ in $M_n$ is given by a $\Sigma^1_{n+2}$ formula and also by a $\Pi^1_{n+2}$ formula.
\end{fact}

We shall use this fact only as a source of canonical choices inside \(M_n\).
In particular, using \(<_{M_n}\) we fix once and for all the canonical
almost disjoint family \[D=\langle d_\xi\mid \xi<\omega_1\rangle\in M_n\]
used for almost disjoint coding, and the canonical \(\diamondsuit_{\omega_1}\)-
sequence used in the bookkeeping.  No later argument uses the well-order
to choose elements from projective sections in the final model.

\begin{lemma}\label{lem:mn-diamond}
There is a sequence
\[
   \vec D^{\,n}=\langle D^n_\alpha\mid \alpha<\omega_1\rangle\in M_n
\]
with $D^n_\alpha\subseteq\alpha$ for all $\alpha<\omega_1$ such that
\[
   M_n\models ``\vec D^{\,n}\text{ is a }\diamondsuit_{\omega_1}\text{-sequence}''.
\]
Moreover $\vec D^{\,n}$ is uniformly definable over $M_n|\omega_1$ from the canonical order of construction of $M_n$.
\end{lemma}

\begin{proof}
Work in $M_n$.  Use the canonical well-order $<_{M_n}$ of the construction to define $\vec D^{\,n}$ recursively.  At a limit ordinal $\alpha<\omega_1$, suppose $\langle D^n_\beta\mid\beta<\alpha\rangle$ has already been defined.  If there is a pair $(A,C)$ such that $A\subseteq\alpha$, $C\subseteq\alpha$ is club in $\alpha$, and
\[
   \forall\beta\in C\,(D^n_\beta\neq A\cap\beta),
\]
then let $(A_\alpha,C_\alpha)$ be the $<_{M_n}$-least such pair and set $D^n_\alpha=A_\alpha$.  If there is no such pair, set $D^n_\alpha=\emptyset$.  At successor stages we again set $D^n_\alpha=\emptyset$.

The recursion is carried out over $M_n|\omega_1$, because at stage $\alpha$ all relevant objects are subsets of the countable ordinal $\alpha$ in $M_n$, and the order used to choose the least pair is the restriction of the canonical construction order of $M_n$.

Suppose toward a contradiction that $\vec D^{\,n}$ is not a diamond sequence in $M_n$.  Let $(A,C)$ be the $<_{M_n}$-least counterexample, so $A\subseteq\omega_1$, $C\subseteq\omega_1$ is club, and
\[
   \forall\alpha\in C\,(D^n_\alpha\neq A\cap\alpha).
\]
Choose an $\omega$-sound initial segment $\mathcal J^{M_n}_\theta$ containing $A$, $C$, and the sequence $\vec D^{\,n}$, and take a countable elementary substructure
\[
   X\prec \mathcal J^{M_n}_\theta
\]
with $A,C,\vec D^{\,n}\in X$ and with $\alpha=X\cap\omega_1\in C$.  Let
\[
   \pi:\bar X\rightarrow X
\]
be the inverse of the transitive collapse.  By Theorem~\ref{thm:steel-condensation-mn}, the collapse $\bar X$ is an initial segment of $M_n$.  Therefore the recursive construction of $\vec D^{\,n}$ inside $\bar X$ is exactly the initial part
\[
   \langle D^n_\beta\mid\beta<\alpha\rangle.
\]
The collapse sends $A$ to $A\cap\alpha$ and $C$ to $C\cap\alpha$, and by elementarity $\bar X$ sees $(A\cap\alpha,C\cap\alpha)$ as the $<_{M_n}$-least witness that the previous sequence is not diamond on $\alpha$.  Hence the recursion at stage $\alpha$ gives
\[
   D^n_\alpha=A\cap\alpha.
\]
This contradicts $\alpha\in C$.  Therefore no counterexample exists, and $\vec D^{\,n}$ is a $\diamondsuit_{\omega_1}$-sequence in $M_n$.
\end{proof}

We fix once and for all the $<_{M_n}$-least sequence $\vec D^{\,n}$ satisfying Lemma~\ref{lem:mn-diamond}.  In the higher-level version of the construction, this sequence plays exactly the role played by the $M_1$-diamond sequence in the detailed $M_1$ case: it gives the canonical lower-part bookkeeping from which the corresponding sequence of independent Suslin trees and the localized coding apparatus are built.

\subsection{The trees $T_n$, weak homogeneity, and small generic absoluteness}\label{subsec:tn-weak-homogeneity}

We next fix the tree representations of the projective pointclasses which will be used in the regularity argument.  The main application in this paper is the case $n=2$: the Martin--Solovay tree $T_2\in M_1$ represents the universal $\boldsymbol{\Sigma}^1_3$ set.  Since the higher-level notation is no harder, we record the construction uniformly.  Thus, for $n\geq 2$, the tree $T_n$ will be a canonical tree belonging to $M_{n-1}$ whose projection is the universal $\boldsymbol{\Sigma}^1_{n+1}$ set of reals.  In the final model we shall use only the following consequence of its construction: membership in $p[T_n]$ is absolute to the small generic extensions in which the relevant real appears.

We recall the form of weak homogeneity used below.  The notation follows Steel's exposition of the Martin--Solovay theorem in terms of towers of measures~\cite{steel2009derived}.  If $Z$ is a set and $\kappa$ is a cardinal, let $\operatorname{meas}_\kappa(Z)$ be the set of $\kappa$-additive measures on $Z^{<\omega}$.  If $\mu\in\operatorname{meas}_\kappa(Z)$, its dimension is the unique $m<\omega$ such that $\mu$ concentrates on $Z^m$.  If $\mu$ has dimension $m$, $\nu$ has dimension $k$, and $m\leq k$, we say that $\nu$ projects to $\mu$ if for every $A\subseteq Z^m$,
\[
   A\in\mu
   \quad\Longleftrightarrow\quad
   \{u\in Z^k\mid u\upharpoonright m\in A\}\in\nu .
\]
A tower $\langle\mu_i\mid i<\omega\rangle$ is countably complete if, whenever $A_i\in\mu_i$ for every $i$, there is an $f\in Z^\omega$ such that $f\upharpoonright\operatorname{dim}(\mu_i)\in A_i$ for every $i$.

\begin{definition}\label{def:weakly-homogeneous-tree}
Let $T$ be a tree on $\omega\times Z$.  For $x\in\omega^\omega$ write
\[
   T_x=\{u\in Z^{<\omega}\mid \forall k<|u|\ ((x\upharpoonright k,u\upharpoonright k)\in T)\}.
\]
We say that $T$ is \emph{$\kappa$-weakly homogeneous} if there is a countable set
\[
   \mathcal M_T\subseteq\operatorname{meas}_\kappa(Z)
\]
closed under projections such that, for every real $x$,
\[
   x\in p[T]
\]
if and only if there is a countably complete tower $\langle\mu_i\mid i<\omega\rangle$ from $\mathcal M_T$ such that each $\mu_i$ concentrates on the corresponding finite section of $T_x$.  Equivalently, $T$ is weakly homogeneous in the sense of a weak homogeneity system whose range is countable.

When defined from such a witnessing system, we call $p[T]$ a $\kappa$-weakly homogeneously Suslin set.
\end{definition}

\begin{definition}\label{def:absolute-complements}
Let $T$ be a tree on $\omega\times Z$ and $S$ a tree on $\omega\times Z'$.  We say that $T$ and $S$ are \emph{$\kappa$-absolute complements} if, for every forcing extension by a partial order of size $<\kappa$,
\[
   p[T]=\omega^\omega\setminus p[S].
\]
A set of reals is \emph{$\kappa$-universally Baire} if it is the projection of a tree which has a $\kappa$-absolute complement.
\end{definition}

The connection between the preceding two notions is the Martin--Solovay tree construction.  If $\bar\mu$ is a weak homogeneity system for $T$ and $\Theta$ is sufficiently large, the Martin--Solovay tree
\[
   \operatorname{ms}(\bar\mu,\Theta)
\]
searches for a coherent descending sequence of ordinal ranks through the ultrapowers determined by the measures in $\bar\mu$.  A branch through this Martin--Solovay tree is precisely a continuous certificate that the corresponding tower is ill-founded.  The following is the form of the Martin--Solovay theorem used in the rest of the paper~\cite[Theorem~2.20 and Corollary~2.21]{steel2009derived}.

\begin{theorem}[Martin--Solovay]\label{thm:ms-weak-homogeneous-ub}
Let $T$ be $\kappa$-weakly homogeneous via a weak homogeneity system $\bar\mu$, and let $\Theta>|T|^+$.  Then $T$ and $\operatorname{ms}(\bar\mu,\Theta)$ are $\kappa$-absolute complements.  In particular, $p[T]$ is $\kappa$-universally Baire.  Moreover, in every forcing extension by a partial order of size $<\kappa$, the measures in $\bar\mu$ lift to a weak homogeneity system witnessing the same assertion for the same ground-model tree $T$.
\end{theorem}

\begin{proof}
This is the Martin--Solovay theorem for weakly homogeneous trees.  The key point is that a weak homogeneity system gives, for each real $x$, a countable tree of possible measure towers.  If one of these towers is countably complete, then $x\in p[T]$.  If no such tower is countably complete, the associated ill-foundedness is witnessed uniformly by a descending sequence of ordinal ranks; these ranks form a branch through $\operatorname{ms}(\bar\mu,\Theta)_x$, provided $\Theta$ is chosen above the size of the relevant tree.

Small forcing preserves the measure towers in the required way: the measures extend canonically to the forcing extension, and functions in the extension are represented modulo the lifted measures by ground-model functions.  Consequently the same Martin--Solovay tree remains the complementing tree in every $<\kappa$-generic extension.  This is exactly Steel's proof of the weakly homogeneous case of the Martin--Solovay theorem.
\end{proof}

We shall use Theorem~\ref{thm:ms-weak-homogeneous-ub} only in a localized form.  Let $M$ be one of the mice $M_{n-1}$ and let $\delta$ be its least Woodin cardinal.  If $T\in M$ is $\kappa$-weakly homogeneous in $M$ for every $\kappa<\delta$, then for every forcing $\mathbb P\in M$ of $M$-cardinality $<\delta$, and every $M$-generic $G\subseteq\mathbb P$, the same tree $T$ and its Martin--Solovay complement compute the same projective set in $M[G]$.  Equivalently, for reals $x\in M[G]$, membership in $p[T]$ can be checked by the ground-model tree $T$ and is not changed by passing to further forcing extensions of size below the relevant completeness bound.

\begin{corollary}\label{cor:small-generic-absoluteness-fixed-tree}
Let $M$ be a transitive model containing a $\kappa$-weakly homogeneous tree $T$ and a witnessing system $\bar\mu$.  Let $S=\operatorname{ms}(\bar\mu,\Theta)$ for sufficiently large $\Theta$.  If $G$ is generic over $M$ for a forcing of $M$-cardinality $<\kappa$, then in $M[G]$,
\[
   p[T]=\omega^\omega\setminus p[S].
\]
In particular, if a projective formula $\varphi(x)$ is represented over $M$ by the complementing pair $(T,S)$, then for every real $x\in M[G]$,
\[
   M[G]\models \varphi(x)
   \quad\Longleftrightarrow\quad
   x\in p[T]^{M[G]} .
\]
\end{corollary}

\begin{proof}
The first assertion is Theorem~\ref{thm:ms-weak-homogeneous-ub}.  For the final assertion, the projective formula is represented in $M$ by the pair $(T,S)$, and the pair remains an absolute complementing pair in $M[G]$.  Thus exactly one of the trees has a branch over the real $x$ in $M[G]$, and this is the same truth value assigned by the projective definition in the small generic extension.
\end{proof}

\paragraph{Remark.}\label{rem:localized-smallness-of-absoluteness}
Later forcing notions need not be regarded globally as small over the relevant mouse.  What is used is the local consequence: every real and every name appearing in the coding construction is read in a bounded regular subforcing, and the relevant subforcing has size below the completeness bound of the homogeneity system.  The assertion above is therefore applied in the intermediate model generated by that local support.

\subsection{Defining the canonical trees $T_n$ inside $M_{n-1}$}\label{subsec:canonical-trees-tn}

We now choose the particular trees to which the preceding subsection will be applied.  Fix $n\geq 2$.  Work in $M_{n-1}$, and let
\[
   \delta^n_0<\delta^n_1<\cdots<\delta^n_{n-2}
\]
be the Woodin cardinals of $M_{n-1}$.  We write simply $\delta_0$ for the least of them when $n$ is fixed.

Let $U_{n+1}\subseteq\omega^\omega$ be the standard universal $\boldsymbol{\Sigma}^1_{n+1}$ set, with the first real coordinate coding the index and the remaining coordinates coding the parameter and the argument.  We fix this universal set once and for all using the usual Moschovakis parametrization conventions for projective pointclasses~\cite{Moschovakis}.  Thus every boldface $\boldsymbol{\Sigma}^1_{n+1}$ set is a section of $U_{n+1}$.

The following theorem is the precise object needed later.  It is a standard consequence of the Martin--Steel analysis of projective scales, the Martin--Solovay construction, and Steel's tree-production argument for mice with finitely many Woodin cardinals~\cite{martin2008tree,Steel2,steel2009derived}.

\begin{theorem}\label{thm:canonical-tn-existence}
In $M_{n-1}$ there is a tree
\[
   T_n\subseteq (\omega\times\lambda_n)^{<\omega}
\]
for some ordinal $\lambda_n$ of $M_{n-1}$ such that
\[
   p[T_n]=U_{n+1}
\]
inside $M_{n-1}$.  Moreover, for every $\kappa<\delta_0$, the tree $T_n$ is $\kappa$-weakly homogeneous in $M_{n-1}$.
\end{theorem}

\begin{proof}
We recall the construction, since the parity shift is a common source of confusion.  One begins with the projective pointclass immediately below $\boldsymbol{\Sigma}^1_{n+1}$.  By the periodicity theorems and the Martin--Steel scale analysis, the relevant universal set at that lower level has a scale whose associated tree is homogeneous in $M_{n-1}$.  This is the higher analogue of the familiar Martin--Solovay representation of the complete $\boldsymbol{\Sigma}^1_3$ set by the tree $T_2$ over $M_1$.

If the lower-level universal set occurs with the correct polarity, the tree of the scale already gives the required homogeneous representation.  If the polarity is the complementary one, one applies the Martin--Solovay construction to the homogeneity system.  The Martin--Solovay tree represents the complement by searching for coherent descending ordinal ranks in the corresponding ultrapowers.  This is the step which accounts for the odd/even alternation in the projective hierarchy.

Finally, the passage from the lower-level matrix to the universal $\boldsymbol{\Sigma}^1_{n+1}$ set is an existential real projection.  Homogeneous representations are stable under this operation in the weak sense: the additional real witness is absorbed into the weak choice of a countably complete tower.  Equivalently, the projection of a homogeneously Suslin representation is weakly homogeneously Suslin.  Thus one obtains a tree $T_n$ with $p[T_n]=U_{n+1}$ and with $\kappa$-complete weak homogeneity systems for every $\kappa<\delta_0$.

All objects used in this construction are chosen inside $M_{n-1}$.  The completeness of the systems below the least Woodin follows from the extender strength available in $M_{n-1}$ and the standard Steel tree-production argument.
\end{proof}

For definiteness, we make the choice canonical.

\begin{definition}\label{def:canonical-tn}
For $n\geq 2$, $T_n$ denotes the $<_{M_{n-1}}$-least tree $T$ for which $M_{n-1}$ verifies the conclusion of Theorem~\ref{thm:canonical-tn-existence}.  Along with $T_n$ we fix the $<_{M_{n-1}}$-least coherent choice of weak homogeneity systems
\[
   \bar\mu^n_\kappa
   \qquad (\kappa<\delta^n_0)
\]
which witness that $T_n$ is $\kappa$-weakly homogeneous.  We also fix, for sufficiently large $\Theta_n$, the Martin--Solovay complement
\[
   S_n=\operatorname{ms}(\bar\mu^n_\kappa,\Theta_n)
\]
whenever a completeness bound $\kappa$ has been specified.  The value of $S_n$ may depend on the chosen bound $\kappa$, but this will never matter: in each application we choose $\kappa$ above the size of the relevant forcing.
\end{definition}

\begin{lemma}\label{lem:tn-small-generic-absoluteness}
Let $n\geq 2$, let $\mathbb P\in M_{n-1}$ have $M_{n-1}$-cardinality $<\kappa<\delta^n_0$, and let $G\subseteq\mathbb P$ be $M_{n-1}$-generic.  Then in $M_{n-1}[G]$ the ground-model tree $T_n$ and the corresponding Martin--Solovay complement $S_n$ are complements.  Consequently, for every real $x\in M_{n-1}[G]$,
\[
   M_{n-1}[G]\models x\in U_{n+1}
   \quad\Longleftrightarrow\quad
   x\in p[T_n]^{M_{n-1}[G]} .
\]
The same equivalence remains true in all further forcing extensions of $M_{n-1}[G]$ by forcing notions of size below the remaining completeness bound.
\end{lemma}

\begin{proof}
By Definition~\ref{def:canonical-tn}, $T_n$ is $\kappa$-weakly homogeneous in $M_{n-1}$ via $\bar\mu^n_\kappa$.  Theorem~\ref{thm:ms-weak-homogeneous-ub} gives a Martin--Solovay complement $S_n$ such that $T_n$ and $S_n$ are $\kappa$-absolute complements.  Since $|\mathbb P|^{M_{n-1}}<\kappa$, the complementing relation holds in $M_{n-1}[G]$.

The tree $T_n$ represents the universal $\boldsymbol{\Sigma}^1_{n+1}$ set in the ground model, and the complementing pair remains absolute in the extension.  Therefore a real $x$ in the extension satisfies the projective universal formula exactly when the $T_n$-section above $x$ is ill-founded.  The final sentence is the same argument applied once more to the lifted homogeneity system.
\end{proof}

\begin{corollary}\label{cor:t2-sigma13-absoluteness}
In $M_1$ there is a canonical tree $T_2$ such that
\[
   p[T_2]=U_3,
\]
where $U_3$ is the universal $\boldsymbol{\Sigma}^1_3$ set.  If $\delta$ is the Woodin cardinal of $M_1$, then for every $\kappa<\delta$, $T_2$ is $\kappa$-weakly homogeneous.  Hence the $T_2$ representation of $\boldsymbol{\Sigma}^1_3$ truth is absolute to all forcing extensions of $M_1$ obtained by forcing of size $<\kappa$, and locally to all later intermediate extensions whose relevant support has size $<\kappa$.
\end{corollary}

\begin{proof}
This is Lemma~\ref{lem:tn-small-generic-absoluteness} with $n=2$.  The tree $T_2$ is the Martin--Solovay tree in the form used by Hjorth and Solovay~\cite{martin1969basis,Hjorth}: it represents the universal $\boldsymbol{\Sigma}^1_3$ set and carries weak homogeneity systems below the Woodin cardinal of $M_1$.
\end{proof}

\paragraph{Remark (how $T_n$ will be used).}\label{rem:tn-use-later}
The tree $T_n$ is not an additional coding device.  It is a fixed descriptive-set-theoretic witness, belonging to the relevant mouse, which makes the projective truth predicate robust in small generic extensions.  In the proof of Lebesgue measurability and the Baire property, the only instance needed is $T_2$: after a real parameter $a$ appears in a local intermediate model, cofinally many later Random and Cohen stages ensure that almost every, respectively comeagerly many, reals are generic over the corresponding $L[T_2,a]$-type model.  The weak homogeneity and Martin--Solovay complementing pair are what allows Hjorth's Solovay-style argument to read the $\boldsymbol{\Sigma}^1_3$ definition correctly in those generic extensions.

\section{The coding forcings which are used}

\subsection{Coding Reals by Triples of Ordinals}

We utilize the coding technique developed by A. Caicedo and B. Veli\v{c}kovi\'{c} (see \cite{CV}). Sets coded via their method serve as a suitable ground-model $H(\omega_2)$. Over this $H(\omega_2)$, we apply further coding utilizing independent Suslin trees; in essence, this amounts to coding on top of already coded sets. A key property of the Caicedo-Veli\v{c}kovi\'{c} coding is that ccc forcings cannot introduce new codes. Consequently, the ground-model $H(\omega_2)$ remains definable in ccc generic extensions. This allows us to identify objects from the ground-model $H(\omega_2)$—in our case, Suslin trees—and modify them generically by adding cofinal branches or specializing trees via ccc forcings. This ultimately yields a coding mechanism whose codes are easily decoded.

Although several proofs are omitted in the following discussion (details can be found in \cite{NS_saturated_and_definable}), it is necessary to describe the Caicedo-Veli\v{c}kovi\'{c} coding in detail, as its features are crucial for the present work.

\begin{definition}
A $\vec{C}$-sequence, also known as a ladder system, is a sequence $(C_{\alpha} : \alpha < \omega_1, \alpha \text{ is a limit ordinal})$ such that for every $\alpha$, $C_{\alpha}$ is a cofinal subset of $\alpha$ and has an order type of $\omega$.
\end{definition}

For the rest of this article, we will fix easily definable ladder systems, which are the only ones we will use throughout. Working over $M_n$, we will use the $\Delta^1_{n+2}$-definable wellorder $<_n$ of the
reals to define $\vec{C}= ( C_{\alpha} \mid \alpha < \omega_1 \cap Lim)$ inductively via picking at each stage $\alpha < \omega_1 \cap  Lim$ the $<_n$ least real coding a ladder $C_{\alpha} \subset \alpha$. A standard computations yields that this ladder system is $\Sigma_1 (\{ \omega_1 \} )$-definable over $M_n | \omega_2$. From now on whenever a ladder system is mentioned when working over some $M_n$, it is always implied that we talk about this particular ladder system $\vec{C}$.

Given three subsets of natural numbers, $x, y, z \subseteq \omega$, we can define an oscillation function. First, the set $x$ is transformed into an equivalence relation $\sim_x$ defined on $\omega \setminus x$. For natural numbers $n, m$ in the complement of $x$ where $n \leq m$, we define $n \sim_x m$ if and only if the interval $[n,m]$ has no elements in common with $x$ (i.e., $[n,m] \cap x = \emptyset$). This allows us to define the following:

\begin{definition}
For a triple of subsets of natural numbers $(x,y,z)$, let $(I_n : n \in k \leq \omega)$ be the sequence of equivalence classes of $\sim_x$ that have a non-empty intersection with both $y$ and $z$. The oscillation map $o(x,y,z): k \rightarrow 2$ is then defined as the function satisfying:

$$o(x,y,z)(n) = \begin{cases} 0 & \text{if } \min(I_n \cap y) \leq \min(I_n \cap z) \\ 1 & \text{otherwise} \end{cases}$$
\end{definition}

Next, we will define how appropriate countable subsets of ordinals can be used to code real numbers. For the remainder of this section, we fix a ladder system $\vec{C}$. Suppose $\omega_1 < \beta < \gamma < \delta$ are fixed limit ordinals, each with uncountable cofinality. Let $N \subseteq M$ be countable subsets of $\delta$. Furthermore, assume that $\{ \omega_1, \beta, \gamma\} \subseteq N$, and for each $\eta \in \{ \omega_1, \beta, \gamma\}$, $M \cap \eta$ is a limit ordinal and $N \cap \eta < M \cap \eta$.

We can use the pair $(N,M)$ to code a finite binary string. Let $\bar{M}$ be the transitive collapse of $M$, and let $\pi : M \rightarrow \bar{M}$ be the collapsing map. Define $\alpha_M := \pi(\omega_1)$, $\beta_M := \pi(\beta)$, $\gamma_M := \pi(\gamma)$, and $\delta_M := \bar{M}$. These are all countable limit ordinals. Also, set $\alpha_N := \sup(\pi[\omega_1 \cap N])$. The height $n(N,M)$ of $\alpha_N$ in $\alpha_M$ is the natural number defined by:

$$n(N,M) := |\alpha_N \cap C_{\alpha_M}|$$
where $C_{\alpha_M}$ is an element of our pre-fixed ladder system. Since $n(N,M)$ will appear frequently, we will abbreviate it as $n$. Note that because the order type of each $C_{\alpha}$ is $\omega$, and $N \cap \omega_1$ is bounded below $M \cap \omega_1$, $n(N,M)$ is indeed a natural number.

Now, we can assign a triple $(x,y,z)$ of finite subsets of natural numbers to the pair $(N,M)$ as follows:
$$x := \{ |\pi(\xi) \cap C_{\beta_M}| : \xi \in \beta \cap N \}$$
Note that $x$ is finite because $\beta \cap N$ is bounded in the set $C_{\beta_M}$ (which is cofinal in $\beta_M$ and has order type $\omega$). Similarly, we define:
$$y := \{ |\pi(\xi) \cap C_{\gamma_M}| : \xi \in \gamma \cap N \}$$
and
$$z := \{ |\pi(\xi) \cap C_{\delta_M}| : \xi \in \delta \cap N \}$$
Again, it is straightforward to see that these are finite subsets of natural numbers.

We can then consider the oscillation $o(x \setminus n, y \setminus n, z \setminus n)$ (recalling that $n = n(N,M)$). If the domain of the oscillation function at these points is greater than or equal to $n$, we write:
$$s_{\beta, \gamma, \delta} (N,M) := \begin{cases} o(x \setminus n, y \setminus n, z \setminus n) \restriction n & \text{if defined} \\ \ast & \text{otherwise} \end{cases}$$
Here, $\ast$ denotes an undefined state. Similarly, if $l > n$, we let $s_{\beta, \gamma, \delta} (N,M) \restriction l = \ast$.

Finally, we can define what it means for a triple of ordinals $(\beta, \gamma, \delta)$ to code a real number $r$.

\begin{definition}
For a triple of limit ordinals $\omega_1 < \beta < \gamma < \delta$, each with uncountable cofinality, we say that it codes a real $r \in 2^{\omega}$ if there exists a continuous, increasing sequence $(N_{\xi} : \xi < \omega_1)$ of countable sets of ordinals whose union is $\delta$. This sequence must also satisfy the condition that there is a club $C \subseteq \omega_1$ such that whenever $\xi \in C$ is a limit ordinal, there exists a $\nu < \xi$ for which:
$$r = \bigcup_{\nu < \eta < \xi} s_{\beta, \gamma, \delta} (N_{\eta}, N_{\xi})$$
We call the sequence $(N_{\xi} : \xi < \omega_1)$ a reflecting sequence.
\end{definition}

Witnesses for this coding can be added using a proper forcing. Conversely, for fixed triples of ordinals, there is a degree of control over the behavior of continuous, increasing sequences on them:

\begin{theorem}[Caicedo-Veli\v{c}kovi\'{c}]
\begin{itemize}
\item[$(\dagger)$] Given ordinals $\omega_1 < \beta < \gamma < \delta < \omega_2$, each with cofinality $\omega_1$, there exists a proper notion of forcing $\forceP_{\beta \gamma \delta}$ such that after forcing with it, the following holds: There is an increasing, continuous sequence $(N_{\xi} : \xi < \omega_1)$ where $N_{\xi} \in [\delta]^{\omega}$ and their union is $\delta$. This sequence is such that for every limit ordinal $\xi < \omega_1$ and every $n \in \omega$, there exist $\nu < \xi$ and $s_{\xi}^n \in 2^n$ satisfying $s_{\beta, \gamma, \delta}(N_{\eta}, N_{\xi}) \restriction n = s_{\xi}^{n}$ for every $\eta$ in the interval $(\nu, \xi)$. In this case, we say the triple $(\beta, \gamma, \delta)$ is stabilized.

\item[$(\ddagger)$] Furthermore, if we fix a real $r$, there is a proper notion of forcing $\forceP_r$. This forcing will produce, for a triple of ordinals $(\beta_r, \gamma_r, \delta_r)$ of size and cofinality $\omega_1$, a reflecting sequence $(P_{\xi} : \xi < \omega_1)$, where $P_{\xi} \in [\delta_r]^{\omega}$ and $\bigcup P_{\xi} = \delta_r$. Additionally, there is a club $C \subseteq \omega_1$ such that for every limit ordinal $\xi \in C$, there is a $\nu < \xi$ for which:
    $$\bigcup_{\nu < \eta < \xi} s_{\beta_r, \gamma_r, \delta_r} (P_{\eta}, P_{\xi}) = r.$$
\end{itemize}
\end{theorem}

Both partial orders $\forceP_{\beta \gamma \delta}$ and $\forceP_{r}$ which force $(\dagger)$ and $(\ddagger)$ respectively are actually instances of a general class of notions of forcing which were investigated first by J. Moore's in his work on the Set Mapping Reflection Principle ($\MRP$) (see \cite{Moore}). We shall see later that these forcings never kill Suslin trees. For that reason we have to introduce a couple of notions from \cite{Moore}.
We need first the following local version of stationarity: 
\begin{definition}
Let $\theta$ be a regular cardinal, $X$ be an uncountable set, let $M \prec H_{\theta}$ be a countable 
elementary submodel which contains $[X]^{\omega}$ as an element. Then $S \subset [X]^{\omega}$ is $M$-stationary if 
for every club subset $C$ of $[X]^{\omega}$, $C \in M$ it holds that $$C \cap S \cap M \ne \emptyset.$$
\end{definition}

\begin{definition}
Let $X$ be an uncountable set, $N \in [X]^{\omega}$ and $x \subset N$ finite. Then the Ellentuck 
topology on the set $[X]^{\omega}$ is generated by base sets of the form
$$ [x,N]:= \{ Y \in [X]^{\omega} \, : \, x \subset Y \subset N \}.$$ From now on whenever we say 
open we mean open with respect to the Ellentuck topology.
\end{definition}

\begin{definition}
Let $X$ be an uncountable set, let $\theta$ be a large enough regular cardinal so that
 $[X]^{\omega} \in H_{\theta}$. Then a function $\Sigma$ is said to be open stationary 
if and only if its domain is a club $C \subset [H_{\theta}]^{\omega}$ and for every 
countable $M \in C$, $\Sigma(M) \subset [X]^{\omega}$ is open and $M$-stationary.
\end{definition}

Moore has shown that for any open stationary map $\Sigma$ it is possible to force a  reflecting sequence $(N_{\xi} \, : \, \xi < \omega_1)$ with a proper forcing $\forceP_{\Sigma}$ (see \cite{Moore}, Theorem 3.1.).

\begin{proposition}[Moore]
Let $\Sigma$ be an open stationary function defined on some club $C \subset [H_{\theta}]^{\omega}$ with range $P([X]^{\omega})$ for some uncountable set $X$. Then there is a proper notion of forcing $\forceP_{\Sigma}$ which adds a continuous sequence of models $(N_{\xi} \, : \, \xi < \omega_1)$ (a reflecting sequence) in $dom(\Sigma)$ 
such that for every limit ordinal $\xi$ there is a $\nu < \xi$ such that for every $\eta$ with 
$\nu < \eta < \xi$, $N_{\eta} \cap X \in \Sigma(N_{\xi})$.
\end{proposition}
The forcing $\forceP_{\Sigma}$ is defined as expected: for an open stationary map $\Sigma$ let $\forceP_{\Sigma}$ consist of conditions $p$ which are  functions $p: \alpha +1 \rightarrow dom(\Sigma)$, $\alpha$ countable, 
which are continuous and $\in$-increasing, and which additionally satisfy the 
$\MRP$-condition on its limit points, namely that for every $0< \nu < \alpha$, $\nu$ a limit ordinal,  
there is a $\nu_0 < \nu$ such that $p(\xi) \cap X \in \Sigma(p(\nu))$ for every 
$\xi$ in the interval $(\nu_0, \nu)$. The order is by extension.

Now, as already mentioned above, both forcings $\forceP_{\beta \gamma \delta}$ and $\forceP_{r}$ which will produce $(\dagger)$ and $(\ddagger)$ respectively are of the form $\forceP_{\Sigma}$ (see \cite{CV}, Lemma 1, 4 and 5).

\begin{proposition}\label{MRPpreservesSuslin}
There are two open stationary maps $\Sigma_r$ and $\Sigma_{\beta \gamma \delta}$ such that $\forceP_{\beta \gamma \delta} = \forceP_{\Sigma_{\beta \gamma \delta}}$ and $\forceP_r = \forceP_{\Sigma_r}$.
\end{proposition}
As a consequence, if we show that $\forceP_{\Sigma}$ always preserves Suslin trees for $\Sigma$ an arbitrary open stationary map, we will have proven that $\forceP_r$ and $\forceP_{\beta \gamma \delta}$ preserve Suslin trees. This is indeed the case as we will show now. We need a Lemma of T. Miyamoto (\cite{Miyamoto}) first. 
\begin{lemma}\label{preservation of Suslin trees}
 Let $T$ be a Suslin tree and assume that $\forceP$ is a proper
 poset. Let $\theta$ be a sufficiently large cardinal.
 Then the following are equivalent:
 \begin{enumerate}
  \item $\Vdash_{\forceP} T$ is Suslin
 
  \item if $M \prec H_{\theta}$ is countable, $\eta = M \cap \omega_1$, and $\forceP$ and $T$ are in $M$,
  further if $p \in \forceP \cap M$, then there is a condition $q<p$ such that 
  for every condition $t \in T_{\eta}$, 
  $(q,t)$ is $(M, \forceP \times T)$-generic.
 \end{enumerate}

\end{lemma}

\begin{proposition}\label{PFAMRP}
Let $\Sigma$ be an arbitrary open stationary map and let $\forceP_{\Sigma}$ be as defined above. Then $\forceP_{\Sigma}$ preserves Suslin trees.
\end{proposition}

\begin{proof}
We fix an arbitrary Suslin tree $T$. Let $\lambda$ be a sufficiently large regular cardinal, and pick a countable elementary submodel $M \prec H_{\lambda}$ containing $T, \Sigma, \forceP_{\Sigma}$, a condition $p \in \forceP_{\Sigma}$, and the structure $H_{|\forceP_{\Sigma}|^+}$. Let $\eta = M \cap \omega_1$. Our goal is to produce a stronger condition $q \le p$ such that for every $t \in T_{\eta}$, the pair $(q,t)$ is an $(M, \forceP_{\Sigma} \times T)$-generic condition. 

Because $T$ is Suslin, every node $t \in T_{\eta}$ is $(M,T)$-generic. We can therefore form the generic extensions $M[t]$ for each $t \in T_{\eta}$. Since $T$ is ccc, $M[t]$ does not add any new countable sets of ordinals to $M$. Consequently, $M$ and $M[t]$ contain the exact same countable sequences of $M$-elements, meaning $\Sigma$ and $\forceP_{\Sigma}$ are interpreted identically in both $M$ and $M[t]$. 

Let $\theta$ be the regular cardinal such that $\operatorname{dom}(\Sigma) \subset [H_{\theta}]^{\omega}$. We enumerate the elements of $T_{\eta}$ as $\langle t_n : n \in \omega \rangle$. Since there are only countably many dense subsets of $\forceP_{\Sigma}$ in the union $\bigcup_{n \in \omega} M[t_n]$, we can list them as $\langle D_i : i \in \omega \rangle$. 

We now build by recursion a descending sequence of conditions $p = p_0 \ge p_1 \ge p_2 \ge \dots$ in $M$ such that $p_{i+1} \in D_i$. Assume that we have successfully constructed the sequence up to $p_i \in M$, and let $\zeta_i = \max(\operatorname{dom}(p_i))$. 

Consider the collection of countable elementary submodels $N' \prec H_{|\forceP_{\Sigma}|^+}$ that contain $H_{\theta}, D_i, \forceP_{\Sigma}$, and $p_i$. We define the club of countable structures 
$$C_i := \{ N' \cap X : N' \text{ as just described} \}.$$ 
Since all parameters defining $C_i$ belong to $M$, we have $C_i \in M$, which implies $M \cap H_{\theta} \in C_i$. Because $M \cap H_{\theta}$ is in the domain of $\Sigma$, the set $\Sigma(M \cap H_{\theta})$ is $M \cap H_{\theta}$-stationary and open. Therefore, there exists some $N_i = N'_i \cap X \in C_i \cap \Sigma(M \cap H_{\theta}) \cap M$. By the definition of the Ellentuck topology, the fact that $\Sigma(M \cap H_{\theta})$ is open guarantees the existence of a finite subset $x_i \subset N_i$ such that the basic open set $[x_i, N_i] \subset \Sigma(M \cap H_{\theta})$.

To form $p_{i+1}$, we first extend $p_i$ to an intermediate condition:
$$q_i := p_i \cup \{( \zeta_i + 1, \operatorname{hull}^{H_{\theta}} (p_i(\zeta_i) \cup x_i)) \}.$$ 
Because all of its defining parameters belong to $N'_i$, the condition $q_i$ is an element of $N'_i$. Since $D_i \in N'_i$ as well, we can extend $q_i$ to a condition $p_{i+1} \in N'_i \cap D_i$. Crucially, because we are working inside $N'_i$, any extension of $q_i$ will have the property that its range, when intersected with $X$, is contained in $N'_i \cap X = N_i$. Since $[x_i, N_i] \subset \Sigma(M \cap H_{\theta})$, this intersection remains entirely within $\Sigma(M \cap H_{\theta})$.

Finally, let $\delta = \sup_{i \in \omega} \zeta_i$. We define the limit condition:
$$q := \left( \bigcup_{i \in \omega} p_i \right) \cup \{(\delta, M \cap H_{\theta})\}.$$ 
This $q$ forms an $\in$-increasing, continuous function from $\delta + 1$ into $\operatorname{dom}(\Sigma)$. Furthermore, by our construction, for every $i \in \omega$, the reflection requirement $p_{i+1}(\xi) \cap X \in \Sigma(M \cap H_{\theta})$ holds. Thus, $q$ is a valid condition in $\forceP_{\Sigma}$. By construction, $q \le p$, and the pair $(q,t_n)$ is $(M, \forceP_{\Sigma} \times T)$-generic for every $n \in \omega$. Therefore, the forcing preserves the arbitrary Suslin tree $T$.
\end{proof}

\subsection{Almost Disjoint Coding}

The second coding method we utilize is almost disjoint coding, originally developed by R. B. Jensen and R. M. Solovay \cite{JensenSolovay}. Throughout this section, we identify subsets of $\omega$ with their characteristic functions, using the term ``reals'' interchangeably for both elements of $2^{\omega}$ and subsets of $\omega$.

Let $F=\{f_{\alpha} : \alpha < 2^{\aleph_0}\}$ be a family of almost disjoint subsets of $\omega$; that is, for any distinct $r, s \in F$, the intersection $r \cap s$ is finite. Given a set of ordinals $X \subset \kappa$ (where $\kappa \le 2^{\aleph_0}$), there exists a ccc forcing notion, denoted $\mathbb{A}_F(X)$, which generically adds a new real $x$. This real $x$ codes the set $X$ relative to the family $F$ via the following equivalence:
$$\alpha \in X \iff x \cap f_{\alpha} \text{ is finite.}$$

\begin{definition}
Let $F$ be an almost disjoint family and let $X \subset \kappa$. The almost disjoint coding poset $\mathbb{A}_F(X)$ consists of conditions of the form $(r, R)$, where $r \in \omega^{<\omega}$ is a finite sequence of natural numbers and $R \subset F$ is a finite subset of $F$. 
For two conditions $(s, S)$ and $(r, R)$ in $\mathbb{A}_F(X)$, we say that $(s, S)$ is an extension of $(r, R)$ (denoted $(s, S) \le (r, R)$) if and only if the following hold:
\begin{enumerate}
    \item $r$ is an initial segment of $s$ (i.e., $r \subseteq s$), and $R \subseteq S$.
    \item For every $f_{\alpha} \in R$, if $\alpha \in X$, then $s \cap f_{\alpha} = r \cap f_{\alpha}$.
\end{enumerate}
\end{definition}

Another variant of this forcing, presumably due to L. Harrington \cite{Harrington}, codes sets of reals relative to a newly added real. To define this, we first fix a definable bijection between the set of finite sequences of integers and $\omega$. For any real $b \in \omega^{\omega}$, let $\bar{b}(n)$ denote the natural number coding the initial segment $b \upharpoonright n$. We can associate to each real $b$ a corresponding set of integers $S(b) \subset \omega$, defined as the set of codes of all its finite initial segments:
$$S(b) := \{ \bar{b}(n) : n \in \omega \}.$$

\begin{definition}
Let $A \subset [\omega]^{\omega}$ be a set of infinite subsets of $\omega$. The almost disjoint coding forcing for $A$, denoted $\mathbb{A}(A)$, consists of conditions $p = (p(0), p(1))$ where $p(0)$ is a finite subset of $\omega$ and $p(1)$ is a finite subset of $A$. 
For two conditions $p, q \in \mathbb{A}(A)$, we say $q \le p$ ($q$ extends $p$) if and only if:
\begin{itemize}
    \item $p(0) \subseteq q(0)$ and $p(1) \subseteq q(1)$.
    \item For every $a \in p(1)$, $S(a) \cap q(0) \subseteq p(0)$.
\end{itemize}
\end{definition}

It is a standard fact that $\mathbb{A}(A)$ possesses the Knaster property, ensuring that arbitrary products of $\mathbb{A}(A)$ satisfy the ccc. If $A$ is a set of reals in the ground model $V$, forcing with $\mathbb{A}(A)$ yields a generic filter $G$. The union of the first coordinates of the conditions in $G$ forms a new real $a = \bigcup \{ p(0) : p \in G \}$. This real $a$ perfectly codes the set $A$ relative to the ground-model reals. Specifically, in the generic extension $V[G]$, membership in $A$ is characterized by:
$$x \in A \iff x \in V \land (S(x) \cap a \text{ is finite}).$$

This precise characterization plays a crucial role in our broader construction. As established earlier, our objective is to isolate a universe where the $H(\omega_2)$ of the ground model remains definable across arbitrary ccc generic extensions. By giving definable access to $H(\omega_2)^V$, we can use this forcing variant to condense complex structural information into a single generic real $a$, guaranteeing that this information can be  decoded in any subsequent ccc extension (see the upcoming Lemma \ref{sigma_1 set of suitable models}).

\subsection{Independent Suslin trees}

Suslin trees form the foundation of the third coding technique used in our proof. Recall that a set-theoretic tree $(T, <)$ is a Suslin tree if it is a normal tree of height $\omega_1$ with no uncountable antichains. Because every tree appearing in this paper is normal, we implicitly assume normality in all subsequent tree constructions.

For two trees $(T_0, <_{T_0})$ and $(T_1, <_{T_1})$, their tree product $T_0 \times T_1$ consists of the nodes $\{ (t_0, t_1) \in T_0 \times T_1 : \operatorname{height}(t_0) = \operatorname{height}(t_1) \}$, ordered by $(t_0, t_1) <_{T_0 \times T_1} (s_0, s_1)$ if and only if $t_0 <_{T_0} s_0$ and $t_1 <_{T_1} s_1$. Hereafter, whenever we refer to a product of trees, we specifically mean this tree product.

For our purposes, we must iteratively add sequences of Suslin trees $\langle T_\alpha : \alpha < \kappa \rangle$ such that any finite subproduct of these trees remains Suslin. We construct such sequences using Jech's forcing, which adds a Suslin tree via countable conditions.

\begin{definition}
 Let $\forceP_J$ be the forcing whose conditions are countable, normal trees ordered by end-extension. That is, $T_1 < T_2$ (meaning $T_1$ is stronger than $T_2$) if and only if:
 \[ \exists \alpha < \operatorname{height}(T_1) \left( T_2 = \{ t \upharpoonright \alpha : t \in T_1 \} \right). \]
\end{definition}

It is well known that $\forceP_J$ is $\sigma$-closed and generically adds a Suslin tree. Furthermore, the generically added tree $T$ has the property that for any ground-model Suslin tree $S$, the product $S \times T$ remains a Suslin tree in the generic extension.

\begin{lemma}
 Let $S \in V$ be a Suslin tree. If $g_T \subset \forceP_J$ is a generic filter over $V$ producing the tree $T$, then:
 $$V[g_T] \models S \times T \text{ is a Suslin tree.}$$
\end{lemma}

\begin{proof}
Let $\dot{T}$ be the standard $\forceP_J$-name for the generic Suslin tree, viewed here as a forcing notion. We claim that the two-step iteration $\forceP_J \ast \dot{T}$ contains a $\sigma$-closed dense subset. Because $\sigma$-closed forcings preserve ground-model Suslin trees, establishing this claim completes the proof. Consider the set:
$$ D = \{ (p, \check{q}) : p \in \forceP_J \land \operatorname{height}(p) = \alpha + 1 \land \check{q} \text{ is a node of } p \text{ on level } \alpha \}. $$
It is straightforward to verify that $D$ is both dense and $\sigma$-closed in $\forceP_J \ast \dot{T}$.
\end{proof}

A similar argument demonstrates that we can add an $\omega$-sequence of such Suslin trees via a full-support iteration. While longer sequences are possible by extending the iteration, $\omega$-blocks suffice for our construction.

\begin{lemma}\label{ManySuslinTrees}
 Let $S \in V$ be a Suslin tree, and let $\forceP$ be the full-support iteration of $\forceP_J$ of length $\omega$. In the generic extension $V[G]$, there exists an $\omega$-sequence of Suslin trees $\vec{T} = \langle T_n : n \in \omega \rangle$ such that for any finite subset $e \subset \omega$, the product $S \times \prod_{i \in e} T_i$ remains a Suslin tree in $V[G]$.
\end{lemma}

\begin{proof}
Let $G$ be a $\forceP$-generic filter over $V$. First, observe that the full-support product $\prod_{n \in \omega} \forceP_J$, followed by the full-support forcing with the product of the generically added trees $T_n$, contains a $\sigma$-closed dense subset (defined coordinate-wise exactly as in the previous lemma). Consequently, ground-model Suslin trees are preserved. 

For any finite $e \subset \omega$ and arbitrary Suslin tree $S \in V$, the product $S \times \prod_{i \in e} T_i$ is a Suslin tree in the intermediate generic extension generated by the trees $\langle T_i : i \in e \rangle$ over $V$. This Suslin property is preserved when passing to the full generic extension $V[G]$.
\end{proof}

Because these sequences of Suslin trees play a central role in our subsequent coding arguments, we formalize their defining property.

\begin{definition}
 Let $\vec{T} = \langle T_\alpha : \alpha < \kappa \rangle$ be a sequence of Suslin trees. We say that $\vec{T}$ is an \emph{independent family of Suslin trees} if, for every finite subset $e = \{e_0, e_1, \dots, e_n\} \subset \kappa$, the tree product $\prod_{\alpha \in e} T_\alpha$ is a Suslin tree.
\end{definition}

Finally, we will utilize the following preservation theorem, due to T. Miyamoto \cite{Miyamoto} and independently to U. Abraham and S. Shelah \cite{abraham1993delta22}.

\begin{theorem}\label{preservation of Suslin trees under countable support}
Let $\forceP = \langle \forceP_\beta, \dot{\forceQ}_\beta : \beta < \delta \rangle$ be a countable support iteration of proper forcings, and let $S$ be a Suslin tree. If, for every $\beta < \delta$, 
\[ \forceP_\beta \Vdash \text{``}\dot{\forceQ}_\beta \text{ preserves } S \text{ as a Suslin tree,''} \] 
then $S$ remains a Suslin tree in the generic extension by $\forceP$.
\end{theorem}
\section{Definition of $W_0$}

We finally have all the ingredients to successfully define the suitable ground model $W_0$ over which coding arguments will work in a nice way.
To form $W_0$ we start with $M_1$ as our ground model and let $\eta$ be an inaccessible limit of inaccessible cardinals in $M_1$. Using a $\eta$-long iteration we shall produce a generic extension $W_0$ of $M_1$ which forces with forcings of the form
\begin{enumerate}
\item $\forceP_{\beta \gamma \delta}$  for triples $\omega_1 < \beta < \gamma < \delta <\omega_2$ of uncountable cofinality.
\item $\forceP_r$ for reals $r$.
\item $\mathbb{A}_F (X)$ for $X \subset \omega_1$.
\item $\forceP_J$  which denotes Jech's forcing for adding a Suslin tree.

\end{enumerate}
We use countable support for the iteration which has the feature that ``Suslin tree preservation$"$ is preserved by theorem \ref{preservation of Suslin trees under countable support}. This preservation is key as we will use the added Suslin trees for a second iteration over $W_0$ which will use these Suslin trees for coding purposes. The goal is to arrive after $\eta$-many steps at a universe $W_0$ with the following properties
\begin{enumerate}
\item $\aleph_1^{M_1} = \aleph_1^{W_0}$.
\item The inaccessible limit of inaccessibles $\eta$ becomes $\aleph_2$ in $W_0$.

\item There is a $\eta=\omega_2$-length, independent sequence $\vec{T}$ of $\omega_1$-Suslin trees.
\item For every subset $X$ of $\omega_1$, there is a real $r_X$ which is an almost disjoint code of $X$ with respect to the canonical $M_1$-definable, almost disjoint family of reals $D \subset \omega^{\omega} \cap M_1$.
\item Every triple of limit ordinals $(\alpha, \beta, \gamma) < \omega_2$ of uncountable cofinality is stabilized in the sense of $(\dagger)$.
\item Every real is coded by a triple of limit ordinals of uncountable cofinality $(\alpha,\beta,\gamma)$.
\end{enumerate}

The actual iteration to obtain $W_0$ is standard and we are not interested in its specific properties as long as the model $W_0$ it produces has the 6 properties above. There is a canonical, definable well-order of $P(\omega_1)$ in $W_0$.
 
  \begin{definition}\label{Definition Wellorder}
  Let $X, Y \in P(\omega_1)^{W_0}$ then let
  $X \unlhd Y$ if the antilexicographically least triple of ordinals $(\alpha_0, \beta_0, \gamma_0)$
  which code a real $r_0$ which codes $X$ with the help of the a.d. family $F$ is antilexicographically less or equal than 
  the antilexicographically least triple of ordinals $(\alpha_1, \beta_1, \gamma_1)$ which codes a real $r_1$ which in turn
  codes $Y$.

  \end{definition}

The definable wellorder $\unlhd$ of $P(\omega_1)$ unlocks a definition for a
canonical sequence of length $\omega_2$ of independent Suslin trees. The first entry of 
that sequence is, for technical reasons which will 
become clear later defined differently.

Working over $M_1$ and using $\diamondsuit$, we define the canonical Suslin tree $T$ on $\omega_1$.
Specifically, let $\langle A_\alpha : \alpha < \omega_1 \rangle$ be a $\diamondsuit$-sequence. Then $T$ is constructed inductively as follows:

\begin{itemize}
    \item Initialize the tree with a single root node.
    \item  For every node $x \in T_\alpha$, assign it $\aleph_0$ many immediate successors in $T_{\alpha+1}$.
    \item  To form $T_\alpha$, we must select branches through the lower part of the tree, $T_{<\alpha}$, to extend. 
    \begin{itemize}
        \item Check if the guessed set $A_\alpha$ happens to be a maximal antichain in $T_{<\alpha}$.
        \item If $A_\alpha$ is a maximal antichain, ensure that every branch chosen to be extended into $T_\alpha$ intersects $A_\alpha$. Because $T_{<\alpha}$ is countable, it is always possible to find enough branches that hit $A_\alpha$ to continue the tree.
        \item If $A_\alpha$ is not a maximal antichain, extend any set of branches that keeps the level countable.
    \end{itemize}
\end{itemize}
It is known that the tree $T$ obtained this way is a full Suslin tree (see \cite{fuchs2009degrees}, pp. 431), which means that for every antichain $(a_n \in T \mid n \in \omega)$, the sequence of derived trees $(T_{a_n} \mid n \in \omega)$ is an independent sequence of Suslin trees. 
\begin{definition}
    Working in $M_1$, we let $T$ be the canonical Suslin tree as defined above and let $(a_n \in T \mid n \in\omega)$ be a list of the nodes of the first level of $T$. Let $T_{a_n}$ denote the subtree of $T$ with stem $a_n$ and set $T_{a_n} := T^0_n$ for every $n \in \omega$. We let \[\vec{T}^0= \{ T^0_n \mid n \in \omega \}=\{ T_{a_n} \mid n \in \omega\}\] be the resulting independent sequence of Suslin trees.
\end{definition}
We start with our fixed independent $\omega$-sequence $\vec{T}^0$ and let $\vec{T}^{\alpha}$ be the $\unlhd$-least $\omega_1$ sequence of Suslin trees such that $\bigcup_{\beta < \alpha} \vec{T}^{\beta}$ concatenated with $\vec{T}^{\alpha}$ remains an independent sequence of Suslin trees.

As the wellorder $\unlhd$ in fact talks about the reals which are almost disjoint codes for the corresponding elements of $P(\omega_1)$ it will be useful to give that sequence of reals a name as well. For every element $X$ in $P(\omega_1)$, the set of reals which are almost disjoint codes for $X$ is infinite. In the following we nevertheless talk about \emph{the} real $r_X$ which codes $X \in P(\omega_1)$ by which we mean the $\unlhd$-least such real coding $X$.

\begin{definition}
In $W_0$, 
let $(r_i \, : \, i < \omega_2)$ be the sequence of reals defined recursively
 as follows:
 
 \begin{itemize}
  \item $r_0$ is the real which codes a subset of $\omega_1$, which codes the independent $\omega$-sequence of Suslin trees $\vec{T}^0$.
 
  \item $r_{\alpha}$, for $\alpha > 0$ is the least real which is an almost disjoint code for the $\unlhd$-least subset of $\omega_1$ which itself is a code for an $\omega_1$-sequence of independent Suslin trees $\vec{T}^{\alpha}$,
  such that the concatenated sequence of the union of the Suslin trees coded in $(r_{i} \, : \, i < \alpha)$ and $\vec{T}^{\alpha}$
  forms an independent sequence again.
 \end{itemize}
 
\end{definition}

What is very important is that this definable $\omega_2$-sequence of
independent Suslin trees will be definable in certain outer models of $W_0$.

 \begin{lemma}\label{Definability of Suslin trees}
 Suppose that $W^{\ast}$ is a set-generic, ccc extension of $W_0$.
 Then $W^{\ast}$ is still able to define the $\omega_2$-sequence of independent
 $W_0$-Suslin trees $\vec{T}$.
\end{lemma}
\begin{proof}
 
 Note first that if $r \in W^{\ast}$ is a real coded by a triple of ordinals in $(\alpha, \beta, \gamma)$ in
 $W^{\ast}$, then there is a reflecting sequence $(N_{\xi} \, : \, \xi < \omega_1)$ in
 $W^{\ast}$, $\bigcup_{\xi < \omega_1} N_{\xi} = \gamma$, such that 
 for club-many $\xi$, $r = \bigcup_{\eta \in (\nu, \xi)} s_{\alpha \beta \gamma} (N_{\eta}, N_{\xi})$.
 As $W^{\ast}$ is a ccc-extension of $W_0$, there is a reflecting sequence $(P_{\xi} \, : \, \xi < \omega_1)$ which is 
 an element in $W_0$, and such that $C:=\{ \xi < \omega_1 \, : \, P_{\xi} = N_{\xi}\}$ is 
 club containing in $W^{\ast}$. Indeed, every element of $(N_{\xi} \, : \, \xi < \omega_1)$ is a countable set of ordinals in $W^{\ast}$, thus can be covered by a countable set of ordinals from $W_0$.  As a consequence the sequence $(N_{\xi} \, : \, \xi < \omega_1)$ can be transformed into a continuous, increasing sequence $(P_{\xi} \, : \, \xi < \omega_1)$ in $W_0$ which coincides on its limit point with $(N_{\xi} \, : \, \xi < \omega_1)$, just as desired.
 
 But as ccc extensions preserve stationarity, the set $$\{ \zeta < \omega_1 \, : \, \exists \nu < \zeta \, ( \bigcup_{\eta \in (\nu, \zeta))}
 s_{\alpha \beta \gamma} (P_{\eta}, P_{\zeta}) = r )\}$$ which is an element of $W_0$ must contain a club from 
 $W_0$. Hence $r$ is coded by the triple $(\alpha, \beta, \gamma)$ already in $W_0$.
 
 As a consequence $P(\omega_1)^{W_0}$ is definable in $W^{\ast}$, it will be precisely
 the set of subsets of $\omega_1$ which have reals which code it with the help of 
 the almost disjoint family $F$, and such that these reals are themselves
 coded by triples of ordinals below $\omega_2$ in the sense of $(\ddagger)$.
 
 Thus $W^{\ast}$ can define $W_0$-Suslin trees and our wellorder $<$ on $P(\omega_1)^{W_0}$, hence will be able to define 
 the $\omega_2$-sequence of independent Suslin trees of $W_0$.
 
\end{proof}

\section{Definition of $W_1$ and Suitability}
We begin to define thoroughly how the forcing over $W_0$ which will yield $W_1$, does look like. 
We already hinted that, in order to use the independent sequence of Suslin trees $\vec{T}$ we need a new notion for suitable models which
will be able to define the sequence of Suslin trees $\vec{T}$ correctly. 
With the notion of suitability it will become possible 
to correctly define $\vec{T}$ already in $\aleph_1$-sized $\ZFP$ models,
as we shall see soon, which is non-trivial, as being an $\omega_1$-Suslin tree in some $\aleph_1$-sized $\ZFP$ model not necessarily implies that this tree is actually Suslin in the real world.
\begin{definition}
 Let $M$ be a transitive model of $\ZFP$ of size $\aleph_1$. We say that $M$ is pre-suitable if it satisfies the 
 following list of properties:
 \begin{enumerate}
  \item $M_1 | \omega_1 \subset M$ so in particular our distinguished $\omega$-sequence of independent Suslin tree $\vec{T}^0$ and our distinguished ladder system $\vec{C}$ are both in $M$.
  \item $\aleph_1$ is the biggest cardinal in $M$ and $M \models \forall x (|x|\le \aleph_1)$.
  \item Every set in $M$ has a real in $M$ which codes it in the sense of almost disjoint coding
  relative to the fixed family of almost disjoint reals $F$.
  \item Every real in $M$ is coded by a triple of ordinals in $M$, i.e. 
  if $r \in M$ then there is a triple $(\alpha, \beta, \gamma) \in M$ and a reflecting
  sequence $(N_{\xi} \, : \, \xi < \omega_1)\in M$ which code $r$.
  \item Every triple of ordinals in $M$ is stabilized in $M$: for $(\alpha, \beta, \gamma)$
  there is a reflecting sequence $(P_{\xi} \, : \, \xi < \omega_1) \in M$ which witnesses that
  $(\alpha, \beta, \gamma)$ is stabilized.
 \end{enumerate}

\end{definition}

Note that the statement '``$M$ is a pre-suitable model$"$ is completely internal in $M$ and hence a $\Sigma_1(\vec{C}, \vec{T}^0)$ and in particular a $\Sigma_1 (\{\omega_1 \} )$-formula.
Further note that by the proof of Lemma \ref{Definability of Suslin trees}, if $W^{\ast}$ is a ccc extension of $W_0$ and $M$ is a pre-suitable
model in $W^{\ast}$ then $M \subset W_0$, as ccc extensions will not add new reflecting sequences.
\begin{definition}
 Let $M$ be a pre-suitable model. We say that $M$ is $W_0$-absolute for Susliness
 if $T \in M$ is an element from $W_0$ and $M \models T \text{ is Suslin}$, then $T$ is Suslin in $W_0$.
 Likewise we say that $M$ is $W_0$-absolute for stationarity if
 $S \in M$ and $M$ thinks that
 $S$ is a stationary subset of $\omega_1$ then $S$ is a stationary subset of $\omega_1$ in $W_0$. 
 A pre-suitable model which is $W_0$-absolute for stationarity and Susliness is called suitable.
\end{definition}

We have already seen in Lemma \ref{Definability of Suslin trees} that ccc extensions of $W_0$ will still be able to define our
$\omega_2$-sequence of independent $W_0$-Suslin trees $\vec{T}$. With the notion of suitability
we can localize this property in the following sense:
\begin{lemma}
 Let $W^{\ast}$ be a ccc extension of $W_0$, and let $M \in W^{\ast}$ be a suitable model.
 If $M$ computes the $\omega_2$-length sequence of independent Suslin trees from $W_0$ using its local wellorder $\unlhd_M$,
 then the computation will be correct, i.e. $\vec{T}^M = \vec{T} \cap M$.
\end{lemma}
\begin{proof}
 We shall show inductively that for every $\eta \in M$, the $\eta$-th
 block of $\vec{T}$, $\vec{T}^{\eta}$ will be computed correctly by $M$ (if an element of $M$).
 For $\eta=0$ this is true as $\vec{T}^0$ is by definition of suitability an element of $M$.
 Let $\eta \in M$ and assume by induction that the sequence $\vec{T}$ up to the $\eta$-th block is computed correctly by $M$. We shall show that $M$ computes the $\eta+1$-th block correctly. Recall that $\vec{T}^{\eta+1}$ was defined to be the $\unlhd$-least $\omega_1$-block of independent Suslin trees such that $\bigcup_{\beta \le\eta} \vec{T}^{\beta}$ concatenated with $\vec{T}^{\eta+1}$ remains an independent sequence in $W_0$. 
 
 Assume for a contradiction that the suitable $M$ computes $\vec{T}'$ as its own different version of $\vec{T}^{\eta+1}$. Thus there is a real
 $r'$ which codes $\vec{T}'$, and $r'$ itself is coded into a triple of $M$-ordinals $(\alpha', \beta', \gamma')< \omega_2^{M}$. Let $r_{\eta+1}$ be the real which codes $\vec{T}^{\eta+1}$.
 We claim that $r_{\eta+1} < r'$ by which we mean that the least triple of ordinals which codes $r_{\eta+1}$ is antilexicographically less than the least triple which codes $r'$. Otherwise $r'< r_{\eta+1}$ and by the Suslin-absoluteness of $M$
 the independent-$M$-Suslin trees coded into $r'$ would be an independent $\omega_1$-sequence of Suslin trees in $W_0$,
 moreover they would still form an independent sequence when concatenated with $\vec{T}^{\eta}$ in $W_0$ which
 is a contradiction to the way $\vec{T}^{\eta+1}$ was defined.
 
 So $r_{\eta+1}<r'$, so the least triple of ordinals $(\alpha, \beta, \gamma)$ coding $r_{\eta+1}$
 is antilexicographically less than $(\alpha', \beta', \gamma')$. Note that the suitability of $M$ implies
 that $(\alpha, \beta, \gamma)$ is stabilized in $M$. Thus there is a reflecting sequence
 $(P_{\xi} \, : \, \xi < \omega_1)$ in $M$ witnessing this. As $W^{\ast}$ is a ccc extension of $W_0$ we can assume
 that the sequence is in fact an element of $W_0$. At the same time there is a reflecting 
 sequence $(N_{\xi} \, : \, \xi < \omega_1)$ in $W_0$ which witnesses that $r_{\eta+1}$ is coded by $(\alpha, \beta, \gamma)$.
 By the continuity of both sequences, there is a club $C$ in $W_0$ such that
 $\forall \xi \in C ( N_{\xi} = P_{\xi}).$ Thus the limit points of $C$ witness that
 in fact the sequence $(P_{\xi} \, : \, \xi < \omega_1)\in M$ codes $r_{\eta+1}$ as well but the club
 $C$ is in $W_0$, so we need an additional argument to finish. Recall that the suitability of $M$
 implies that $M$ is absolute for stationarity, thus if the set 
 $\{ \xi < \omega_1 \, : \, \exists \nu < \xi ( \bigcup_{\zeta \in (\nu, \xi)} s_{\alpha \beta \gamma} (P_{\zeta}, P_{\xi}) \ne r_{\eta+1}) \}$
 would be stationary in $M$ it would be stationary in $W_0$ which is a contradiction.
 So $M$ computes $r_{\eta+1}$ correctly and the rest of the inductive argument can be repeated exactly as 
 above to show that $\vec{T}^M = \vec{T} \upharpoonright (M \cap Ord)$ as desired.
 
  \end{proof}

So suitable models will compute $\vec{T}$ correctly, and if we start to write information into $\vec{T}$, a suitable model can be used to read it off. In $W_0$ cofinally many suitable model below $\omega_2$ exist.

\begin{lemma}
Work in $W_0$. If we set
$$ P:= \{ \zeta < \omega_2 \, : \, \exists M (M \text{ is suitable } \land M \cap Ord = \zeta\}$$ then $P$ is unbounded in $\omega_2$.
\end{lemma}
\begin{proof}
Recall the iteration $(\forceP_{\alpha} \, : \, \alpha < \eta)$ to force $W_0$. Whenever we are at an intermediate stage $\kappa < \eta$ such that $\kappa$ is inaccessible then, if $G_{\kappa}$ denotes the generic filter for $\forceP_{\kappa}$, $H(\omega_2)^{M_1[G_{\kappa}]}$ will be a pre-suitable model.
Moreover, by the properties of countably supported iterations, Suslin trees in $H(\omega_2)^{M_1[G_{\kappa}]}$ will remain Suslin trees in $W_0$, as the tail iteration $\forceP_{[\kappa, \eta)}$ is a countable support iteration of proper, Suslin tree preserving notions of forcing over the ground model $M_1[G_{\kappa}]$. Thus every  $H(\omega_2)^{M_1[G_{\kappa}]}$, for $\kappa$ inaccessible, is a suitable model which gives the assertion of the Lemma.

\end{proof}

We shall use forcing to obtain a universe in which the just defined set of suitable models 
\begin{itemize}
\item[$U'$]
$:=\{ M \, : \, \exists \kappa < \eta (\kappa \text{ is inaccessible } \land M=H(\omega_2)^{M_1[G_{\kappa}]} \}$
\end{itemize} 
becomes easily definable. 
This will be our first forcing in order to obtain $W_1$ over $W_0$. We note that every $M \in U'$ is itself coded by a real relative to the almost disjoint family $F$. We want to use  the variant of almost disjoint coding forcing to code up  \begin{itemize}
\item[$U$] $:= \{ r_M \, : \, r_M $ is the least almost disjoint code for a subset of $\omega_1$ which codes $M \in U' \}$
\end{itemize}
into one real $r_{U}$. Recall that $\mathbb{A}(U)$ is a forcing of size $|U|=\eta$ which is Knaster and which adds a real $r_U$ such that in $W_0^{\mathbb{A}(U)}$ the following holds:
\begin{itemize}
\item[$(\ast)$] $\forall x \in 2^{\omega} \cap W_0 \,( x \in U \leftrightarrow r_U \cap S(x)$ is finite$)$.
\end{itemize}
In a next step we will code the characteristic function of the real $r_U$ into a pattern on $\vec{T}^0$. We use the fact that a Suslin tree can generically be destroyed in two mutually exclusive ways. We can generically add a branch or generically add an antichain without adding a branch.
We fix the first $\omega$-block of independent Suslin trees $\vec{T}^0$ and we let 
$\mathbb{D} (r_U) := \prod_{n \in \omega} \forceP_n$ with finite support, where
\begin{equation*}
\forceP_n = \begin{cases}
T^{0}_n  & \text{ if } r_U(n)=1 \\ Sp(T^{0}_n) & \text{ if } r_U(n)=0
\end{cases}
\end{equation*}
and $T^{0}_n$ denotes the forcing notion one obtains when forcing with nodes of $T^{0}_n$ as conditions, and
$Sp(T^{0}_n)$ denotes Baumgartner's forcing which specializes $T^{0}_n$ with finite conditions and which is known to be ccc (see \cite{Jech}).
Iterations of the just described form always have the countable chain condition.

\begin{lemma}\label{ccc Coding}
Let $\vec{T}= (T_{\alpha} \, : \, \alpha < \eta)$ be an independent sequence of Suslin trees of length $\eta$. Let $f: \eta \rightarrow 2$ be an arbitrary function and let $T_{\alpha}$ also denote the partial order when forcing with the tree $T_{\alpha}$ and let $Sp(T_{\alpha})$ be the forcing which specializes the tree $T_{\alpha}$.

Then if we consider the finitely supported product $\mathbb{D} (f) := \prod_{\beta < \eta} \forceP_{\beta}$ where \[
   \forceP_\beta=
   \begin{cases}
   T_\beta, &\text{if } f(\beta)=1,\\
   \operatorname{Sp}(T_\beta), &\text{if } f(\beta)=0.
   \end{cases}
\]
then $\mathbb{D}$ has the countable chain condition.
\end{lemma}
\begin{proof}
Fix an arbitrary $f: \eta \rightarrow 2$. We prove the Lemma using induction over the length $\eta$. The limit case is true as we use finite support. Thus assume the assertion of the Lemma is true for $\eta$ and we want to show it is true for products of length $\eta+1$. Assume for a contradiction that $\prod_{\alpha < \eta+1}\forceP_{\alpha}$ (according to $f$) does not have the countable chain condition.
Hence the tree $T_{\eta}$ is not a Suslin tree in $V^{\prod_{\alpha < \eta} \forceP_{\alpha}}$, as otherwise both forcings $T_{\eta}$ and $Sp(T_{\eta})$ would have the countable chain condition.
But $\prod_{\alpha < \eta} \forceP_{\alpha} \ast T_{\eta} = \prod_{\alpha < \eta} \forceP_{\alpha} \times T_{\eta} = T_{\eta} \times \prod_{\alpha < \eta} \forceP_{\alpha}$, and the latter is a forcing with the countable chain condition. Indeed as $T_{\eta}$ does not touch the Susliness of any member of $\vec{T}$ besides $T_{\eta}$,  $(T_{\alpha} \, : \, \alpha < \eta)$ is an independent sequence of Suslin trees in $V^{T_{\eta}}$, thus
by induction hypothesis,  $\prod_{\alpha < \eta} \forceP_{\alpha}$ has the countable chain condition in $V^{T_{\eta}}$, so it has the countable chain condition in $V$.
\end{proof}

In particular this means that $\mathbb{D} (r_U)$ as defined above is a forcing with the countable chain condition which writes the real $r_U$ into a pattern of 0 and 1's on the sequence $\vec{T}^0$ of Suslin trees. We
let $H_0$ be
a generic filter for $\mathbb{A}_U$ over $W_0$ and
let
 $H_1$ denote a $W_0[H_0]$-generic filter for $\mathbb{D} (r_U)$. The resulting model $W_1:=W_0[H_0][H_1]$ is the ground model for a second iteration we define later.
 
\begin{lemma}
 Let $W_1$ be the universe $W_0[H_0][H_1]$

 Then
 $W_1:= W_0[H_0][H_1]$ is a ccc extension of $W_0$ which satisfies:
 \begin{itemize}
 \item $ r_U(n)=1$ if and only if $T^{0}_n$ has a branch.
 \item $ r_U(n)=0$ if and only if $T^{0}_n$ is special.
\end{itemize}
\end{lemma}

\begin{proof}
It suffices to show that $W_0[H_0]$ is a Suslin-tree-preserving extension of $W_0$. This is clear as $\mathbb{A}(U)$ is Knaster. 

\end{proof}
To summarize we arrived at a situation where the real $r_U$, which captures all the information about a set of suitable models, is written into the sequence of Suslin trees $\vec{T}^0$. Consequently, any transitive, $\aleph_1$-sized model  $M$ which contains $\vec{C}$, $\vec{T}^0$ and which sees that every tree in $\vec{T}$ is destroyed can compute the real $r_U$ and thus has access to a set of suitable models, which in turn can be used to compute $\vec{T}$ inside $M$ in a correct way. This line of reasoning remains sound in all outer ccc extensions of $W_1$ as we shall see. Thus the ability of finding $\vec{T}$ in suitable models gives rise to the possibility of using $\vec{T}$ for additional coding arguments over $W_1$.

\begin{lemma}\label{sigma_1 set of suitable models}
 Let $W^{\ast}$ be a ccc extension of $W_1$ then there is a $\Sigma_1(\vec{C}, \vec{T}^0)$-formula $\Phi(v)$ such that whenever $x \in 2^{\omega}$ and $W^{\ast} \models \Phi(x)$ then $x$ is the almost disjoint code for a suitable model. 
\end{lemma}
\begin{proof}
The formula $\Phi(x)$, is defined as follows:

 \begin{itemize}
 \item[$\Phi(x)$]   if and only if there is a transitive, $\aleph_1$-sized $\ZFP$ model $N$ which contains $\vec{C}$ and $\vec{T}^0$ such that the following holds in $N$:  
   \begin{itemize}
    \item $N$ sees a full pattern on $\vec{T}^0$, i.e. for every $n \in \omega$ and 
   every member $T^0_{n}$, $N$ has either a branch through $T^0_{n}$ 
   or a function which specializes $T^0_{n}$. 
\item  The pattern on $\vec{T}^0$ corresponds to the characteristic function of a real $r$ and $N$ thinks that there is a pre-suitable $N'$ such that $x \in N'$ and $(S(x) \cap r)$ is a finite set.
 \end{itemize}
   \end{itemize}
This is a $\Sigma_1(\vec{C}, \vec{T}^0)$-formula, as it is of the form $\exists N( N\models ...)$.

We will show that whenever $x$ is a real in \(W^{\ast}\) satisfying \(\Phi(x)\), then \(x\) is the almost-disjoint code of some suitable model. Recall the sets \(U'\) and \(U\) defined earlier. First observe that \(N\) has access to \(U\), the collection of reals that code the suitable models \(U'\). This follows from the fact that \(\vec{T}^0 \in N\); hence any pattern on \(\vec{T}^0\) recognized by \(N\) must be the unique pattern coming from \(W_1\), which corresponds to the characteristic function of the real \(r_U\).

The assertion ``\(N'\) is a pre-suitable model$"$ is given by a \(\Delta_1(\vec{C})\) formula in \(N'\), and thus is absolute for the transitive model \(N\). Therefore, if \(N\) believes that there exists a pre-suitable \(N'\) containing \(x\), then this is indeed true in \(W^{\ast}\). Since \(W^{\ast}\) is a ccc extension of \(W_1\), and therefore also of \(W_0\), we already know that \(W^{\ast}\) can express the statement ``\(y \in 2^\omega \cap W_0\) $"$ as ``there exists a pre-suitable \(P\) with \(y \in P\) $"$. Consequently, if \(N\) thinks that some pre-suitable \(N'\) contains \(x\), then \(x \in W_0\). Applying \((\ast)\) from above, we conclude that \(x \in U\), which is exactly what we set out to prove.

\end{proof}

\section{Definition of $W_2$}
We shall use the just created $W_1$ to start a second iteration which will code a fresh sequence of Suslin trees $\vec{S}$ into the $W_1$-definable sequence $\vec{T}$ of Suslin trees. The definability of $\vec{T}$ will therefore be passed over to $\vec{S}$.
The reason for this seemingly redundant choice, is that we later want to force $\MA$, which of course will necessarily create a lot of noise on any definable sequence of independent Suslin trees. This noise is a threat to any intentional coding we aim to build in order to have a universe for the $\Sigma^1_n$-uniformization property.

The addition of the second sequence $\vec{S}$ of independent Suslin trees enables us to later use coding with $\vec{S}$ trees along the usual forcing which produces $\MA$ and yet be in full control of the codes we create on the $\vec{S}$-sequence. This effect cannot be reproduced by forcing over $W_1$ alone; this is why we need to pass from $W_1$ to a more suitable generic extension.

The universe $W_2=W_1^{\forceP_0 \ast \forceP_1}$ is a two step generic extension of $W_1$. The first factor $\forceP_0$ is just an $\omega_2$-length product of ordinary Cohen forcing with finite support
\[\forceP_0:= \prod_{\alpha < \omega_2} \mathbb{C}. \]
As is known due to S. Shelah (see \cite{Jech}, Theorem 28.12, Lemma 28.13 for a proof of this {which is due to S. Todor\v cevi\'c}) every generically added Cohen real $r$ can be used to define a new Suslin tree $S_r$. To be more precise, if $V$ is our ground model and
\[ (e_{\alpha} \, : \, \alpha < \omega_1) \] is an $\omega_1$ sequence of functions in $V$ which satisfies
\begin{enumerate}
\item for every $\alpha < \omega_1$, $e_{\alpha}$ is an injection from $\alpha$ to $\omega$;
\item and for every $\alpha < \beta < \omega_1$, $e_{\alpha} (\xi)=e_{\beta} (\xi)$ for all but finitely many $\xi < \alpha$,
\end{enumerate}
then
\[ S_r :=\{ r \circ (e_{\alpha} \upharpoonright \beta) \, : \, \alpha,\beta < \omega_1\} \]
 is a Suslin tree, as long as $r$ is a Cohen real over $V$.
 
 Consequently, if $(c_{\alpha} \, : \, \alpha < \omega_2)$ denotes the sequence of Cohen reals in $W_1^{\forceP_0}$, the sequence of trees 
 \[\vec{S}:=(S_{{\alpha}} \, : \, \alpha < \omega_2) \] {(note that here we write $S_{\alpha}$ where we actually mean $S_{c_{\alpha}}$)}
 is an independent sequence of Suslin trees. Indeed, any finite product of $S_{\alpha}$'s must be a Suslin tree again, as otherwise, if say $S_1$ and $S_2$ are such that $S_1 \times S_2$ is not Suslin, then $\mathbb{C} \times \mathbb{C} \times S_1 \times S_2$ is not a ccc forcing. But rearranging the factors yields that $S_2$ is a Suslin tree in $W_1^{\mathbb{C} \times S_1 \times \mathbb{C}}$, hence $\mathbb{C} \times S_1 \times \mathbb{C} \times S_2$ has the ccc, which is a contradiction.
 
 Our second forcing $\forceP_1$ uses our already created independent sequence $\vec{T}$ to code up the trees from $\vec{S}$ with a method completely {analogous to our forcing $\forceQ^1$ which} we used in the definition of $W_0$. That is if we fix an arbitrary $\alpha < \omega_2$ with the associated tree $S_{\alpha} \subset \omega_1$, 
and for $\beta < \omega_1$, we let
 \[ \forceP^{\beta}_{1,\alpha}= \begin{cases} T_{ \omega_1 \alpha + \beta} & \text{ if $S_{\alpha} (\beta)=1$ }
 \\
 \hbox{SP} (T_{\omega_1 \alpha+\beta}) & \text{ if $S_{\alpha} (\beta) =0$ }
 \end{cases} \]
where $T_{ \omega_1 \alpha + \beta}$ just denotes the forcing which adds an $\omega_1$-branch through $T_{ \omega_1 \alpha + \beta}$, and SP$(T_{ \omega_1 \alpha + \beta})$ is the forcing which specializes $T_{ \omega_1 \alpha + \beta}$.

Then
\[ \forceP_{1,\alpha} = \prod_{\beta < \omega_1} \forceP^{\beta}_{1,\alpha} \] using finite support. Note that for every $\alpha < \omega_2$, $\forceP_{1,\alpha}$ has the ccc.

Finally we let
\[ \forceP^1 := \prod_{\alpha < \omega_2} \forceP_{1,\alpha} \]
again using finite support. The upshot of these manipulations is, that $\vec{S}$ is now a uniformly definable $\omega_2$-sequence of independent Suslin trees which has a second crucial feature which is not shared by $\vec{T}$, and {which we briefly want to sketch: Let }\[ W_2= W_1^{\forceP_0 \ast \forceP_1}\]{ and suppose that 
$\forceQ$ is a partial order of size $\aleph_1$ in $W_2$. There will be an $\gamma < \omega_2$ such that $\forceQ$ is already in the intermediate model $W^{(\prod_{\alpha < \gamma} \mathbb{C} ) \ast (\prod_{\alpha < \gamma}  \forceP_{1,\alpha} )}$ which is an inner model of $W_2$}.
{If we decide to force with $\forceQ$ over $W_2$ then the resulting generic extension $W_2^{\forceQ}$ can be written as}
\[W^{(\prod_{\alpha < \gamma} \mathbb{C} ) \ast (\prod_{\alpha < \gamma}  \forceP_{1,\alpha}) \ast ( \forceQ ) \ast (\prod_{\alpha \in [\gamma, \omega_2) } \mathbb{C} ) \ast ( \prod_{\alpha \in [\gamma, \omega_2)}  \forceP_{1,\alpha} )}.\] {As a result the Suslin trees which $\prod_{\alpha \in [\gamma,\omega_2)} \mathbb{C}$ creates are still Suslin in
$W_2^{\forceQ}$. This feature will be very useful when we try to force $\MA$ and simultaneously a $\Delta^1_4$-definable wellorder of the reals and the $\Sigma^1_n$-uniformization property.}

\begin{lemma}\label{definabilityofvecT}
The universe $W_2$ is a cardinal preserving generic extension of $W_0$, whose continuum is $\aleph_2^W$. Moreover
there is a $\Sigma_1(\omega_1)$-formula $\Xi(v_0,v_1,\omega_1)$ {such that for every} $\gamma < \omega_2$,

\[ W_2 \models \forall \beta < \omega_1 (\Xi (\beta, \gamma,\omega_1) \Leftrightarrow \beta \in S_{\gamma}) \]
 
Thus initial segments  of the sequence $\vec{S} $ are uniformly  $\Sigma_1(\omega_1)$-definable over $W_2$.
\end{lemma}
\begin{proof}
The proof of the definability claim is almost identical to the proof of Lemma \ref{Definability of Suslin trees}, so we omit it. The rest is just a summary of the discussions above.
\end{proof}

\section{Coding forcings to alter a $\Sigma^1_4$-predicate}
We work now over $W_2$. Our goal is to define a $\Sigma^1_4$-predicate $\sigma(x)$ such that
\begin{enumerate}
\item $W_2 \models \lnot \exists x \sigma (x)$.
\item For an arbitrary real $x$, there is a forcing $\operatorname{Code} (x)$ which has size $\aleph_1$ and has the ccc and which has the effect that $W_2^{\operatorname{Code} (x)} \models \sigma (x)$, yet for every real $y \ne x$, 
$W_2^{\operatorname{Code} (x)} \models \lnot \sigma (y)$
\end{enumerate}

We shall define the coding forcing $\operatorname{Code} (x)$ along our desired $\Sigma^1_4$-formula $\sigma$. Assume that $x \in 2^{\omega}$ is an arbitrary real in $W_2$. 
In a first step we fix the first $\omega$-block $\vec{S}^1$ from $\vec{S}$ and define the forcing $\operatorname{Code} (x)$ relative to this first $\omega$-block of trees. This choice should decrease notation as we spare an additional ordinal in the definition of $\operatorname{Code}$. We add however that the coding forcing can be defined in exactly the same way for other $\omega$-blocks of Suslin trees from $\vec{S}.$ 

We use the same technique as in Lemma \ref{ccc Coding} to write $x$ into $\vec{S}^1$, i.e. we first let the first factor of $\operatorname{Code}(x)$ be the finitely supported product of factors defined as follows:
\begin{equation*}
\forceP_{n } = \begin{cases}
S^1_{n}  & \text{ if } x(n)=1 \\ Sp(S^1_{n}) & \text{ if } x(n)=0
\end{cases}
\end{equation*}

To form the second forcing we need to find a ``nice$"$ subset of $\omega_1$ which codes all the branches and specializing functions along with some suitable models.
Let $A$ be a suitable model which computes $\vec{T}^1$ correctly.
We collect the set $\mathcal{J}^{M_1}_{\omega_1}$, the relevant branches and specializing functions through elements of $\vec{T}^0$ which create the pattern which codes up $r_U$, the relevant reflecting sequences to define the suitable $A$ and write everything into one set $X \subset \omega_1$. 
Note that if $L_{\zeta}[X]$ is a $\ZFP$-model which contains $X \subset \omega_1$, we obtain that
\begin{align*}
L_{\zeta} [X] \models &n \in x \rightarrow S^1_{n} \text{ has an $\omega_1$-branch and } \\& 
n \notin x \rightarrow S^1_{n}  \text{ has a specializing function}
\end{align*}
where we define $\vec{T}^1$ and $\vec{S}^1$ inside $A \in L_{\zeta} [X]$ and $A$'s suitability is confirmed via $M_1 | \omega_1$, the $M_1 | \omega_1$-definable ladder system $\vec{C}$ and the $M_1 | \omega_1$-definable trees $\vec{T}^0$.

We first note that any transitive, $\aleph_1$-sized $\ZFP$ model $M$ which contains $X$ will satisfy
\begin{align*}
(M, \in,  \mathcal{J}^{M_1}_{\omega_1}) \models &``\text{Decoding X yields several objects namely} \\& \text{ a model $m$ and $m=\mathcal{J}^{ M_1}_{\omega_1}$}=\bigcup_{\mathcal{J}^{M_1}_{\eta} \in \mathcal{I}} \mathcal{J}^{M_1}_{\eta},  \\ & \text{ a suitable model $a$ which in turn computes trees $\vec{s^1}$ } \\& \text{some branches $\vec{b}$ and specializing functions $\vec{f}$ through $\vec{s^1}$},\\& \text{ such that for any $\eta< \omega_2$ such that $L_{\eta}[a][\vec{b}] [\vec{f}] \models \ZFP$} 	 \\&
			L_{\eta}[a][\vec{b}] [\vec{f}] \models  \forall n  ( n \in x \rightarrow S^1_{n}  \text{ has an $\omega_1$-branch}  \\&
		\qquad \qquad  	n \notin x \rightarrow S^1_{n}  \text{ has a specializing function})
			\end{align*}
In particular, this will be true for a $\ZFP +`` \aleph_1$ exists$"$ model of the form $(L_{\xi}[X],\in, \mathcal{J}^{M_1}_{\omega_1})$, $\xi < \aleph_2$ as we can always assume that the decoding function is absolute. 

We consider the club of countable structures \[ C:= \{ \eta < \omega_1 \, : \, \exists (M, \in, P) \prec (L_{\xi}[X],\in,\mathcal{J}^{M_1}_{\omega_1})( |M|=\aleph_0 \land \eta= \omega_1 \cap M) \}\]
If we assume that $(M,\in,P) \in C$, $M \cap \omega_1 = \eta$, then the model $m\in  M$ which satisfies the formula $m=\mathcal{J}^{M_1}_{\omega_1}$ locally in $M$, will also satisfy $m \in \mathcal{I}$ as seen from the outside, as it is the limit of reals which are elements of $\mathcal{I}$ and taking unions is absolute. So $m= \mathcal{J}^{M_1}_{\eta} \in \mathcal{I}$.

Consequently if $(N, \in)$ is an arbitrary countable transitive model of $\ZFC^-$ and ``$\aleph_1$ exists'', such that
$X\cap \omega_1^N \in N$ and $\omega_1^N \in C$ then also $\mathcal{J}^{M_1}_{\eta} \in N$ and $N$ will decode out of $X \cap \omega_1^N$ with the help of the ladder system and the almost disjoint family computed from $\mathcal{J}^{M_1}_{\eta}$ exactly what $(\bar{M},\in,\mathcal{J}^{M_1}_{\eta})$ decodes out of $X \cap \omega_1^N$, where the latter is the transitive collapse of  $(M,\in, P) \prec (L_{\xi}[X],\in, \mathcal{J}^{M_1}_{\omega_1})$. In particular, if we denote the $\Delta_1$-definable decoding functions with $dec_1,dec_2$, $dec_3$ and $dec_4$ respectively, then we obtain 
\begin{align*}
N \models \exists m_1 \, \exists \vec{c} \, \exists \vec{b}, \vec{f} (& dec_1(X\cap \omega_1^N)=m_1 \land dec_2(X \cap \omega_1^N)= a \\& \land dec_3(X \cap \omega_1^N)=\vec{b} \land dec_4(X\cap \omega_1^N)=\vec{f} \\& \text{ and for any $\ZFP$ model of the form
$L_{\eta}[X \cap \omega_1^N]$ we have that} \\&
			L_{\eta} [X \cap \omega_1^N] \models  \forall n ( n \in x\rightarrow S^1_{n}  \text{ has an $\omega_1$-branch}  \\&
		 \qquad \qquad \qquad  \quad \quad  	n \notin x\rightarrow S^1_{n}  \text{ is special})) ),
        \end{align*}
        where again the trees $S^1_n$ are computed inside $a$ with the help of the formula which defines $\vec{S}^1$ in $W_2$.
Further, as $dec_1(X \cap \omega_1^N)=m_1=\mathcal{J}^{M_1}_{\eta}$, we get that (from the outside, not from $N$'s perspective)
\[ dec_1(X \cap \omega_1^N) \in \mathcal{I}.\]

Now let the set $Y\subset \omega_1$ code the pair $(C, X)$ such that the odd entries of $Y$ should code $X$ and if $Y_0:=E(Y)$ where the latter is the set of even entries of $Y$ and $\{c_{\alpha} \, : \, \alpha < \omega_1\}$ is the enumeration of $C$ then
\begin{enumerate}
\item $E(Y) \cap \omega$ codes a well-ordering of type $c_0$.
\item $E(Y) \cap [\omega, c_0) = \emptyset$.
\item For all $\beta$, $E(Y) \cap [c_{\beta}, c_{\beta} + \omega)$ codes a well-ordering of type $c_{\beta+1}$.
\item For all $\beta$, $E(Y) \cap [c_{\beta}+\omega, c_{\beta+1})= \emptyset$.
\end{enumerate}
The set $Y\subset \omega_1$ codes $X$ but has the following additional property:

\begin{itemize}
\item[] Let $M$ be an arbitrary countable transitive model of $\ZFP + `` \aleph_1$ exists$"$ for which 
there is a $\mathcal{J}^{M_1}_{\eta} \in \mathcal{I}$ (again from the outside, not from $M$'s perspective) such that $\omega_1^M=\omega_1^{\mathcal{J}^{M_1}_{\eta}}$ and $\mathcal{J}^{M_1}_{\eta} \in M$. Assume that $Y \cap \omega_1^M \in M$ then $M$ can decode out of  $Y \cap \omega_1$ with the help of the almost disjoint family of $\mathcal{J}^{M_1}_{\omega_1}$ the following objects:
\begin{itemize}
\item a transitive model $m$ which $M$ believes to be of size $\aleph_1$,
\item a model $a$ which, $M$ believes, is suitable relative to the ladder system and the almost disjoint family of reals defined in $\mathcal{J}^{M_1}_{\eta}$ and which computes $\vec{s}^1$ using the same formula which computes the trees in $W_2$,
\item a set of branches $\vec{b}$ through $\vec{s}^1$.
\item and specializing functions $\vec{f}$ through elements of $\vec{s}^1$ such that for any $\ZFP+ ``\aleph_1$ exists$"$-model of the form $L_{\zeta} [a,\vec{b},\vec{f}]$:
\begin{align*}
L_{\zeta} [a,\vec{b},\vec{f} \models &\forall n ( n \in x \rightarrow s^1_{n}  \text{ has an $\omega_1$-branch}  \\&
		 \qquad   	n \notin x \rightarrow s^1_{n}  \text{ is special})) ).
\end{align*}

\end{itemize}
Moreover $m$ is an $M_1$ initial segment as seen from the outside, i.e. $m = \mathcal{J}^{M_1}_{\eta} \in \mathcal{I}$.
\end{itemize}
In the last step, we use almost disjoint coding forcing relative to the $\mathcal{J}^{M_1}_{\omega_1}$-definable almost disjoint family of reals to obtain a real $r_Y$ which codes our set $Y \subset \omega_1$. Thus we obtain that the following formula $\psi(x,r_Y)$ holds, where $\psi(x,r_Y)$ is defined as:
\begin{itemize}
\item[] Assume $M$ is a countable transitive model of $\ZFP+ ``\aleph_1$ exists$"$, and $r_Y \in M$. Assume that for $M$ 
there is a $\mathcal{J}^{M_1}_{\eta} \in \mathcal{I}$ such that $\omega_1^M=\omega_1^{\mathcal{J}^{M_1}_{\eta}}$ and $\mathcal{J}^{M_1}_{\eta} \in M$.  Assume further that $r_Y \in M$ then $M$, relative to the almost disjoint family of reals from $\mathcal{J}^{M_1}_{\eta}$, can decode out of  $r_Y$ the following
\begin{itemize}
\item a transitive model $m$ of $\ZFC^-$ and ``$\aleph_1$ exists$"$,
\item a model $a$ which, $M$ believes, is suitable, and which computes $\vec{s}^1$,
\item a set of branches $\vec{b}$ through $\vec{s}^1$.
\item and specializing functions $\vec{f}$ through elements of $\vec{s}^1$ such that for any $\ZFP+ ``\aleph_1$ exists$"$-model of the form $L_{\zeta} [a,\vec{b},\vec{f}]$ (computed in $m$):
\begin{align*}
L_{\zeta} [a,\vec{b},\vec{f} \models &\forall n ( n \in x \rightarrow s^1_{n}  \text{ has an $\omega_1$-branch}  \\&
		 \qquad   	n \notin x \rightarrow s^1_{n}  \text{ is special})) ).
\end{align*}

\end{itemize}
Moreover $m$ is an $M_1$ initial segment as seen from the outside, i.e. $m = \mathcal{J}^{M_1}_{\eta} \in \mathcal{I}$.
\end{itemize}
A straightforward calculation shows that the statement $\psi(x,r_Y)$ is of the form $ (\Sigma^1_3 \rightarrow \Pi^1_3)$, thus it is a $\Pi^1_3$-formula, and stating the existence of such a real $r_Y$ is $\Sigma^1_4$ and results in our desired $\Sigma^1_4$-formula $\sigma$:
\[ \sigma(x) \equiv \exists r \psi (x,r). \]

The existence of a real $r$ witnessing $\psi(x,r)$ is sufficient to conclude that $L[r]$ contains branches through $\aleph_1$-many trees from $\vec{S}$.
\begin{lemma}\label{definable well-order of reals determines real world}
Let $r$ be such that $\psi(x,r)$ is true.
Then, working inside $L[r]$, there is a suitable $A$ such that $A$ computes $\vec{S}^1$ and there are branches $\vec{b}$ and specializing functions $\vec{f}$ such that
\begin{align*}
\forall n (&n \in x \rightarrow S^1_{n} \text{ has an $\omega_1$-branch} \\& 
n \notin x \rightarrow S^1_n \text{ is special} )
\end{align*}
\end{lemma}
\begin{proof}
We note first that $\psi(x,r)$ must also be true (ignoring its statements involving $\mathcal{I}$) for models of uncountable size where we replace $\mathcal{J}^{M_1}_{\eta}$, in $\psi$, with $\mathcal{J}^{M_1}_{\omega_1}$. Indeed, if $(M, \in , \mathcal{J}^{M_1}_{\omega_1} )$ would be an uncountable, transitive model containing $r$ and $\mathcal{J}^{M_1}_{\omega_1}$ for which $\psi(x,r)$ is false, then we let $(\bar{N}, \in \mathcal{J}^{M_1}_{\eta} )$ be the transitive collapse of $N \prec M$, $r, \mathcal{J}^{M_1}_{\eta} \in \bar{N}$ and $\bar{N}$ would reject $\psi(x,r)$ as well, even though $\bar{N}$ is of the right form as externally $\mathcal{J}^{M_1}_{\eta} \in \mathcal{I}$, which gives us a contradiction.

But if $\psi(x,r)$ holds for arbitrarily large models $M$, it must be true in the universe $L[r]$. Indeed if some $\aleph_1$-sized  $\ZFP$-model of the form $L_{\zeta} [M,A, \vec{B}, \vec{F}]$, where $M,A, \vec{B}, \vec{F}$ are just the unions of the computations of $m,a, \vec{b}$ and $\vec{f}$ in suitable countable transitive models of increasing (with limit $\omega_1$) ordinal height, then first note that $M=M_1 | \omega_1$ and $L_{\zeta} [M,A, \vec{B}, \vec{F}]$ sees that there are branches $\vec{B}$ and specializing functions $\vec{F}$ such that 
\begin{align*}
 L_{\zeta} [M,A, \vec{B}, \vec{F}] \models  &\forall  n (n  \in x  \rightarrow S^1_{n}  \text{ has an $\omega_1$-branch}  \land \\&
		 \quad  	n \notin x \rightarrow S^1_{n}  \text{ is special})) .
\end{align*}
and $L_{\zeta} [M,A, \vec{B}, \vec{F}] $'s computation of $\vec{S^1}$ must be correct. As the existence of an $\omega_1$-branch through $S^1_{n}$, and the existence of a specializing function for $S^1_n$ is upwards absolute we obtain the assertion of the lemma.
\end{proof}

As an immediate consequence of the preceding lemma, the truth of the $\Sigma^1_4$-formula $\sigma$ has absolute consequences for the ambient universe: whenever $\sigma(x)$ holds, the characteristic function of $x$ is genuinely written into $\vec{S}^1$. This property ensures that our iteration does not inadvertently introduce spurious codes; that is, no unintended real $x$ will satisfy $\sigma(x)$.

Furthermore, we state again that while we initially defined $\operatorname{Code}(x)$ with respect to $\vec{S}^1$, this choice was arbitrary. We can generalize the definition of $\operatorname{Code}(x)$ to utilize $\vec{S}^{\eta}$ for any $\eta < \omega_2$.

\begin{definition}
Let $\eta < \omega_2$ and let $x$ be a real. The coding forcing at level $\eta$ is defined as:
\[ \operatorname{Code}(x,\eta) := \left( \prod_{n \in x} S^{\eta}_{n} \times \prod_{n \notin x} Sp(S^{\eta}_n) \right) \ast \mathbb{A}(Y) \]
In this generalized form, $\operatorname{Code}(x,\eta)$ operates exactly as the previously defined coding forcing, but it utilizes the $\eta$-th $\omega$-block of $\vec{S}$ instead of $\vec{S}^1$.
\end{definition}

Moving forward, we adopt the convention that whenever a coding forcing $\operatorname{Code}(x,\eta)$ is used as a factor in an iteration, $\eta$ is always chosen to be the least ordinal such that no forcing of the form $\operatorname{Code}(y,\eta)$ has appeared in any prior stage. In other words, the iteration uses the $\omega$-blocks of $\vec{S}$ in consecutive order, leaving no unused gaps.

\section{Appropriate $\Sigma^1_n$-predicates}\label{sec:appropriate-sigma-predicates}

We now fix the higher coding predicates used in the final iteration.  Let
\[
\operatorname{Coded}(z)
\]
abbreviate the $\Sigma^1_4$ assertion that the real $z$ is coded into the sequence
$\vec S$ in the sense of the preceding section.  We shall also write
``$z$ is not coded'' for $\neg\operatorname{Coded}(z)$.  The point is that the sign of the final coding statement alternates with the parity of the desired projective level.

For $k\geq 4$ put $\ell=k-4$.  If $z$ is a real, define $\Phi^k(z)$ by
\[
\Phi^k(z) \quad\Longleftrightarrow\quad
\exists a_0\, Q_1 a_1\, Q_2 a_2\cdots Q_\ell a_\ell\;
\Theta_\ell(z,a_0,\ldots,a_\ell),
\]
where $Q_i$ is $\forall$ for odd $i$ and $\exists$ for even $i$, and where
\[
\Theta_\ell(z,a_0,\ldots,a_\ell)=
\begin{cases}
\operatorname{Coded}(z,a_0,\ldots,a_\ell), & \text{if $\ell$ is even},\\
\neg\operatorname{Coded}(z,a_0,\ldots,a_\ell), & \text{if $\ell$ is odd}.
\end{cases}
\]
Here, as throughout, the expression $(z,a_0,\ldots,a_\ell)$ denotes the fixed recursive real coding the tuple.  In applications $z$ will itself often be a code for a tuple, for instance $z=(x,y,m)$.

Thus the first few cases are
\begin{align*}
\Phi^4(x,y,m) &\Longleftrightarrow \exists a_0\; \operatorname{Coded}(x,y,m,a_0),\\
\Phi^5(x,y,m) &\Longleftrightarrow \exists a_0\forall a_1\; \neg\operatorname{Coded}(x,y,m,a_0,a_1),\\
\Phi^6(x,y,m) &\Longleftrightarrow \exists a_0\forall a_1\exists a_2\; \operatorname{Coded}(x,y,m,a_0,a_1,a_2),\\
\Phi^7(x,y,m) &\Longleftrightarrow \exists a_0\forall a_1\exists a_2\forall a_3\;
\neg\operatorname{Coded}(x,y,m,a_0,a_1,a_2,a_3).
\end{align*}
Since $\operatorname{Coded}$ is $\Sigma^1_4$, the displayed definition gives exactly a $\Sigma^1_k$ predicate.  The cases $k=2,3$ are dealt with by the earlier $\Sigma^1_2$ and $\Sigma^1_3$ uniformization arguments; the present predicates are the ones needed from level $4$ upward.

The following observation is the only feature of these predicates used later.

\begin{lemma}\label{lemma:higher-coding-predicates}
Let $k\geq 4$, and let $z$ be a real in $W_2$.  There is a ccc iteration of the coding forcings, denoted by $\operatorname{Code}^k(z)$, such that whenever $G\subseteq\operatorname{Code}^k(z)$ is generic over $W_2$,
\[
W_2[G]\models \Phi^k(z),
\]
and for every real $z'\neq z$ from the relevant ground model,
\[
W_2[G]\models \neg\Phi^k(z').
\]
Moreover the same assertion remains true over any intermediate extension obtained by the construction, provided the next unused $\omega$-blocks of $\vec S$ are used.
\end{lemma}

\begin{proof}
This is the usual induction on $k$.  At the terminal stage one either writes the required tuple into the next unused block of $\vec S$, or deliberately leaves it unwritten, depending on the parity specified above.  The preceding no-accidental-coding argument shows that the only tuples which become coded are the tuples intentionally written.  At successor steps we iterate the corresponding coding forcings, using fresh blocks of $\vec S$ at each stage.  The forcings are ccc and of size $\aleph_1$, and the freshness convention ensures that no later step changes the truth value of an already decided terminal coding statement.  This gives the displayed equivalence for $z$ and prevents any $z'\neq z$ from satisfying the same predicate accidentally.
\end{proof}

\section{Forcing $\mathsf{MA}$ and global $\Sigma$-uniformization and a definable wellorder}\label{sec:final-iteration}

We now perform the final forcing over $W_2$.  The forcing is a finite-support iteration of length $\omega_2$ whose nontrivial stages have three roles: they add the diagonal ccc forcing needed for $\mathsf{MA}$, they code the comparisons which yield the definable wellorder of the reals, and they carry out the global $\Sigma$-uniformization construction.  The diagonal form of the $\mathsf{MA}$ stages is used only to preserve the unused tail of the sequence $\vec S$.

Before starting the final iteration we make one harmless normalization which will be used only for the compression of long lists of reals into single reals.  Fix, in a part of the construction disjoint from the Suslin-tree coding apparatus, an almost disjoint family
\[
   \mathcal H=\langle h_\xi:\xi<\omega_1\rangle
\]
of reals.  If $A\subseteq\omega_1$, let $\mathbb A_{\mathcal H}(A)$ be the usual almost disjoint coding forcing which adds a real $r_A$ such that, for all $\xi<\omega_1$,
\[
   \xi\in A \quad\Longleftrightarrow\quad |r_A\cap h_\xi|=\omega .
\]
Equivalently, one may use the dual convention with finite intersections for membership; the choice is immaterial.  We use a finite-support bookkeeping iteration of such forcings, relative to $\mathcal H$, to code all subsets of $\omega_1$ which occur in the relevant intermediate extensions into reals.  This auxiliary forcing is completely separate from the forcings which read or destroy the trees in $\vec S$.

\begin{lemma}\label{lemma:H-normalization}
We may assume, before the final iteration starts, that
\[
   2^{\aleph_0}=2^{\aleph_1}=\aleph_2.
\]
Moreover this preliminary normalization preserves all trees in the sequence $\vec S$ and all finite products of unused blocks of $\vec S$.
\end{lemma}

\begin{proof}
The forcing $\mathbb A_{\mathcal H}(A)$ is $\sigma$-centered: conditions with the same finite binary stem are compatible, since the finite restraint sets may simply be united.  Hence it is ccc and, more importantly for us, it preserves every Suslin tree.  Indeed, if $T$ is Suslin and $\mathbb Q$ is $\sigma$-centered, then $\mathbb Q\times T$ is ccc: on an uncountable set of conditions in the product, thin to one centered piece of $\mathbb Q$; incompatibility in the product would then have to come from the $T$-coordinates, contradicting the ccc of $T$.

The same argument applies with $T$ replaced by any finite tree product built from unused members of $\vec S$.  We now run the $\mathcal H$-almost-disjoint coding iteration by finite support, and prove by induction on the stage that each finite product of still unused trees from $\vec S$ remains Suslin.  At a successor step the current iterand is a $\sigma$-centered almost-disjoint coding forcing, and hence preserves the relevant finite products by the preceding paragraph.  At a limit stage the standard finite-support preservation argument applies: if an antichain in such a finite product appeared after the limit iteration, the finite supports of the names and the ccc of the initial segments would reflect an uncountable antichain to some earlier stage.  This contradicts the induction hypothesis.  Thus the auxiliary normalization is Suslin-tree preserving, and in particular it preserves the independence of all unused $\vec S$-blocks.

The bookkeeping for the $\mathcal H$-coding iteration codes every subset of $\omega_1$ appearing in the resulting extension by a real.  Hence in the normalized model there is an injection $\mathcal P(\omega_1)\to 2^\omega$.  The reverse injection is trivial, so $2^{\aleph_1}=2^{\aleph_0}$.  The iteration has size $\omega_2$, adds $\omega_2$ many distinct reals, and, starting from the background GCH cardinal arithmetic of the mouse construction, adds no more than $\omega_2$ reals.  Thus both cardinals are equal to $\omega_2$.

The auxiliary codes introduced through $\mathcal H$ are irrelevant for the later coding predicate.  Throughout the paper, $\operatorname{Coded}$ refers only to the pattern written into the dedicated Suslin-tree blocks of $\vec S$, not to the almost disjoint $\mathcal H$-codes.  Thus the normalization may add many harmless almost-disjoint codes, but it does not create an unintended $\vec S$-code.
\end{proof}

After this normalization we keep the notation $W_2$ for the resulting model.

\begin{lemma}\label{lemma:cardinal-equations-at-stages}
Let $\bigl(\forceP_\beta,\dot{\forceQ}_\beta:\beta<\omega_2\bigr)$ be the final ccc iteration defined below.  Then for every $\beta\leq\omega_2$,
\[
   W_2[G_\beta]\models 2^{\aleph_0}=2^{\aleph_1}=\aleph_2.
\]
In particular the equations hold at every stage at which a uniformization request is evaluated.
\end{lemma}

\begin{proof}
For $\beta<\omega_2$, the forcing $\forceP_\beta$ has cardinality at most $\aleph_1$.  Since the ground model satisfies $2^{\aleph_0}=2^{\aleph_1}=\aleph_2$, the usual nice-name calculation for ccc forcing gives at most $\aleph_2$ names for reals and at most $\aleph_2$ names for subsets of $\omega_1$.  The ground-model reals and subsets of $\omega_1$ already give the lower bounds.  Hence the equations remain true in $W_2[G_\beta]$.

For $\beta=\omega_2$, the forcing has size $\aleph_2$.  Again using ccc nice names, the number of names for subsets of $\omega_1$ is bounded by
\[
   (\aleph_2)^{\aleph_1}=(2^{\aleph_1})^{\aleph_1}=2^{\aleph_1}=\aleph_2,
\]
and the same bound applies to names for reals.  Since the iteration has added $\aleph_2$ many reals, the lower bound is $\aleph_2$, and the displayed equations hold in the final model as well.
\end{proof}

\subsection{Definition of the final iteration}

Fix a bookkeeping function
\[
F:\omega_2\longrightarrow H(\omega_2)
\]
such that every relevant object, name, and tuple of names occurs unboundedly often.  We define by induction a finite-support iteration
\[
\bigl((\forceP_\beta,\dot{\forceQ}_\beta):\beta<\omega_2\bigr)
\]
of ccc forcings over $W_2$.  Let $G_\beta$ denote the $\forceP_\beta$-generic filter.  At stage $\beta$ we use the least still unused $\omega$-block of $\vec S$ whenever a coding forcing is required.  Thus, after stage $\beta$, the blocks already used for coding form an initial segment of the block sequence, while the remaining blocks still form an independent sequence of Suslin trees in the current extension.

The value of $F(\beta)$ is interpreted in one of the following ways.  If it is not of one of these forms, we let $\dot{\forceQ}_\beta$ be trivial.

\begin{enumerate}
\item[\textup{(U)}] \textbf{Uniformization request.}  The value $F(\beta)$ codes an $\omega$-tuple of ordinals which, in the usual way, determines a tuple of $\forceP_\beta$-names
\[
(\dot x,\dot m,\dot\alpha,\dot b_1,\dot b_2,\ldots).
\]
Here $\dot x$ and the $\dot b_i$ are names for reals, $\dot m$ is a name for a natural number coding a projective formula, and $\dot\alpha$ is a name for an ordinal below the continuum.  The precise factor $\dot{\forceQ}_\beta$ is then defined by the four uniformization cases in Subsection~\ref{subsec:forcing-global-sigma-uniformization} below.

\item[\textup{(W)}] \textbf{Wellorder coding request.}  The value $F(\beta)$ is a pair $(\dot x_0,
\dot x_1)$ of $\forceP_\beta$-names for reals.  Let $x_i=\dot x_i^{G_\beta}$, and let $\sigma_i$ be the $<_{M_1}$-least $\forceP_\beta$-name whose $G_\beta$-value is $x_i$.  After interchanging the indices if necessary, assume $\sigma_0<_{M_1}\sigma_1$.  We then put
\[
\dot{\forceQ}_\beta^{G_\beta}=\operatorname{Code}((0,x_0,x_1),\eta_\beta),
\]
where $\eta_\beta$ is the least unused block index.  The initial coordinate $0$ only marks this as a wellorder code.

\item[\textup{(M)}] \textbf{Diagonal $\mathsf{MA}$ request.}  The value $F(\beta)$ is a pair $(\delta,\gamma)$.  In $W_2[G_\delta]$, let $\mathbb B$ be the $\gamma$-th partial order of size $\aleph_1$, provided that $\mathbb B$ is still ccc in $W_2[G_\beta]$.  The forcing $\forceP_\beta\ast\mathbb B$ has a dense subforcing of size $\aleph_1$ consisting of fully decided conditions, and therefore can be coded as a subset of $\omega_1$ in $W_2$.  Since $W_2=W_1[H_{\omega_2}]$, there is some $\nu<\omega_2$ such that a forcing equivalent to this dense subforcing belongs to $W_1[H_\nu]$.  If $\nu\leq\beta$, we set
\[
\dot{\forceQ}_\beta^{G_\beta}=\mathbb B;
\]
otherwise we force trivially at stage $\beta$.
\end{enumerate}

The point of the condition $\nu\leq\beta$ in the $\mathsf{MA}$ case is that the unused tail of the $\vec S$-sequence remains genuinely generic over the model in which the side forcing $\mathbb B$ has already appeared.  This prevents the $\mathsf{MA}$ stages from creating unintended coding patterns.

\begin{lemma}\label{lemma:tail-preservation-ma-stages}
Suppose stage $\beta$ is a nontrivial diagonal $\mathsf{MA}$ stage, so that $\dot{\forceQ}_\beta^{G_\beta}=\mathbb B$ is chosen as in \textup{(M)} with witness $\nu\leq\beta$.  Then every unused block $\vec S^{\zeta}$, $\zeta\geq\beta$, remains an independent block of Suslin trees in $W_2[G_{\beta+1}]$.
\end{lemma}

\begin{proof}
By the choice of $\nu$, the forcing $\forceP_\beta\ast\mathbb B$ has a dense copy in $W_1[H_\nu]$.  Factor the construction of $W_2$ over $W_1[H_\nu]$ into the initial part up to $\nu$ and the tail $H_{\nu,\omega_2}$.  Since the forcing $\forceP_\beta\ast\mathbb B$ belongs to the initial side, the tail forcing commutes with it.  Equivalently, $H_{\nu,\omega_2}$ remains generic over the extension by $\forceP_\beta\ast\mathbb B$.

The blocks $\vec S^\zeta$ for $\zeta>\nu$, and in particular for $\zeta\geq\beta$, are read from this tail generic.  The Cohen construction of the Suslin trees is absolute between models with the same $\omega_1$, and finite products of tail blocks are handled in the same way.  Hence all trees in the unused blocks remain Suslin, and the unused blocks remain independent in $W_2[G_{\beta+1}]$.
\end{proof}

The same freshness convention handles the wellorder and uniformization coding stages: such stages intentionally consume the next unused block, and Lemma~\ref{lemma:tail-preservation-ma-stages} handles the only possible interference from the diagonal $\mathsf{MA}$ side forcing.

\subsection{Proof that $\mathsf{MA}$ holds true}

Let $G_{\omega_2}$ be $\forceP_{\omega_2}$-generic over $W_2$.

\begin{lemma}\label{lemma:ma-final-model}
$W_2[G_{\omega_2}]\models\mathsf{MA}$.
\end{lemma}

\begin{proof}
Let $\mathbb B$ be a ccc forcing of size $\aleph_1$ in the final model, and let $\langle D_\xi:\xi<\omega_1\rangle$ be a family of dense subsets of $\mathbb B$.  Since the iteration is ccc and has length $\omega_2$, there is some $\delta<\omega_2$ such that $\mathbb B$ and the sequence of dense sets belong to $W_2[G_\delta]$.  By bookkeeping, unboundedly many later stages consider the corresponding pair $(\delta,\gamma)$.  Choose such a stage $\beta$ above a witness $\nu$ for the diagonal requirement in \textup{(M)}.  At that stage the iteration forces with $\mathbb B$.  The $\mathbb B$-generic added at stage $\beta$ meets every dense set from the displayed sequence, and this remains true in the final ccc extension.  Hence $\mathsf{MA}$ holds.
\end{proof}

The combination of $\mathsf{MA}$ with the $\Sigma^1_3$ generic absoluteness available in our background model gives the third-level regularity consequences needed for the theorem.

\begin{lemma}\label{HjorthsLemma}
In $W_2[G_{\omega_2}]$, every boldface $\Sigma^1_3$ set of reals is Lebesgue measurable and has the Baire property.
\end{lemma}

\begin{proof}
We recall the standard argument of Hjorth.  Let $A$ be defined by a $\Sigma^1_3(x)$ formula $\phi(v,x)$.  Let $T_2$ be the Martin--Solovay tree computed in $M_1$.  In the present extension, $T_2$ is still the correct tree for the universal $\Sigma^1_3$ set, and $L[T_2,x]$ has at most $\aleph_1$ many reals relevant to the construction.

By $\mathsf{MA}$, the union of the null Borel sets coded in $L[T_2,x]$ is null.  Thus almost every real is random over $L[T_2,x]$.  In $L[T_2,x]$, take the Boolean value of $\phi(\dot r,x)$ in the random algebra; this Boolean value is represented by a Borel set $B$.  If $r$ is random over $L[T_2,x]$, then the forcing theorem gives
\[
r\in B \quad\Longleftrightarrow\quad L[T_2,x,r]\models \phi(r,x).
\]
The $\Sigma^1_3$ correctness supplied by $T_2$ identifies the right hand side with $W_2[G_{\omega_2}]\models\phi(r,x)$.  Hence $A\triangle B$ is contained in the null set of reals which are not random over $L[T_2,x]$.  So $A$ is Lebesgue measurable.

The proof of the Baire property is identical, replacing the random algebra by Cohen forcing and null sets by meager sets.
\end{proof}

\subsection{Global $\Sigma$-uniformization: the coding rules}\label{subsec:forcing-global-sigma-uniformization}

We now make the uniformization part of the iteration explicit.  The point of the construction is to mark, at the same projective level, the first successful triple in the canonical list of possible witnesses.  We keep the notation from the previous subsection and write simply
\[
\operatorname{Code}(z)
\]
for coding $z$ into the next unused block of $\vec S$.

We shall use the following compression lemma throughout the argument.

\begin{lemma}\label{lemma:compression-pi-alpha}
In every intermediate model in which a uniformization stage is evaluated, and in the final model, for every $\alpha<\omega_2$, $\alpha> 0$ there is a canonical bijection
\[
   \pi_\alpha:(2^\omega)^\alpha\longrightarrow 2^\omega .
\]
These bijections may be chosen uniformly from the canonical wellorder of the reals.
\end{lemma}

\begin{proof}
By Lemma~\ref{lemma:cardinal-equations-at-stages}, at the relevant stages
\[
   2^{\aleph_0}=2^{\aleph_1}=\aleph_2.
\]
If $\alpha<\omega_2$, then $|\alpha|\leq\aleph_1$, and hence
\[
   |(2^\omega)^\alpha|
   =(2^{\aleph_0})^{|\alpha|}
   \leq (\aleph_2)^{\aleph_1}
   =(2^{\aleph_1})^{\aleph_1}
   =2^{\aleph_1}
   =\aleph_2.
\]
The reverse inequality is immediate because $2^\omega$ embeds into $(2^\omega)^\alpha$ whenever $\alpha>0$.  Thus the two sets have the same cardinality.  The canonical wellorder of the reals, equivalently the canonical wellorder of $H(\omega_2)$ available in the construction, selects the least such bijection uniformly.
\end{proof}

If $b$ is a real and $\alpha$ is understood, we write
\[
\pi_\alpha^{-1}(b)=(b^\eta:\eta<\alpha).
\]
Thus one real $b$ codes the whole sequence of reals $(b^\eta)_{\eta<\alpha}$.  The bookkeeping is understood at the level of names; in the verification below we suppress the dots and write the corresponding final reals.

Let $<$ be the final canonical wellorder of the reals.  For each real $x$ let
\[
(x,y^0,a_0^0),(x,y^1,a_0^1),\ldots,(x,y^\xi,a_0^\xi),\ldots\qquad(\xi<\omega_2)
\]
be the $<$-increasing enumeration of all triples with first coordinate $x$, without repetitions, and with $(x,y^0,a_0^0)=(x,0,0)$.

For the even levels write a formula coded by $m$ in the form
\[
\varphi_m(x,y)\equiv
\exists a_0\,\forall a_1\exists a_2\cdots\exists a_{2n-4}\,
\psi(x,y,a_0,a_1,\ldots,a_{2n-4}),
\tag{E}
\]
where $n\geq 2$ and $\psi\in\Pi^1_3$.  Let
\[
S^{\mathrm{ev}}_m(x,y,a_0)
\]
abbreviate the displayed tail formula
\[
\forall a_1\exists a_2\cdots\exists a_{2n-4}\,
\psi(x,y,a_0,a_1,\ldots,a_{2n-4}).
\]
Here and below, if $2n-4=0$, the string of quantifiers is empty and the formula is just $\psi(x,y,a_0)$.

The even marker is
\[
M^{\mathrm{ev}}_m(x,y,a_0)
\]
and means
\[
\exists b_1\forall b_2\exists b_3\cdots\forall b_{2n-4}\,
\operatorname{Coded}(\#\psi,x,y,a_0,b_1,\ldots,b_{2n-4}).
\tag{ME}
\]
Again, for $2n-4=0$ this means simply
\[
\operatorname{Coded}(\#\psi,x,y,a_0).
\]
The associated uniformizing relation is
\begin{align*}
\sigma_{\mathrm{even}}(x,y,m)\quad\Longleftrightarrow\quad
&\bigl(y=0\wedge S^{\mathrm{ev}}_m(x,0,0)\bigr)\;\lor\; \\
&\bigl(\neg S^{\mathrm{ev}}_m(x,0,0)\wedge
\exists a_0\bigl(S^{\mathrm{ev}}_m(x,y,a_0)\wedge M^{\mathrm{ev}}_m(x,y,a_0)\bigr)\bigr).
\end{align*}
This is a $\Sigma^1_{2n}$ formula.  Indeed, the first disjunct is $\Pi^1_{2n-1}$, hence belongs to $\Sigma^1_{2n}$, while the second disjunct is a conjunction of a $\Sigma^1_{2n-1}$ statement with a $\Sigma^1_{2n}$ statement.

For the odd levels write
\[
\varphi_m(x,y)\equiv
\exists a_0\,\forall a_1\exists a_2\cdots\forall a_{2n-3}\,
\psi(x,y,a_0,a_1,\ldots,a_{2n-3}),
\tag{O}
\]
where $n\geq 2$ and $\psi\in\Sigma^1_3$.  Let
\[
S^{\mathrm{odd}}_m(x,y,a_0)
\]
abbreviate
\[
\forall a_1\exists a_2\cdots\forall a_{2n-3}\,
\psi(x,y,a_0,a_1,\ldots,a_{2n-3}).
\]
The odd marker is the dual non-coding marker
\[
M^{\mathrm{odd}}_m(x,y,a_0)
\]
meaning
\[
\exists b_1\forall b_2\exists b_3\cdots\exists b_{2n-3}\,
\neg\operatorname{Coded}(\#\psi,x,y,a_0,b_1,\ldots,b_{2n-3}).
\tag{MO}
\]
The final quantifier in (MO) is existential, since $2n-3$ is odd.  The associated relation is
\begin{align*}
\sigma_{\mathrm{odd}}(x,y,m)\quad\Longleftrightarrow\quad
&\bigl(y=0\wedge S^{\mathrm{odd}}_m(x,0,0)\bigr)\;\lor\; \\
&\bigl(\neg S^{\mathrm{odd}}_m(x,0,0)\wedge
\exists a_0\bigl(S^{\mathrm{odd}}_m(x,y,a_0)\wedge M^{\mathrm{odd}}_m(x,y,a_0)\bigr)\bigr).
\end{align*}
This is a $\Sigma^1_{2n+1}$ formula.

We now describe the forcing factor used at a uniformization stage.  Work in $W_2[G_\beta]$ and suppose that the bookkeeping value at stage $\beta$ gives
\[
(x,m,\alpha,b_1,b_2,\ldots),
\]
where $0<\alpha<\omega_2$.  Only the finitely many $b_i$ needed by the formula coded by $m$ are used.  If
\[
\pi_\alpha^{-1}(b_i)=(b_i^\eta:\eta<\alpha),
\]
then the stage is defined as follows.

If $m$ codes an even formula as in (E), put
\[
\dot{\forceQ}_\beta^{G_\beta}=
\operatorname{Code}(\#\psi,x,y^\alpha,a_0^\alpha,b_1,\ldots,b_{2n-4})
\]
exactly in the case that
\[
\neg\psi(x,y^\eta,a_0^\eta,b_1^\eta,\ldots,b_{2n-4}^\eta)
\]
holds for every $\eta<\alpha$.  If this condition fails, we force trivially.

If $m$ codes an odd formula as in (O), put
\[
\dot{\forceQ}_\beta^{G_\beta}=
\operatorname{Code}(\#\psi,x,y^\alpha,a_0^\alpha,b_1,\ldots,b_{2n-3})
\]
exactly in the case that
\[
\psi(x,y^\eta,a_0^\eta,b_1^\eta,\ldots,b_{2n-3}^\eta)
\]
holds for some $\eta<\alpha$.  If no such $\eta$ exists, we force trivially.

If the bookkeeping value is ill formed, or if $m$ does not code one of the displayed normal forms, the stage is trivial.  Notice that the coding marker is always $\#\psi$, never $\#\varphi_m$.  This is important: the marker records the matrix computation which is being compressed.

The truth of the $\Sigma^1_3$ or $\Pi^1_3$ matrix $\psi$ is permanent between the intermediate model in which the stage is evaluated and the final model.  This is the base-level generic absoluteness supplied by the $M_1$ analysis used in Section~\ref{sec:sigma13-uniformization}.  Thus the above decisions are not later reversed.

\section{Proof that $\Sigma^1_n$-uniformization holds in the final model}

Let
\[
V^*=W_2[G_{\omega_2}].
\]
We verify that the relations $\sigma_{\mathrm{even}}$ and $\sigma_{\mathrm{odd}}$ defined above uniformize the corresponding projective relations in $V^*$.

We first record the no-accidental-coding consequence used below.  If a tuple $z$ is not intentionally coded at a stage using the block assigned to that stage, then no later stage can make $z$ coded in that block.  Different coding stages use different blocks, and the predicate $\operatorname{Coded}(z)$ is read from the unique block assigned to the relevant marker.  Hence all marker statements decided by the rules above remain fixed in the final extension.

\begin{lemma}\label{lemma:even-marker-verification}
Let $m$ code an even formula as in \textup{(E)}, and fix a real $x$.  Suppose $\alpha>0$ is least such that
\[
V^*\models S^{\mathrm{ev}}_m(x,y^\alpha,a_0^\alpha).
\]
Then
\[
V^*\models M^{\mathrm{ev}}_m(x,y^\alpha,a_0^\alpha),
\]
and for every $\beta>\alpha$,
\[
V^*\models \neg M^{\mathrm{ev}}_m(x,y^\beta,a_0^\beta).
\]
\end{lemma}

\begin{proof}
Put $l=2n-4$.  We write the proof for $l>0$; if $l=0$ the same argument applies with the quantifier string omitted.

First let $\eta<\alpha$.  By minimality of $\alpha$,
\[
V^*\models \neg S^{\mathrm{ev}}_m(x,y^\eta,a_0^\eta),
\]
that is,
\[
V^*\models
\exists c_1^\eta\forall c_2^\eta\exists c_3^\eta\cdots\forall c_l^\eta\neg\psi(x,y^\eta,a_0^\eta,c_1^\eta,\ldots,c_l^\eta).
\]
Choose the witnesses coherently using the final wellorder.  Compress the first witnesses by setting
\[
c_1=\pi_\alpha((c_1^\eta:\eta<\alpha)).
\]
Now let $c_2$ be arbitrary and write $\pi_\alpha^{-1}(c_2)=(c_2^\eta:\eta<\alpha)$.  For each $\eta<\alpha$, choose the corresponding response $c_3^\eta$, compress these responses to a real $c_3$, and continue in this way through the alternating quantifier string.  At the end we have produced witnesses for
\[
\exists c_1\forall c_2\exists c_3\cdots\forall c_l
\Bigl[\forall\eta<\alpha\;\neg\psi(x,y^\eta,a_0^\eta,c_1^\eta,\ldots,c_l^\eta)\Bigr].
\]
By the even coding rule, whenever the bookkeeping reaches the tuple
\[
(x,m,\alpha,c_1,\ldots,c_l),
\]
the forcing codes
\[
(\#\psi,x,y^\alpha,a_0^\alpha,c_1,\ldots,c_l).
\]
Bookkeeping reaches all such tuples cofinally often after the parameters have appeared, and the no-accidental-coding lemma preserves the resulting code.  Therefore
\[
V^*\models M^{\mathrm{ev}}_m(x,y^\alpha,a_0^\alpha).
\]

Now let $\beta>\alpha$.  We show that the marker fails at $\beta$.  Fix an arbitrary first real $d_1$, and decode its $\alpha$-coordinate $d_1^\alpha$.  Since $S^{\mathrm{ev}}_m(x,y^\alpha,a_0^\alpha)$ holds, choose $d_2^\alpha$ so that the remaining tail of the success condition is satisfied.  Put arbitrary values at all other coordinates $\eta<\beta$, and compress them to a real $d_2$.  Continue through the alternating quantifiers: whenever a universal move is made, use the success of the $\alpha$-coordinate to choose the next existential response at coordinate $\alpha$ and fill the other coordinates arbitrarily.  At the end the decoded tuple satisfies
\[
\psi(x,y^\alpha,a_0^\alpha,d_1^\alpha,\ldots,d_l^\alpha).
\]
Thus, at every bookkeeping stage for
\[
(x,m,\beta,d_1,\ldots,d_l),
\]
the even rule sees an earlier successful coordinate, namely $\eta=\alpha$, and therefore forces trivially.  By no accidental coding,
\[
\operatorname{Coded}(\#\psi,x,y^\beta,a_0^\beta,d_1,\ldots,d_l)
\]
fails for the tuple produced by this counterplay.  This gives the negation of the alternating marker (ME), and hence
\[
V^*\models \neg M^{\mathrm{ev}}_m(x,y^\beta,a_0^\beta).
\]
\end{proof}

\begin{lemma}\label{lemma:odd-marker-verification}
Let $m$ code an odd formula as in \textup{(O)}, and fix a real $x$.  Suppose $\alpha>0$ is least such that
\[
V^*\models S^{\mathrm{odd}}_m(x,y^\alpha,a_0^\alpha).
\]
Then
\[
V^*\models M^{\mathrm{odd}}_m(x,y^\alpha,a_0^\alpha),
\]
and for every $\beta>\alpha$,
\[
V^*\models \neg M^{\mathrm{odd}}_m(x,y^\beta,a_0^\beta).
\]
\end{lemma}

\begin{proof}
Put $l=2n-3$.  Since $\alpha$ is least successful, for every $\eta<\alpha$ we have
\[
V^*\models
\exists c_1^\eta\forall c_2^\eta\exists c_3^\eta\cdots\exists c_l^\eta
\neg\psi(x,y^\eta,a_0^\eta,c_1^\eta,\ldots,c_l^\eta).
\]
Choose and compress these witnesses exactly as in the proof of Lemma~\ref{lemma:even-marker-verification}.  Thus we obtain
\[
\exists c_1\forall c_2\exists c_3\cdots\exists c_l
\Bigl[\forall\eta<\alpha\;\neg\psi(x,y^\eta,a_0^\eta,c_1^\eta,\ldots,c_l^\eta)\Bigr].
\]
For the odd levels the coding rule is reversed: if all earlier coordinates satisfy the negated matrix, the forcing is trivial.  Therefore the relevant tuple
\[
(\#\psi,x,y^\alpha,a_0^\alpha,c_1,\ldots,c_l)
\]
is not coded.  Since no later stage can accidentally code it, the non-coding marker (MO) holds at $\alpha$.

Now let $\beta>\alpha$.  To refute the marker at $\beta$, fix an arbitrary first move $d_1$ and use the success of the $\alpha$-coordinate to choose the next response $d_2^\alpha$; continue through the alternating string.  After compression, every final tuple produced by this counterplay has
\[
\psi(x,y^\alpha,a_0^\alpha,d_1^\alpha,\ldots,d_l^\alpha)
\]
at the earlier coordinate $\alpha$.  Hence the odd coding rule codes
\[
(\#\psi,x,y^\beta,a_0^\beta,d_1,\ldots,d_l)
\]
for the corresponding tuple.  This gives the complement of the alternating non-coding marker (MO).  Thus
\[
V^*\models \neg M^{\mathrm{odd}}_m(x,y^\beta,a_0^\beta).
\]
\end{proof}

\begin{theorem}\label{theorem:global-sigma-uniformization-final}
In $V^*$, every boldface $\Sigma^1_k$ relation on the reals, for $k\geq 4$, has a boldface $\Sigma^1_k$ uniformization.
\end{theorem}

\begin{proof}
We prove the even case; the odd case is identical, using Lemma~\ref{lemma:odd-marker-verification} in place of Lemma~\ref{lemma:even-marker-verification}.

Let $m$ code a $\Sigma^1_{2n}$ relation $\varphi_m(x,y)$ in the normal form (E), with $n\geq 2$.  Fix $x$ such that $V^*\models\exists y\,\varphi_m(x,y)$.  If
\[
V^*\models S^{\mathrm{ev}}_m(x,0,0),
\]
then the first disjunct in $\sigma_{\mathrm{even}}(x,y,m)$ selects $y=0$, and the second disjunct is blocked by $\neg S^{\mathrm{ev}}_m(x,0,0)$.

Assume therefore that $S^{\mathrm{ev}}_m(x,0,0)$ fails.  Let $\alpha>0$ be least such that
\[
V^*\models S^{\mathrm{ev}}_m(x,y^\alpha,a_0^\alpha).
\]
Such an $\alpha$ exists because $\varphi_m(x,y)$ holds for some $y$, and the list contains every pair $(y,a_0)$.  By Lemma~\ref{lemma:even-marker-verification},
\[
V^*\models M^{\mathrm{ev}}_m(x,y^\alpha,a_0^\alpha),
\]
so $\sigma_{\mathrm{even}}(x,y^\alpha,m)$ holds.

If $\beta<\alpha$, then $S^{\mathrm{ev}}_m(x,y^\beta,a_0^\beta)$ fails by the choice of $\alpha$, so the pair $(y^\beta,a_0^\beta)$ cannot satisfy the second disjunct of $\sigma_{\mathrm{even}}$.  If $\beta>\alpha$, then Lemma~\ref{lemma:even-marker-verification} gives
\[
V^*\models\neg M^{\mathrm{ev}}_m(x,y^\beta,a_0^\beta),
\]
so the pair $(y^\beta,a_0^\beta)$ again cannot satisfy the second disjunct.  Thus exactly one $y$ is selected.  Since the selected $y$ satisfies $S^{\mathrm{ev}}_m(x,y,a_0)$ for some $a_0$, it also satisfies $\varphi_m(x,y)$.

The formula $\sigma_{\mathrm{even}}(x,y,m)$ is $\Sigma^1_{2n}$, as observed above.  Hence it is a $\Sigma^1_{2n}$ uniformization of $\varphi_m$.  The odd case uses $\sigma_{\mathrm{odd}}$ and gives a $\Sigma^1_{2n+1}$ uniformization.  This proves the theorem.
\end{proof}

Combining Theorem~\ref{theorem:global-sigma-uniformization-final} with the previously established $\Sigma^1_2$ case and the $\Sigma^1_3$ base case proved in Section~\ref{sec:sigma13-uniformization}, we obtain the announced global tail uniformization over the final model.

\section{The $\Sigma^1_3$-uniformization in $W_{2}[G_{\omega_2}]$}\label{sec:sigma13-uniformization}

The higher uniformization argument leaves the $\Sigma^1_3$ base case to be checked separately.  At this first nontrivial level the coding iteration is not needed.  What is used is the standard Steel analysis of $M_1$ and of its relativizations.  We only need this analysis for the small generic extensions which occur in the construction.  Thus, throughout this section, a \emph{small generic extension of $M_1$} means an extension $M_1[G]$ by a set forcing whose size in $M_1$ is below the Woodin cardinal of $M_1$ and to which the canonical iteration strategy of $M_1$ lifts.  All intermediate models used in the construction, and in particular the final model $W_2[G_{\omega_2}]$, are of this form.

For a real $s$ in such an extension, let $M_1(s)$ denote the canonical proper class mouse over $s$ with one Woodin cardinal.  The relevant mice exist and are iterable by the lifted $M_1$ strategy.  The argument below uses a capture theorem for this canonical mouse over $s$; it does not identify $M_1(s)$ with the forcing extension generated by $s$.

\begin{definition}
A \emph{simple $s$-mouse} is a sound, $\Pi^1_2$-iterable $s$-premouse which projects to $\omega$ and is an initial segment of the canonical construction of $M_1(s)$.
\end{definition}

The assertion that a real $c$ codes a sound simple $s$-mouse is a $\Pi^1_2(s,c)$ condition.  Any two such mice compare by initial segment.  Hence the reals of $M_1(s)$ carry the usual good $\Sigma^1_3(s)$ wellorder $<^1_s$: first compare the least simple $s$-mouse containing the real, and then use the canonical wellorder of that mouse.

We shall use the following standard form of Steel's relativized capture theorem at the first mouse level; see the projective definability and comparison analysis of $M_1$ in \cite{Steel2,Steel3}.

\begin{fact}\label{fact:m1-capture-and-good-wellorder}
Let $N=M_1[G]$ be a small generic extension of $M_1$, and let $s\in\mathbb R^N$.

\begin{enumerate}
    \item If $N$ satisfies a $\Sigma^1_3(s)$ statement of the form
    \[
       \exists u\,\exists v\,\theta(s,u,v),
    \]
    where $\theta$ is $\Pi^1_2$, then there are such witnesses $u,v\in M_1(s)$.

    \item For reals from $M_1(s)$, $\Pi^1_2$ truth is computed correctly by $M_1(s)$ and is preserved to $N$.  Equivalently, if $u,v\in M_1(s)$ and $\theta$ is $\Pi^1_2$, then
    \[
       M_1(s)\models\theta(s,u,v)
       \quad\Longleftrightarrow\quad
       N\models\theta(s,u,v).
    \]

    \item The wellorder $<^1_s$ is good for $\Sigma^1_3$ definitions.  More explicitly, for every $\Pi^1_2$ formula $\theta(s,u,v)$ the relation
    \[
       \operatorname{Least}_{\theta}(s,u)
    \]
    saying that $u\in M_1(s)$ and $u$ is the $<^1_s$-least real for which there is a $v\in M_1(s)$ with $M_1(s)\models\theta(s,u,v)$ is uniformly $\Sigma^1_3(s)$.
\end{enumerate}
\end{fact}

For completeness, let us spell out why this is the right theorem to apply.  The mouse $M_1(s)$ is the canonical mouse over the parameter $s$.  The comparison theorem for $\Pi^1_2$-iterable $s$-mice identifies its countable initial segments with the segments captured by the $\Pi^1_2$ mouse condition.  Therefore a small generic extension of $M_1$ cannot create a new $\Sigma^1_3(s)$ witness without some countable initial segment of $M_1(s)$ already capturing the corresponding branch.  The same comparison analysis gives the $\Pi^1_2$ correctness used in item (2), and the usual definition of the wellorder by least simple $s$-mouse gives item (3).

\begin{theorem}\label{thm:sigma13-uniformization-m1-generic}
Every small generic extension of $M_1$ satisfies boldface $\Sigma^1_3$-uniformization.  In particular,
\[
   W_2[G_{\omega_2}]
\]
satisfies boldface $\Sigma^1_3$-uniformization.
\end{theorem}

\begin{proof}
Let $N=M_1[G]$ be a small generic extension of $M_1$.  Work in $N$, and let $A\subseteq\mathbb R^2$ be a boldface $\Sigma^1_3$ relation.  Fix a real parameter $a$ and a $\Pi^1_2$ formula $\psi$ such that
\[
   A(x,y)\quad\Longleftrightarrow\quad \exists z\,\psi(a,x,y,z).
\]
For each real $x$, put $s=a\oplus x$.  Define $A^*(x,y)$ to hold iff
\[
   \operatorname{Least}_{\theta}(s,y),
\]
where
\[
   \theta(s,y,z)
\]
is the $\Pi^1_2$ formula obtained from $\psi(a,x,y,z)$ after decoding $s=a\oplus x$.
By Fact~\ref{fact:m1-capture-and-good-wellorder}(3), the relation $A^*$ is $\Sigma^1_3(a)$, hence boldface $\Sigma^1_3$.

We verify that $A^*$ uniformizes $A$.  If $A^*(x,y)$ holds, then by definition there is some $z\in M_1(a\oplus x)$ such that
\[
   M_1(a\oplus x)\models \psi(a,x,y,z).
\]
By Fact~\ref{fact:m1-capture-and-good-wellorder}(2), $N\models\psi(a,x,y,z)$.  Hence $A(x,y)$ holds.

The relation $A^*$ is single-valued because $<^1_{a\oplus x}$ is a wellorder and $A^*(x,y)$ asserts that $y$ is the $<^1_{a\oplus x}$-least real in $M_1(a\oplus x)$ for which a suitable $z$ exists.

Finally suppose that the $x$-section of $A$ is nonempty in $N$.  Then
\[
   N\models \exists y\exists z\,\psi(a,x,y,z).
\]
By Fact~\ref{fact:m1-capture-and-good-wellorder}(1), there are witnesses $y,z\in M_1(a\oplus x)$.  Therefore the set of $<^1_{a\oplus x}$-candidates is nonempty, so it has a least element.  For this least element $y_0$, Fact~\ref{fact:m1-capture-and-good-wellorder}(3) gives $A^*(x,y_0)$.

Thus $A^*$ is a boldface $\Sigma^1_3$ uniformization of $A$ in $N$.  The final model $W_2[G_{\omega_2}]$ is a small generic extension of $M_1$ by the forcings fixed in the construction, so the final assertion follows.
\end{proof}

\section{A byproduct on generalized regularity properties}

In this section we record a consequence of the $M_1$ case for generalized
notions of measurability associated with tree forcings. We use the notation of
Fischer--Friedman--Khomskii (see \cite{FischerFriedmanKhomskii}) and Friedman--Schrittesser (\cite{FriedmanSchrittesser}). Thus $B$ denotes the
random algebra, equivalently Lebesgue measurability; $C$ denotes Cohen forcing,
equivalently the Baire property; $S$ denotes Sacks measurability, equivalently
Marczewski measurability; and $M$ denotes Miller measurability. If $P$ is one
of these tree forcings and $\Gamma$ is a pointclass, we write $\Gamma(P)$ for
the assertion that every set in $\Gamma$ is $P$-measurable.

We shall use the following theorem of Brendle and L\"owe, as quoted in
\cite{FischerFriedmanKhomskii}. If $\Gamma$ is closed under continuous
preimages, then
\[
\Gamma(C)\Rightarrow \Gamma(M)
\quad\text{and}\quad
\Gamma(B)\Rightarrow \Gamma(S).
\]
They also show the implication chain
\[
\Gamma(L)\Rightarrow \Gamma(M)\Rightarrow \Gamma(S),
\]
where $L$ denotes Laver measurability. We shall only use the first two
implications above.

\begin{corollary}\label{corollary-generalized-measurability}
Assume that $M_1$ exists. Then there is a generic extension preserving one
Woodin cardinal in which
\[
\Sigma^1_3(B)+\Sigma^1_3(C)+\Sigma^1_3(S)+\Sigma^1_3(M)
\]
hold, and in which there is a $\Delta^1_4$-definable wellorder of the reals.
Consequently, for each $P\in\{B,C,S,M\}$, the model satisfies
\[
\Sigma^1_3(P)+\neg\Delta^1_4(P).
\]
Thus the construction gives the $m=1$, $m_i=3$ instance of the
Friedman--Schrittesser question 1.21 for the forcings $B,C,S$ and $M$.
\end{corollary}

\begin{proof}
By the $n=1$ case of the main construction, the final model satisfies $\MA$,
every $\boldsymbol{\Sigma}^1_3$ set of reals is Lebesgue measurable and has
the Baire property, and there is a $\Delta^1_4$-definable wellorder of the
reals. In the notation above this says that
\[
\Sigma^1_3(B)+\Sigma^1_3(C)
\]
hold. Since $\Sigma^1_3$ is closed under continuous preimages, the
Brendle--L\"owe implications yield
\[
\Sigma^1_3(B)\Rightarrow\Sigma^1_3(S)
\quad\text{and}\quad
\Sigma^1_3(C)\Rightarrow\Sigma^1_3(M).
\]
Therefore
\[
\Sigma^1_3(B)+\Sigma^1_3(C)+\Sigma^1_3(S)+\Sigma^1_3(M)
\]
hold in the final model.

It remains to explain the failure of $\Delta^1_4(P)$ for
$P\in\{B,C,S,M\}$. We use the $\Delta^1_4$ wellorder of the reals produced by
the construction. For $P=B$ or $P=C$, this is the usual construction of a
$\Delta^1_4$ Vitali/Bernstein type counterexample to Lebesgue measurability,
respectively to the Baire property. For $P=S$ or $P=M$, the same argument is
carried out with the corresponding tree conditions.

More explicitly, fix $P\in\{S,M\}$ and use the $\Delta^1_4$ wellorder to
enumerate all $P$-conditions. Recursively, for each condition $T$, choose two
fresh branches $x_T^0,x_T^1\in [T]$, putting $x_T^0$ into a set $A_P$ and
keeping $x_T^1$ outside $A_P$. The usual Bernstein argument then shows that no
$P$-condition can be homogeneous for $A_P$: every $P$-condition has a stronger
condition whose body meets both $A_P$ and its complement. Hence $A_P$ is not
$P$-measurable. Since the recursion is performed using the constructed 
$\Delta^1_4$ wellorder, the resulting set may be chosen $\Delta^1_4$. Thus
$\neg\Delta^1_4(P)$ holds for $P\in\{S,M\}$ as well. This completes the proof.
\end{proof}

The preceding corollary should not be read as treating all tree forcings in
the Friedman--Schrittesser diagram. In particular, the Brendle--L\"owe
implications used above do not yield $\Sigma^1_3(L)$, $\Sigma^1_3(V)$ or
$\Sigma^1_3(R)$, where $L,V,R$ denote Laver, Silver and Mathias measurability,
respectively. The implication involving Laver goes in the direction
$\Gamma(L)\Rightarrow\Gamma(M)$, not conversely. Thus the corollary is limited
to $P\in\{B,C,S,M\}$.

\section{Forcing over $M_n$}

In this final section we explain how the preceding construction is lifted
from $M_1$ to the canonical inner model $M_n$ with $n$ Woodin cardinals.
The point of this section is not to repeat the coding computations of
Section 8. Those computations are entirely formal once the relevant
initial segments of the ground model are available in the decoding models.
The only change is that, over $M_n$, the assertion that a decoded premouse is
an initial segment of the ground model has higher projective complexity.

We therefore isolate the precise modification. In the construction over
$M_1$ we used the fact that, in any $\omega_1$-preserving generic extension,
the relevant countable initial segments of $M_1$ are captured by the
$\Pi^1_2$-definable class $\mathcal I$. Over $M_n$ the corresponding class is
more complex.

\begin{lemma}\label{Mn-initial-segments}
Let $M_n[G]$ be an $\omega_1$-preserving forcing extension of $M_n$ of the kind fixed in Paragraph~\ref{par:steel-mn-preservation-convention}. Then in $M_n[G]$ there is a $\Pi^1_{n+1}$-definable class $\mathcal I_n$ of countable premice satisfying the conclusions of Lemma~\ref{lem:definable-mn-initial-segments}. In particular, its elements are externally exactly the relevant countable initial segments $\mathcal J^{M_n}_\eta$ of $M_n$, and these segments are cofinal below $\omega_1$.
\end{lemma}

\begin{proof}
This is precisely Lemma~\ref{lem:definable-mn-initial-segments}, using the definition of the lower-part approximation class from Definition~\ref{def:mn-lower-part-approximations}. The final compatibility statement used in the decoding arguments follows from the recovery convention in Definition~\ref{def:recover-mn-omega1} and Lemma~\ref{lem:recover-mn-omega1}: the decoded premouse is used only externally as the correct initial segment of $M_n$, so the ladder system, almost disjoint family, and Suslin-tree blocks computed from it agree with the ambient construction.
\end{proof}

We now repeat the definition of the coding predicate in the form in which it
will be used over $M_n$. Let $x\in 2^\omega$ be a real. As before, we first
force a pattern into an unused $\omega$-block
\[
\vec S^\xi=(S^\xi_k:k<\omega)
\]
of the independent sequence of Suslin trees: for each $k<\omega$ we add an
$\omega_1$-branch through $S^\xi_k$ if $x(k)=1$, and specialize $S^\xi_k$ if
$x(k)=0$. We then collect the resulting branches and specializing functions,
together with the auxiliary objects needed to reconstruct the suitable model
which computes the block $\vec S^\xi$, into a set $X\subseteq\omega_1$.
Finally, exactly as in Section 8, we replace $X$ by a set
$Y\subseteq\omega_1$ which also codes a club of countable elementary
submodels, and then almost-disjoint-code $Y$ by a real $r_Y$ using the almost
disjoint family computed from the relevant initial segment of $M_n$.

The resulting decoding formula is the following.

\begin{definition}\label{definition-psi-n}
Let $\psi_n(x,r)$ be the assertion that for every countable transitive model
$M$ of $\ZFP+``\aleph_1$ exists$"$ such that $r\in M$, the following holds.

Assume that, externally, there is some
$\mathcal J^{M_n}_{\eta}\in\mathcal I_n$ such that
\[
\omega_1^M=\omega_1^{\mathcal J^{M_n}_{\eta}}
\quad\text{and}\quad
\mathcal J^{M_n}_{\eta}\in M.
\]
Then, using the almost disjoint family computed from
$\mathcal J^{M_n}_{\eta}$, the model $M$ decodes from $r$ the following
objects:
\begin{enumerate}
    \item a transitive model $m$ of $\ZFC^-$ satisfying
    $``\aleph_1$ exists$"$;
    \item a model $a$ which $M$ believes to be suitable and which computes the
    relevant block $\vec s^\xi$ of Suslin trees;
    \item branches $\vec b$ through the trees $s^\xi_k$ for those
    $k\in x$;
    \item specializing functions $\vec f$ for the trees $s^\xi_k$ for those
    $k\notin x$;
\end{enumerate}
and for every $\ZFP+``\aleph_1$ exists$"$ model of the form
$L_\zeta[a,\vec b,\vec f]$ computed in $m$,
\[
L_\zeta[a,\vec b,\vec f]\models
\forall k<\omega\,
\bigl(k\in x\rightarrow s^\xi_k
\text{ has an }\omega_1\text{-branch}\bigr)
\]
and
\[
L_\zeta[a,\vec b,\vec f]\models
\forall k<\omega\,
\bigl(k\notin x\rightarrow s^\xi_k
\text{ is special}\bigr).
\]
Moreover the decoded premouse $m$ is, externally, an element of
$\mathcal I_n$.
\end{definition}

Thus $\psi_n(x,r)$ is literally the formula $\psi(x,r)$ of Section 8 with
$\mathcal I$ replaced by $\mathcal I_n$ and $M_1$ replaced by $M_n$.

\begin{lemma}\label{complexity-of-psi-n}
The formula $\psi_n(x,r)$ is $\Pi^1_{n+2}$. Consequently
\[
\sigma_n(x)\quad\Longleftrightarrow\quad \exists r\,\psi_n(x,r)
\]
is a $\Sigma^1_{n+3}$ formula.
\end{lemma}

\begin{proof}
All parts of the decoding assertion other than the verification
$m\in\mathcal I_n$ have exactly the same complexity as in Section 8. They are
arithmetical, or at worst projective of the same bounded complexity coming
from the assertion that countable transitive models decode the prescribed
branches and specializing functions correctly.

The only term whose complexity changes is the assertion that the decoded
premouse $m$ belongs to the class of countable initial segments of the ground
model. Over $M_1$ this was expressed by membership in the
$\Pi^1_2$ class $\mathcal I$. Over $M_n$ it is expressed by membership in
the $\Pi^1_{n+1}$ class $\mathcal I_n$.

Since $\psi_n(x,r)$ universally quantifies over countable transitive
decoding models and then requires the relevant decoding conclusion whenever
such an externally correct initial segment is present, the replacement of
$\mathcal I$ by $\mathcal I_n$ raises the complexity of the formula to
$\Pi^1_{n+2}$. Therefore the assertion that there exists a real $r$
witnessing the code is $\Sigma^1_{n+3}$.
\end{proof}

The crucial absoluteness consequence of the coding predicate is unchanged.

\begin{lemma}\label{Mn-coded-real-is-really-coded}
Let $r$ be a real such that $\psi_n(x,r)$ holds. Then, in $L[r]$, there is a
suitable model $A$ which computes the relevant block $\vec S^\xi$ correctly,
together with branches and specializing functions witnessing that
\[
\forall k<\omega\,
\bigl(k\in x\rightarrow S^\xi_k
\text{ has an }\omega_1\text{-branch}\bigr)
\]
and
\[
\forall k<\omega\,
\bigl(k\notin x\rightarrow S^\xi_k
\text{ is special}\bigr).
\]
In particular, if $\sigma_n(x)$ holds, then the characteristic function of
$x$ is genuinely written into the block $\vec S^\xi$.
\end{lemma}

\begin{proof}
The proof is the proof of Lemma
\ref{definable well-order of reals determines real world}, with
$\mathcal I$ replaced by $\mathcal I_n$.

We recall the argument to make clear that no new computation is involved.
Suppose first that $\psi_n(x,r)$ holds. By reflecting to countable elementary
submodels and collapsing, the statement $\psi_n(x,r)$ also holds for
arbitrarily large transitive models once the parameter
$\mathcal J^{M_n}_{\omega_1}$ is replaced by the corresponding countable
initial segment $\mathcal J^{M_n}_{\eta}$ in the collapse. The external
hypothesis that this initial segment belongs to $\mathcal I_n$ is exactly
what guarantees that the ladder system, the almost disjoint family, and the
Suslin-tree block computed in the collapse agree with the ambient
construction.

Taking unions along a continuous chain of such countable decoding models of
length $\omega_1$, one obtains in $L[r]$ a suitable model $A$, a sequence of
branches through those $S^\xi_k$ with $k\in x$, and specializing functions
for those $S^\xi_k$ with $k\notin x$. The assertions ``there is an
$\omega_1$-branch through $S^\xi_k$'' and ``$S^\xi_k$ is special'' are upward
absolute from the relevant transitive models to $L[r]$. Hence the decoded
pattern in $L[r]$ is the intended pattern.

The only use of the ground model in this argument is the external correctness
of the decoded initial segment and the objects computed from it. Lemma
\ref{Mn-initial-segments} supplies this over $M_n$, exactly as the
corresponding fact about $\mathcal I$ supplied it over $M_1$.
\end{proof}

We define the $M_n$-version of the coding forcing as follows.

\begin{definition}\label{Mn-Code-definition}
Let $\xi<\omega_2$ and let $x\in 2^\omega$. The forcing
$\operatorname{Code}_n(x,\xi)$ is the two-step forcing
\[
\operatorname{Code}_n(x,\xi)=
\left(
\prod_{k\in x} S^\xi_k
\times
\prod_{k\notin x} Sp(S^\xi_k)
\right)
\ast
\mathbb A(Y),
\]
where the first product has finite support and $Y\subseteq\omega_1$ is the
set constructed from the resulting branches, specializing functions, suitable
model data, and the club of countable decoding structures exactly as in
Section 8, but using the almost disjoint family computed from the relevant
initial segments of $M_n$.
\end{definition}

\begin{lemma}\label{Mn-Code-properties}
For every real $x$ and every unused block $\vec S^\xi$,
$\operatorname{Code}_n(x,\xi)$ has size $\aleph_1$, satisfies the ccc, and
forces $\sigma_n(x)$. Moreover, if $y\neq x$, then
$\operatorname{Code}_n(x,\xi)$ does not force $\sigma_n(y)$ via the same
block $\vec S^\xi$.
\end{lemma}

\begin{proof}
The ccc and size calculation are identical to those in Section 8. The first
step is a finite-support product of factors which are either Suslin trees or
specializing forcings for Suslin trees, and the second step is the same
almost disjoint coding forcing. The preservation argument uses only the
independence of the sequence of Suslin trees and the standard ccc analysis of
almost disjoint coding; neither argument mentions the fine structure of the
ground model except through the already fixed almost disjoint family.

The forcing adds precisely the branches and specializing functions dictated
by $x$ and then almost-disjoint-codes the auxiliary set $Y$ into a real
$r_Y$. By construction, $r_Y$ witnesses $\psi_n(x,r_Y)$, and hence
$\sigma_n(x)$.

Finally, Lemma \ref{Mn-coded-real-is-really-coded} prevents spurious codes.
If $\sigma_n(y)$ held through the same block, then the characteristic
function of $y$ would be genuinely written into the same block
$\vec S^\xi$. But the first step of the forcing arranged the pattern in
$\vec S^\xi$ according to $x$, and the independence of the Suslin trees
ensures that this pattern is not changed in the coordinates reserved for the
code. Hence $y=x$.
\end{proof}

Thus the full computation of Section 8 transfers to $M_n$ after the single
replacement
\[
\mathcal I\quad\rightsquigarrow\quad \mathcal I_n.
\]
The price paid for this replacement is purely descriptive-set-theoretic:
the basic coding predicate becomes $\Sigma^1_{n+3}$ rather than
$\Sigma^1_4$.

The only level not covered by this shifted coding predicate is the first one,
namely $\boldsymbol{\Sigma}^1_{n+2}$.  We now record the exact higher analogue
of the $M_1(s)$ capture theorem used in
Section~\ref{sec:sigma13-uniformization}.  Throughout the statement, a small
generic extension of $M_n$ means an extension of the kind fixed in
Paragraph~\ref{par:steel-mn-preservation-convention}; equivalently, it is one
of the extensions in which the canonical $M_n$ strategy and the relevant
comparison arguments used in this paper are preserved.

\begin{proposition}\label{prop:mn-relativized-capture}
Let $N$ be such a small generic extension of $M_n$, and let
$s\in\mathbb R^N$.  Let $M_n(s)$ be the canonical proper class $s$-mouse with
$n$ Woodin cardinals.  Then the following hold in $N$.
\begin{enumerate}
    \item If
    \[
       N\models \exists u\,\exists v\,\theta(s,u,v),
    \]
    where $\theta$ is $\Pi^1_{n+1}$, then there are witnesses
    $u,v\in\mathbb R\cap M_n(s)$ such that
    \[
       M_n(s)\models \theta(s,u,v).
    \]

    \item For reals $u,v\in M_n(s)$, $\Pi^1_{n+1}$ truth is computed
    correctly by $M_n(s)$ and is preserved to $N$.  Thus, for every
    $\Pi^1_{n+1}$ formula $\theta$,
    \[
       M_n(s)\models\theta(s,u,v)
       \quad\Longleftrightarrow\quad
       N\models\theta(s,u,v).
    \]

    \item The canonical mouse order $<_{n,s}$ on
    $\mathbb R\cap M_n(s)$ is a good $\Sigma^1_{n+2}(s)$ wellorder.  More
    explicitly, if $\theta(s,u,v)$ is $\Pi^1_{n+1}$, then the relation
    \[
       \operatorname{Least}^{n}_{\theta}(s,u)
    \]
    saying that $u\in M_n(s)$ and $u$ is the $<_{n,s}$-least real for which
    there is a $v\in M_n(s)$ with
    $M_n(s)\models\theta(s,u,v)$ is uniformly $\Sigma^1_{n+2}(s)$.
\end{enumerate}
\end{proposition}

\begin{proof}
This is the relativized Steel analysis of the canonical mice $M_n(s)$, in the
same form as Fact~\ref{fact:m1-capture-and-good-wellorder} for $n=1$.
We recall the ingredients in order to make clear that no new coding argument
is being used here.

The projective description of the sound initial segments of $M_n(s)$ is the
relativized version of the $\Pi_n$-iterability analysis recalled in
Section~2.  In the codes, the relevant $s$-premice are described by a
$\Pi^1_{n+1}(s)$ condition: they are sound, project to $\omega$, are
$n$-small over the parameter $s$, and have the required $\Pi_n$-iterability.
Steel comparison linearly orders these premice by initial segment and
identifies the correctly iterable ones with the initial segments of the
canonical mouse $M_n(s)$.

The usual comparison-and-capture argument then gives the first two clauses.
If a small generic extension satisfies a $\Sigma^1_{n+2}(s)$ assertion, write
it in the form $\exists u\exists v\,\theta(s,u,v)$ with $\theta$
$\Pi^1_{n+1}$.  The tree or mouse witnessing this assertion is captured by a
countable initial segment of $M_n(s)$; otherwise comparison with the canonical
construction would produce the same contradiction as in Steel's proof of the
projective correctness of $M_n(s)$.  Conversely, once the witnesses belong to
$M_n(s)$, $\Pi^1_{n+1}$ correctness follows from the same comparison theorem
and the preservation convention for the generic extensions considered here.

Finally, define $<_{n,s}$ by first taking the least sound initial segment of
$M_n(s)$ containing the real in question, and then using the canonical
wellorder of that premouse.  Since membership in the relevant initial segment
class is $\Pi^1_{n+1}(s)$ and comparison gives initial-segment linearity,
initial segments of $<_{n,s}$ are uniformly $\Sigma^1_{n+2}(s)$.  Therefore
least-witness assertions for $\Pi^1_{n+1}$ matrices are again
$\Sigma^1_{n+2}(s)$.
\end{proof}

\begin{lemma}\label{Mn-base-uniformization}
In the final $M_n$-extension, every boldface $\Sigma^1_{n+2}$ relation on
reals has a boldface $\Sigma^1_{n+2}$ uniformization.
\end{lemma}

\begin{proof}
Let $N$ be the final $M_n$-extension.  By construction, $N$ is one of the
small generic extensions covered by Proposition~\ref{prop:mn-relativized-capture}.
Work in $N$, and let $A\subseteq\mathbb R^2$ be a boldface
$\Sigma^1_{n+2}$ relation.  Choose a real parameter $a$ and a
$\Pi^1_{n+1}$ formula $\psi$ such that
\[
   A(x,y)
   \quad\Longleftrightarrow\quad
   \exists z\,\psi(a,x,y,z).
\]
For each real $x$, put $s=a\oplus x$, and let
$\theta(s,y,z)$ be the $\Pi^1_{n+1}$ formula obtained from
$\psi(a,x,y,z)$ after decoding $s$ as $a\oplus x$.  Define
$U(x,y)$ to hold iff
\[
   \operatorname{Least}^{n}_{\theta}(a\oplus x,y).
\]
By Proposition~\ref{prop:mn-relativized-capture}(3), $U$ is
$\Sigma^1_{n+2}(a)$, hence boldface $\Sigma^1_{n+2}$.

If $U(x,y)$ holds, then for some $z\in M_n(a\oplus x)$,
\[
   M_n(a\oplus x)\models\psi(a,x,y,z).
\]
By Proposition~\ref{prop:mn-relativized-capture}(2),
$N\models\psi(a,x,y,z)$, and therefore $A(x,y)$ holds.  Thus
$U\subseteq A$.

The relation $U$ is single-valued because $<_{n,a\oplus x}$ is a wellorder
and $U(x,y)$ asserts that $y$ is the least real in $M_n(a\oplus x)$ for which
a suitable $z$ exists.

Finally suppose that the $x$-section of $A$ is nonempty in $N$.  Then
\[
   N\models\exists y\exists z\,\psi(a,x,y,z).
\]
By Proposition~\ref{prop:mn-relativized-capture}(1), there are witnesses
$y,z\in M_n(a\oplus x)$.  Hence the set of $<_{n,a\oplus x}$-candidates is
nonempty, and it has a least element $y_0$.  By definition, $U(x,y_0)$ holds.
Thus $U$ uniformizes $A$.
\end{proof}

We now define the predicates used for the uniformization part of the final
iteration. Let ``$z$ is coded into $\vec S$'' mean $\sigma_n(z)$. Since this
is a $\Sigma^1_{n+3}$ predicate, we define, for each $m<\omega$, the
corresponding predicate $\Phi^{n+3+m}$ by adding the usual alternating block
of real quantifiers in front of the coding predicate. For example,
\[
\Phi^{n+3}(x,y,k)
\quad\Longleftrightarrow\quad
\exists a_0\,
\bigl((x,y,k,a_0)\text{ is coded into }\vec S\bigr),
\]
and in general

\begin{align*}
 \Phi^{n+3+m}(x,y,k) \Longleftrightarrow  
\exists a_0\,\forall a_1\,\cdots\,Q_m a_m\,
 & ((x,y,k,a_0,\ldots,a_m) \text{ is coded} \\& \text{ into } \vec{S} ) 
\end{align*}
where the final quantifier $Q$ depends on the arity of $m$ alternate.
The complexity of $\Phi^{n+3+m}$ is therefore $\Sigma^1_{n+3+m}$.

The final $\omega_2$-length iteration is now defined exactly as before. At
successor stages we again distinguish the three cases: the stages devoted to
forcing Martin's Axiom, the stages devoted to the definable wellorder of the
reals, and the stages devoted to the uniformization predicates
$\Phi^{n+3+m}$. Whenever the earlier construction used
$\operatorname{Code}(x,\xi)$, we now use $\operatorname{Code}_n(x,\xi)$.
The bookkeeping again ensures that every relevant real, every relevant
forcing for $\MA$, and every relevant projective relation to be uniformized
is considered cofinally often, and that unused $\omega$-blocks of
$\vec S$ are always available.

\begin{proposition}\label{Mn-final-extension}
Let $M_n$ exist. There is a generic extension of $M_n$ in which:
\begin{enumerate}
    \item $\MA$ holds;
    \item there is a $\Delta^1_{n+3}$-definable wellorder of the reals;
    \item $\Sigma^1_{n+2+m}$-uniformization holds for every $m<\omega$.
\end{enumerate}
\end{proposition}

\begin{proof}
The proof is the proof of the main iteration theorem over $M_1$, with
$\operatorname{Code}$ replaced by $\operatorname{Code}_n$ and
$\sigma$ replaced by $\sigma_n$.

Martin's Axiom is obtained from the same bookkeeping of ccc posets. The
definable wellorder is obtained by coding each real into a fresh block of
Suslin trees and then defining the order of reals by the order in which their
codes appear. Since the predicate ``$x$ is coded into $\vec S$'' is now
$\Sigma^1_{n+3}$, the resulting wellorder is $\Delta^1_{n+3}$.

For uniformization, the predicates $\Phi^{n+3+m}$ play exactly the role
played by the predicates $\Phi^k$ in the construction over $M_1$. At the
appropriate stages, if a section of a projective relation is nonempty, the
bookkeeping chooses the least witness with respect to the definable wellorder
and codes the corresponding tuple into a fresh block. Lemma
\ref{Mn-Code-properties} ensures that these codes are genuine and unique in
the required sense. The verification that the resulting relation is a
uniformizing subrelation has the same quantifier-by-quantifier proof as in
the earlier section, with the base point shifted from $\Sigma^1_4$ to
$\Sigma^1_{n+3}$. This gives uniformization for every level
$\Sigma^1_{n+3+m}$, $m<\omega$, equivalently for all
$\Sigma^1_{n+2+m}$ with $m\geq 1$. The missing base case $m=0$ is exactly
Lemma~\ref{Mn-base-uniformization}.
\end{proof}

It remains only to record the regularity conclusion.  We spell out the
higher-level version of the Hjorth argument, since this is the point at which
the tree $T_{n+1}$ replaces the classical Martin--Solovay tree $T_2$.

\begin{proposition}\label{Mn-regularity}
In the final extension, every $\boldsymbol{\Sigma}^1_{n+2}$ set of reals is
Lebesgue measurable and has the Baire property.
\end{proposition}

\begin{proof}
Let $V^*$ denote the final extension, and let
$A\subseteq\mathbb R$ be $\boldsymbol{\Sigma}^1_{n+2}$.  Fix a real parameter
$a$ and a $\Sigma^1_{n+2}(a)$ formula $\varphi(v,a)$ defining $A$ in $V^*$.
Let $T_{n+1}\in M_n$ be the canonical Martin--Solovay--Steel tree projecting
to the universal $\boldsymbol{\Sigma}^1_{n+2}$ set.  By
Lemma~\ref{lem:tn-small-generic-absoluteness}, the $T_{n+1}$ representation
of $\Sigma^1_{n+2}$ truth is absolute to the small generic extensions used in
this construction, and to the further random or Cohen extensions used below.

Choose an intermediate stage after the parameter $a$ has appeared.  In the
corresponding local model the structure $L[T_{n+1},a]$ has at most
$\aleph_1$ many reals and hence codes at most $\aleph_1$ many Borel null sets
and at most $\aleph_1$ many Borel meager sets.  Since $V^*$ satisfies
$\MA$, the union of the null sets coded in $L[T_{n+1},a]$ is null, and the
union of the meager sets coded in $L[T_{n+1},a]$ is meager.  Thus almost every
real is random over $L[T_{n+1},a]$, and comeagerly many reals are Cohen
generic over $L[T_{n+1},a]$.

We first prove Lebesgue measurability.  In $L[T_{n+1},a]$ take the Boolean
value of the assertion $\varphi(\dot r,a)$ in the random algebra, where
$\dot r$ is the canonical random real.  This Boolean value is represented by a
Borel set $B$.  If $r$ is random over $L[T_{n+1},a]$, the forcing theorem in
$L[T_{n+1},a]$ gives
\[
   r\in B
   \quad\Longleftrightarrow\quad
   L[T_{n+1},a,r]\models \varphi(r,a),
\]
where the right-hand side is read through the $T_{n+1}$ representation of the
universal $\Sigma^1_{n+2}$ set.  By the absolute-complement property of
$T_{n+1}$ from Lemma~\ref{lem:tn-small-generic-absoluteness}, this is
equivalent to
\[
   V^*\models \varphi(r,a),
\]
for every such random real $r$.  Therefore $A\triangle B$ is contained in the
set of reals which are not random over $L[T_{n+1},a]$, a null set.  Hence
$A$ is Lebesgue measurable.

The Baire-property argument is identical, replacing the random algebra by
Cohen forcing.  The Boolean value of $\varphi(\dot c,a)$ in the Cohen algebra
is represented by a Borel set $C$, and for every Cohen real $c$ over
$L[T_{n+1},a]$ the same $T_{n+1}$ absoluteness gives
\[
   c\in C
   \quad\Longleftrightarrow\quad
   V^*\models \varphi(c,a).
\]
The exceptional set of reals which are not Cohen generic over
$L[T_{n+1},a]$ is meager, so $A\triangle C$ is meager.  Thus $A$ has the
Baire property.
\end{proof}

Combining Propositions \ref{Mn-final-extension} and \ref{Mn-regularity} gives
the desired theorem.

\begin{theorem}
Assume that $M_n$ exists. Then there is a generic extension of $M_n$ in which
$\MA$ holds, every $\boldsymbol{\Sigma}^1_{n+2}$ set is Lebesgue measurable
and has the Baire property, there is a $\Delta^1_{n+3}$-definable wellorder
of the reals, and $\Sigma^1_{n+2+m}$-uniformization holds for every
$m<\omega$.
\end{theorem}

\section{Open Questions}

We end with a couple of natural questions whose answers would need new ideas. A first and natural question is whether the large cardinals used in our argument can be improved. 
\begin{question}
    What is the large cardinal strength of ``every $\bf{\Sigma}^1_{n+2}$-set is Lebesgue measurable, has the Baire property, there is a $\Delta_{n+3}^1$-definable wellorder of the reals and $\Sigma^1_{n+2+m}$-uniformization holds?
\end{question}
Note that by R. Solovay (\cite{solovay1970model})and S. Shelah (\cite{shelah1984can}) we only need an inaccessible to force all boldface projective sets become Lebesgue measurable and have the Baire property. The other two properties have no large cardinal strength at all. The question of whether already an inaccessible suffices to get a universe with the combination of the four properties was already asked in \cite{friedman2003universally}.

The following strengthening of the original 12th Delfino problem is still wide open (see \cite{caicedo2020fourteen}):
\begin{question}
    Assume that every projective set is Lebesgue measurable, has the Baire property and $\Pi^1_{2n+1}$-uniformization holds for every $n \in \omega$. Must $\PD$ hold then as well?
\end{question}

There are other local versions of the 12th Delfino problem which are mostly open as well:
\begin{question}
    Suppose that  $\Pi^1_{2n+1}$ sets for $n \le N$ have the uniformization property and that all $\bf{\Sigma}^1_{2N+1}$ sets are Lebesgue measurable and have the Baire property. Is $\bf{\Pi}^1_{2N}$-determinacy true? 
\end{question}
For $n=N=1$ the above has a positive answer due to J. Steel (see \cite{schindler_delfino}) but for the higher levels this is open.

\begin{question}
Is there a universe with global $\Sigma$-uniformization and where every projective set is Lebesgue measurable and has the Baire property?
\end{question}

\begin{question}
Is there a universe with global $\Sigma$-uniformization and where every projective set is Lebesgue measurable? 
\end{question}

\begin{question}
Given an inaccessible, is there a universe where each projective set is Lebesgue measurable and projective separation fails? The same question can be asked with the Baire property replacing the notion of Lebesgue measurability.
\end{question}

\begin{question}
For which further tree forcings $P$ from the Friedman--Schrittesser diagram
can one obtain, over $M_1$, a model preserving one Woodin cardinal and
satisfying
\[
\Sigma^1_3(P)+\neg\Delta^1_4(P)?
\]
In particular, can this be done for Laver, Silver or Mathias measurability?
\end{question}

\bibliographystyle{plain}
\bibliography{references}

\end{document}